\documentclass[a4paper,12pt]{amsart} 
\usepackage[utf8]{inputenc}
\usepackage[T1]{fontenc}
\usepackage[english]{babel}
\usepackage{mathtools,amssymb,amsthm,amsmath}
\usepackage{enumitem}
\usepackage{esint}
\usepackage{lmodern}
\usepackage{amsbsy}
\usepackage{mathrsfs}
\usepackage{ stmaryrd }
\usepackage{hyperref}
\usepackage[normalem]{ulem}
\usepackage{geometry} 
\usepackage{ulem}

\usepackage{color}

\newcommand{\BBB}{\color{black}}

\newcommand{\EEE}{\color{black}}
\newcommand \Hn{\mathbb{H}^n}
\newcommand \F{\mathcal{F}}
\newcommand \A{\mathcal{A}}

\newcommand \N{\mathbb{N}}

\newcommand \Om{\Omega}
\newcommand \lb{\lambda}
\newcommand \Sol{\mathcal{S}}
\newcommand \om{\omega}
\newcommand \eps{\epsilon}
\newcommand \loc{\text{loc}}

\newcommand \ov{\overline}
\newcommand \un{\underline}
\newcommand \R{\mathbb{R}}
\newcommand \Rn{\mathbb{R}^n}
\newcommand\B{\mathbb{B}}
\newcommand\C{\mathcal{C}}
\newcommand\Id{\mathrm{Id}}

\newcommand{\cvg}{\;\xrightarrow{\hspace*{1cm}}\;}

\newtheorem{theorem}{Theorem}[section] 
\newtheorem{corollary}[theorem]{Corollary} 
\newtheorem{lemma} [theorem]{Lemma} 
\newtheorem{remark} [theorem]{Remark} 
\newtheorem{proposition} [theorem]{Proposition} 

\newtheorem{definition}[theorem] {Definition}

\title{Sharp Quantitative Stability of the Dirichlet spectrum near the ball}

\author[D. Bucur]
{Dorin Bucur}
\address[Dorin Bucur]{Univ. Savoie Mont Blanc, CNRS, LAMA \\
73000 Chamb\'ery, France
}
\email{  dorin.bucur@univ-savoie.fr}

\author[J. Lamboley]{Jimmy Lamboley}
\address[Jimmy Lamboley]{\'Ecole Normale Supérieure, CNRS, PSL University, DMA and Sorbonne Université, Universit\'e Paris Cit\'e, IMJ-PRG, F-75005 Paris, France
}
\email{  jimmy.lamboley@ens.psl.eu}

\author[M. Nahon]{Micka\"{e}l Nahon}
\address[Micka\"{e}l Nahon]{ Univ. Grenoble Alpes, CNRS, Grenoble INP, LJK, 38000 Grenoble, France.}
\email{mickael.nahon@univ-grenoble-alpes.fr}

\author[R. Prunier]{Raphaël Prunier}
\address[Raphaël Prunier]{CEREMADE, UMR CNRS 7534, Universit\'e Paris-Dauphine, Universit\'e PSL, Place du Mar\'echal De Lattre
De Tassigny, 75775 Paris cedex 16, France.
}
\email{prunier@ceremade.dauphine.fr}

\begin{document}

\begin{abstract}
Let $\Om\subset\Rn$ be an open set with the same volume as the unit ball $B$ and let $\lambda_k(\Om)$ be the $k$-th eigenvalue of the Laplace operator of $\Om$ with Dirichlet boundary conditions on $\partial\Om$. In this work, we answer the following question:
\begin{center}
\textit{If $\lambda_1(\Om)-\lambda_1(B)$ is small, how large can $|\lambda_k(\Om)-\lambda_k(B)|$ be ?}
\end{center}
We establish quantitative bounds of the form $|\lambda_k(\Om)-\lambda_k(B)|\le C (\lambda_1(\Om)-\lambda_1(B))^\alpha$ with sharp exponents $\alpha$ depending on the multiplicity of $\lambda_k(B)$.
We first show that such an inequality is valid with $\alpha=1/2$ for any $k$, improving previous known results and providing the sharpest possible exponent. Then, through the study of a vectorial free boundary problem, we show that one can achieve the better exponent $\alpha=1$ if $\lambda_{k}(B)$ is simple. We also obtain  a similar result for the whole cluster of eigenvalues when $\lambda_{k}(B)$ is multiple, thus providing  a complete answer to the question above.
As a consequence of these results,  we obtain the persistence of the ball as the minimizer for a large class of spectral functionals which are small perturbations of the fundamental eigenvalue on the one hand,  and a full reverse Kohler-Jobin inequality on the other hand, solving an open problem formulated by M. Van Den Berg, G. Buttazzo and A. Pratelli. 
\end{abstract}


\maketitle

\section{Introduction}

\subsection{Presentation of the problem}
Let  $\Rn$ be  Euclidean space, for some $n\geq 2$ and let $\om_n$ denote the measure of the unit ball in $\Rn$. We set
\[\A=\left\{\Om\subset\Rn\text{ open  set of measure }\om_n\right\},\]
and $B(=B_1)$ the unit ball of $\Rn$ centered at the origin. For $\Om$ an open set of finite volume, we write
\[\lambda_k(\Om):=\inf\left\{\sup_{v\in V}\frac{\int_{\Om}|\nabla v|^2}{\int_\Om v^2},\ V\subset H^1_0(\Om)\text{ of dimension }k\right\}\]
the $k$-th eigenvalue of the Laplacian on $\Om$ with Dirichlet boundary conditions on $\partial\Om$ (counting multiplicities). The associated eigenfunctions, normalized in $L^2(\Om)$, are denoted  $(u_k)_{k\geq 1}$ and verify
\[u_k\in H^1_0(\Om),\ -\Delta u_k=\lambda_k(\Om) u_k  \text{ in }\Om.\]

For every $\Om \in \A$,  the Faber-Krahn inequality implies $\lambda_1(\Om)\geq \lambda_1(B)$, with equality if and only if $\Om$ coincides with a ball (up to a set of zero capacity). Several recent works point out that $\Om$ must,  in some sense,  be close to $B$ when $\lambda_1(\Om)$ is close to $\lambda_1(B)$. We refer to \cite{BDV15,AKN21} and the references therein for the most recent results and a history of the quantitative Faber-Krahn inequality. Roughly speaking, the variation of the first eigenvalue $\lambda_1(\Om)-\lambda_1(B)$ controls both (the square of) the Fraenkel asymmetry of $\Om$ and the $L^2$ norm of the variation of the eigenfunction. 
\smallskip

The main purpose of this paper is to get a sharp control by $\lambda_1(\Om)-\lambda_1(B)$ of the variation of the full spectrum. Precisely, {\BBB given $k$ and $l$ such that $\lb_{k-1}(B)< \lb_k(B) =\dots= \lb_l(B) <\lb_{l+1}(B)$,} we seek inequalities  of the form
\begin{equation}\label{ex1.03}
\left|\sum_{i=k}^l\big( \lambda_i(\Om)-\lambda_i(B)\big) \right|\le C_{n,k} \lb_1(\Om) ^{1-\alpha} (\lb_1(\Om)-\lb_1(B))^\alpha,
\end{equation}
for which the power $\alpha$ is sharp (note that this formulation with $\lambda_{1}(\Om)^{1-\alpha}$ is just a way to avoid assuming that $\lambda_{1}(\Om)$ is bounded from above as in the previous references). 

\smallskip

\noindent {\bf Heuristics about sharp power $\alpha$}. In some particular cases, inequality \eqref{ex1.03} has already been studied in the literature. In a first paper \cite{BC06}, in 2006,    Bertrand and Colbois established
$$|\lambda_k(\Om)-\lambda_k(B)|\leq C_{n,k} \left(\lambda_1(\Om)-\lambda_1(B)\right)^\frac{1}{80n},$$
valid for $\lambda_1(\Om)$  bounded from above. Later, in 2019, relying on the quantitative Faber-Krahn inequality from \cite{BDV15}, Mazzoleni and Pratelli  improved the exponents into  (see \cite{MP19})
\begin{equation}\label{eq:MP19}
-c_{n,k}\left(\lambda_1(\Om)-\lambda_1(B)\right)^{\frac{1}{6}-\eps}\leq\lambda_k(\Om)-\lambda_k(B)\leq C_{n,k} \left(\lambda_1(\Om)-\lambda_1(B)\right)^{\frac{1}{12}-\eps},
\end{equation}
for any $\eps>0$ (and with better exponents in dimension $n=2$) when $\lambda_1(\Om)$ is bounded from above, but the authors naturally expected these exponents not to be optimal. 

Indeed, the following observation is in order. When looking at domains which are volume-preserving smooth perturbations of the ball, one may see $B$ as a non-degenerate stable critical point of $\lambda_1$ under volume constraint. On the other hand, for $k\geq 2$, the condition for $B$ to be a critical point of $\lambda_k$ is that the associated eigenfunction $u_k$ has constant gradient on the boundary. This is  \BBB  the case for eigenvalues associated  \EEE  to radial eigenfunctions, which precisely correspond to the simple eigenvalues. In conclusion, when $\lambda_k(B)$ is simple one may expect 
a {\it sharp} bound of the type
\begin{equation}\label{eq:sharpsimple}
|\lambda_k(\Om)-\lambda_k(B)|\leq C_{n,k} (\lambda_1(\Om)-\lambda_1(B)).
\end{equation}
{\BBB However, }when $\lambda_k(B)$ is degenerate (multiple), then $\lambda_k$  has only directional derivatives at $B$ which, in general, are non-zero. Consequently, we cannot expect a better bound than
\begin{equation}\label{eq:main12}
|\lambda_k(\Om)-\lambda_k(B)|\leq C_{n,k} \lambda_1(\Om)^\frac 12 (\lambda_1(\Om)-\lambda_1(B))^\frac 12.
\end{equation}
Nevertheless, as observed in \cite{MP19}, while $\lambda_{2}(B)$ is multiple, one can still get {\BBB the one-sided estimate}
\begin{equation}\label{eq:Ash-Ben}
\lambda_{2}(\Om)-\lambda_{2}(B) \leq C(\lambda_{1}(\Om)-\lambda_{1}(B))
\end{equation}
as a consequence of Ashbaugh-Benguria's inequality which asserts that the ball maximizes the ratio $\lambda_{2}/\lambda_{1}$.
{\BBB The proper generalization of this observation is the following (see also Remark \ref{rk:oneside})}: 
for a whole cluster associated to  a multiple eigenvalue 
$$\lambda_{k-1}(B)<\lambda_k(B)=\dots =\lambda_l(B)<\lambda_{l+1}(B),$$ while each individual $\lambda_i$ is not differentiable at $B$ (for $k\leq i\leq l$) any smooth symmetric function of $(\lambda_k,\hdots,\lambda_l)$ is differentiable and has a critical point at $B$. Consequently, one can indeed still hope for a result better than \eqref{eq:main12}, namely a linear bound on the sum
\begin{equation}\label{eq:sharpmultiple}\left|\sum_{i=k}^{l}\left[\lambda_i(\Om)-\lambda_i(B)\right]\right|\leq C_{n,k}( \lambda_1(\Om)-\lambda_1(B))
\end{equation}
generalizing the estimate for simple eigenvalues.

The goal of this paper is to show that these inequalities \eqref{eq:sharpsimple}, \eqref{eq:main12}, \eqref{eq:sharpmultiple}  indeed hold, and that the above observations turn out to provide the sharp exponents in \eqref{ex1.03}.

\smallskip

\noindent {\bf Strategy.}  As a first result (see Theorem \ref{main_sqrt} below), we will show that one can obtain \eqref{eq:main12} where $\alpha=\frac 12$ (valid for simple and for multiple eigenvalues) with suitable choices of test functions and the quantitative Saint-Venant inequality. This improves the previous results from \cite{BC06,MP19}. \BBB Let us briefly recall that the Saint-Venant inequality states that for every $\Om\in\A$ it holds
$$\int_\Om w_\Om \le \int_B w_B,$$
with equality if and only if $\Om$ and $B$ differ by a set of zero capacity, and where 
$w_\Om$ denotes the torsion function, defined as the (weak) solution of 
\begin{equation}\label{eq:torsion}\begin{cases}
-\Delta w_\Om=1&\text{ in }\Om,\\w_\Om=0&\text{ on }\partial\Om.
\end{cases}\end{equation}  
\EEE

To obtain a sharper result with exponent $\alpha=1$ (the case of a simple eigenvalue or of a whole cluster of a multiple eigenvalue) our proof appeals to the analysis of a new type of vectorial free boundary problem which falls outside from the situations already studied in the literature. 

Indeed, let us consider $\lambda_{k-1}(B)<\lambda_k(B)=\dots =\lambda_l(B)<\lambda_{l+1}(B)$:
inequality \eqref{ex1.03} with $\alpha=1$ becomes
\[\left|\sum_{i=k}^{l}\left[\lambda_i(\Om)-\lambda_i(B)\right]\right|\leq C_{n,k}( \lambda_1(\Om)-\lambda_1(B)).\]
Its proof is equivalent to the fact that, for some $\eps>0$ small enough, the ball is the unique solution of both shape optimization problems below (i.e. for both signs $+ $ and $- $)
\begin{equation}\label{ex1.04}
\min\left\{ \lb_1(\Om)-\eps \sum_{i=k}^l \lb_i(\Om) : \Om \in \A\right\}, \qquad\textrm{ \BBB and }\qquad \min\left\{ \lb_1(\Om)+\eps \sum_{i=k}^l \lb_i(\Om) : \Om \in \A\right\}.
\end{equation}

This assertion is proved through the following strategy, based on regularity theory. We prove first the existence of an optimal domain and  the Lipschitz regularity of the associated {\BBB torsion function and} eigenfunctions {\BBB (although we will prove higher regularity of these functions inside $\Om$, this regularity is already optimal when the functions are seen as extended by 0 in $\R^n\setminus\Om$)}. In a second step, we prove the regularity of the boundary  and that, in some strong  $\C^{3,\gamma}$ sense, the optimal domain is close to the ball. Finally, we use a second order shape derivative argument to conclude that the optimal domain is the ball, provided $\eps$ is small enough.  Those steps have been followed for example by 
Kn\"upfer, Muratov for the study of Gamow's model, which is a perturbation of the classical isoperimetric problem, see \cite{KM1,KM2}. Similar ideas can be found in the work by Cicalese and Leonardi  \cite{CL12}, where the authors prove the quantitative isoperimetric inequality. On the other hand, we are dealing here with a perturbation of $\lambda_{1}$ instead of the perimeter functional, so we use the theory of regularity of free boundary problems, as was done in \cite{BL09} and in the proof of the quantitative Faber-Krahn inequality by Brasco, De Phillipis and Velichkov in \cite{BDV15}.

Although 
the strategy to solve the shape optimization problems \eqref{ex1.04} follows the same main lines as \cite{BDV15}, the nature of our problem raises a series of new technical difficulties, mostly in the case of the negative sign. First, in this case the shape functional is not decreasing for inclusions, so that the existence of a solution is not guaranteed from the general result of Buttazzo-Dal Maso \cite{BDM93}. Second, the optimality condition reads formally
$$\left(\frac{\partial u_1}{\partial \nu_\Om}\right )^2-\eps \sum_{i=k}^l \left(\frac{\partial u_i}{\partial \nu_\Om}\right)^2 = \mbox{constant on } \; \partial \Om,$$
where $\nu_\Om$ is the outward normal vector at the boundary $\partial \Om$. 
The presence of the negative sign falls out from all the situations studied in the literature \cite{KL18,MTV17,CSY18,MTV17} including the degenerate case from  \cite{KL19}. The regularity analysis of this situation requires  most of the technicalities. We will use some key ideas from \cite{MTV21} for the analysis of our problem:  more precisely, when $k=l$ (case of simple eigenvalues), we will be able to apply some results of \cite{MTV21} (see Section \ref{sect:min_are_NS}), but when $k<l$ (case of multiple eigenvalues), we will have to prove the same results in more general situations (see Sections \ref{ssect:harnack} and \ref{ssect:flatness}).

As a global picture, our   analysis will require the analysis of a generalization of the vectorial Alt-Caffarelli problems in the wider setting
\[\begin{cases}
-\Delta v_i=f_i&\text{ in }\Om,\ \forall i=1,\hdots,m,\\
v_i=0 &\text{ in }\partial\Om,\ \forall i=1,\hdots,m,\\
q\left(\frac{\partial v_1}{\partial \nu_\Om},\hdots,\frac{\partial v_m}{\partial \nu_\Om}\right)=1&\text{ in }\partial\Om,
\end{cases}\]
where $\Om$ is the common domain of $(v_i)_{i=1,\hdots,m}$, $f_i\in L^\infty_\loc(\Om)$ and  $q$ is a quadratic form on $\R^m$. The outward normal derivatives  $\frac{\partial v_i}{\partial \nu_\Om}$ at the boundary are  understood in some weak sense - variational or viscosity - and the states $(v_i)$ are assumed to be ``flat'' (in a sense defined in Sections \ref{sec_lin} and \ref{sec_linvect} and   Corollary \ref{prop_flatness}):
\begin{itemize}[label=\textbullet]
\item The case $m=1$, $q(x)=x^2$ corresponds to  the classical Alt-Caffarelli problem of \cite{AC81}.
\item The case $m\geq 2$, $q(x_1,\hdots,x_m)=\sum_{i=1}^{m}c_i x_i^2$ is the one treated in \cite{KL18,KL19} with uniform estimates in $(c_i)$ as long as $c_i\geq 0$, $\sum_{i=1}^{m}c_i=1$. Similar results (obtained through different methods) may also be found in \cite{CSY18,MTV17} in the case $c_i=1$.
\item The case $m=2$, $q(x_1,x_2)=x_1x_2$ under the additional hypothesis that $u_1,u_2$ are positive is treated in \cite{MTV21}.
\end{itemize}
Our problem may be seen as  
\begin{itemize}[label=\textbullet]
\item 
$m\geq 2$ with $q(x_1,\hdots,x_m)=x_1^2+b(x_2,\hdots,x_m)$,
\end{itemize}
where $b$ is a quadratic form on $\R^{m-1}$ with no positivity assumption, with the additional hypothesis that   the function $v_1$ ``dominates'' all the others, meaning $\displaystyle \left|\frac {v_i}{v_1}\right|$ is not too large for every $i\geq 2$ (for   precise statements we refer to Definition \ref{def_Sol}). This hypothesis holds for free in some situations, for instance  if $v_1$ is the torsion function $w_\Om$ defined below and $v_i$ is a small multiple of the eigenfunction $u_i$ of $\Om$, so that as we shall see this will be true in the cases we are interested in.

\smallskip

\noindent {\bf Applications in spectral geometry and a reverse Kohler-Jobin inequality.}   It has been observed numerically in \cite[Fig 5.4]{KO13} that the set minimizing $\lb_k(\Om)$ in $\A$ is also minimizing $\lb_k(\Om)+\eps \lb_{k-1}(\Om)$, provided $\eps >0$ is small (the computations were performed for $3\le k \le 6$). 

This phenomenon of persistence of minimizers for perturbed functionals has also been conjectured in \cite{BBP21} for a functional involving the first Dirichlet eigenvalue and the torsional rigidity  which are interacting in a competing way. Recall that the torsional ridigidy is defined by 

$$T(\Om):= \int_\Om w_\Om=2\int_{\Om}w_{\Om}-\int_{\Om}|\nabla w_{\Om}|^2=\max\left\{2\int_{\Om}v-\int_{\Om}|\nabla v|^2, v\in H^1_{0}(\Om)\right\},$$where $w_\Om$ is the torsion function{\BBB, \textit{i.e.} the weak solution of \eqref{eq:torsion}}. While the Saint-Venant inequality states that the set with {\it maximal} torsional rigidity in $\A$ is the ball,
the conjecture from \cite{BBP21} reads
\begin{equation}\label{ex1.02} \exists p_n >0, \forall \Om \in \A, \quad \quad 
T(\Om)\lambda_1(\Om)^{\frac{1}{p_n}}\le T(B)\lambda_1(B)^{\frac{1}{p_n}}.
\end{equation}
As this inequality becomes Saint-Venant inequality when $p_{n}\to\infty$, the challenge is to prove that the ball $B$ remains a maximizer of $T(\Om)\lambda_{1}(\Om)^{1/p}$ for some finite values of $p$.

If $p_n=\frac{2}{n+2}$, the inequality above occurs in the opposite sense and is due to Kohler-Jobin  \cite{K78}. This is why, for $p_n$ large, inequality \eqref{ex1.02} can be seen as a reverse Kohler-Jobin inequality. In \cite{BBG22} it has been proved to hold locally for some $p_n$ large, in the class of  $\C^{2,\gamma}$  nearly spherical domains.

The main consequence of our analysis is the occurence of the persistance phenomenon of the ball as the minimizer for spectral functionals which are either  small perturbations of the first Dirichlet eigenvalue (for instance as in \eqref{ex1.04}) or of the \BBB (inverse  of the) \EEE torsional rigidity. In  particular we will prove the validity of the full reverse Kohler-Jobin inequality \eqref{ex1.02}, see Corollary \ref{cor_reversekohlerjobin}.

\subsection{Main results}

Inequalities \eqref{ex1.03}  for sharp exponents $\alpha$  will actually be proved in a stronger version, with the torsional deviation $T(\Om)^{-1}-T(B)^{-1}$ on the right-hand side in place of the eigenfrequency deviation $\lambda_1(\Om)-\lambda_1(B)$.

Indeed, as noted just above, Kohler-Jobin's inequality \cite{K78} states that $\A\ni \Om \mapsto T(\Om)\lambda_1(\Om)^{\frac{n+2}{2}}$ is minimal on the ball; this implies the {\BBB following bound,}  for all $\Om\in\A$:
\begin{equation}
\label{eq_kohlerjobin}
\begin{split}
T(\Om)^{-1}-T(B)^{-1}&\leq \lambda_1(B)^{-\frac{n+2}{2}}T(B)^{-1}\left(\lambda_1(\Om)^\frac{n+2}{2}-\lambda_1(B)^\frac{n+2}{2}\right)\\
&\leq \frac{n+2}{2}\lambda_1(B)^{-\frac{n+2}{2}}T(B)^{-1}\lambda_1(\Om)^\frac{n}{2}\left(\lambda_1(\Om)-\lambda_1(B)\right).
\end{split}
\end{equation}
Relying on this inequality when $\lambda_1(\Om)\leq 2\lambda_1(B)$,  and on growth estimates of the type $\lambda_k(\Om)\leq C_{n,k} \lambda_1(\Om)$ (see Proposition \ref{prop:growth_egv}) when $\lambda_1(\Om)\geq 2\lambda_1(B)$,  it is {\BBB therefore} enough to prove \eqref{ex1.03}   for the torsional deviation in the right hand side instead.

One of the reasons why we replace the first eigenvalue with the torsion energy  is  of a technical nature. In our problem, which involves simultaneously several eigenfunctions, we have a clear advantage to do this, since some uniform regularity estimates on those eigenfunctions may be directly deduced from the same estimates on the torsion function (see for instance Lemma \ref{lem_prelim}).  \BBB As a second advantage,  replacing $(\lb_1(\Om)-\lb_1(B))$ by $(T(\Omega)^{-1}-T(B)^{-1})$ in the right hand side of inequality \eqref{ex1.03} and setting $k=1$ in the left hand side, we obtain a nontrivial conjectured inequality,  reverse of  \eqref{eq_kohlerjobin}.\EEE

For the sake of clarity, we split inequality \eqref{ex1.03}  with sharp exponents $\alpha$  in three results.  The first one applies to every eigenvalue, and is sharp when $\lambda_k(B)$ is degenerate (see Proposition \ref{prop:sharp}). 

\BBB Theorems \ref{main_sqrt}, \ref{main_lin} and \ref{main_linvect} are stated in a scale-invariant way among every open sets $\Om \subset\R^n$ of finite measure. In practice, we will restrict to sets $\Om\in\A$, meaning open sets with measure $|B|$.
\begin{theorem}\label{main_sqrt}
There exists $C_n>0$ such that for any \BBB open set $\Om\subset\R^n$ with finite measure,

\[\left|\lambda_k(\Om)-\lambda_k(B_\Om)\right|\leq C_nk^{2+\frac{4}{n}}\lambda_1(\Om)^{\frac{1}{2}}|\Om|^\frac 12\left(T(\Om)^{-1}-T(B_\Om)^{-1}\right)^\frac{1}{2}, \]
where $B_\Om$ is a ball in $\R^n$ with the same measure as $\Om$.\EEE

\end{theorem} 

\EEE
Thanks to \eqref{eq_kohlerjobin}, this result improves the previous best known result \eqref{eq:MP19} from Mazzoleni-Pratelli \cite{MP19} into
\[\left|\lambda_k(\Om)-\lambda_k(B)\right|\leq C_nk^{2+\frac{4}{n}}\lambda_1(\Om)^{\frac{1}{2}}\left(\lambda_1(\Om)-\lambda_1(B)\right)^\frac{1}{2}. \]
{\BBB Note that this inequality} is only relevant when $\lambda_1(\Om)$ is close to  $\lambda_1(B)$; when $\lambda_1(\Om)\geq 2\lambda_1(B)$ we may apply more directly the inequality $\lambda_k(\Om)\leq C_nk^\frac{2}{n}\lambda_1(\Om)$ from \cite[Theorem 3.1] {CY07} to obtain $2C_nk^\frac{2}{n}\lambda_1(\Om)$ on the right-hand side.

\BBB We also point out that in the constant appearing in the right-hand side of the inequality we keep track of the dependence on $k$. This will be exploited in Section \ref{ssect:stability} in order to study the stability of more general functionals. Even if we cannot prove it, we do not expect the exponent $2 +\frac 4n$ to be optimal. We know nevertheless that the optimal power cannot be lower than $\frac 1n$, thanks to the Weyl asymptotic formula.\EEE

{\BBB As a useful tool for following the dependency of the constants on $k$} we thus introduce for any $k\geq 1$ the spectral gap

\begin{equation}\label{eq:gnk}
g_n(k)=\min\Big\{1,\inf_{i:\lambda_i(B)\neq \lambda_k(B)}\left|\lambda_i(B)-\lambda_k(B)\right|\Big\}.
\end{equation}

It is a positive bounded function of $k$. 
We state first  the case of a simple eigenvalue of the ball which gives a sharper estimate than the one from Theorem \ref{main_sqrt}.

\begin{theorem}\label{main_lin}
There exists $C_{n}>0$ such that  for every $k\in\N^*$ with   $\lambda_k(B)$  simple   and for any open set \BBB $\Om\subset\R^n$ with finite measure,
\[\left|\lambda_k(\Om)-\lambda_k(B_\Om)\right|\leq  C_n \frac{k^{4+\frac{8}{n}}}{g_n(k)} |\Om| \left(T(\Om)^{-1}-T(B_\Om)^{-1}\right),\]
where $B_\Om$ is a ball in $\R^n$ with the same measure as $\Om$.\EEE
\end{theorem}
The constant $C_n$ is not explicitly known since there are two   implicit arguments in the proof (the flatness improvement used in Proposition \ref{lem_C1gamma} which is obtained by contradiction, and the application of the quantitative Saint-Venant inequality of \cite{BDV15}).  Combining this result with the Kohler-Jobin inequality \eqref{eq_kohlerjobin} we thus obtain, for any such $k$,
\[\left|\lambda_k(\Om)-\lambda_k(B)\right|\leq  C_n \frac{k^{4+\frac{8}{n}}}{g_n(k)}  \left(\lambda_1(\Om)-\lambda_1(B)\right).\]

In dimension $2$ the valid choices of $k$ are
\[k=1,6, 15, 30, 51, 74, 105, 140, 175, 222, 269, 326, 383, 446, 517, 588,...\]
Let us mention again here that the crucial argument making the previous result work is that the ball is a critical point of $\lambda_k$ when $\lambda_k(B)$ is simple.

Consider now $k\leq l$ such that
\[\lambda_{k-1}(B)<\lambda_k(B)=\dots =\lambda_l(B)<\lambda_{l+1}(B).\]
The function $\Om\mapsto \sum_{i=k}^{l}\lambda_i(\Om)$ has a critical point at the ball (see for instance \cite[Proposition 2.30]{LamLan06}) and a result analogous to Theorem \ref{main_lin} holds:

\begin{theorem}\label{main_linvect}
There exists $C_{n}>0$ such that for every $k,l\in\N^*$ with $k\leq l$ satisfying
\[\lambda_{k-1}(B)<\lambda_k(B)=\dots=\lambda_{l}(B)<\lambda_{l+1}(B),\]
and for any open set \BBB $\Om\subset\R^n$ with finite measure,
\[\left|\sum_{i=k}^{l}\bigg[\lambda_i(\Om) - \lambda_i(B_\Om)\bigg]\right|\leq C_{n} \frac{k^{6+\frac{10}{n}}}{g_n(k)}|\Om| \left(T(\Om)^{-1}-T(B_\Om)^{-1}\right),\]
where $B_\Om$ is a ball in $\R^n$ with the same measure as $\Om$.\EEE
\end{theorem}

Using again the Kohler-Jobin inequality \eqref{eq_kohlerjobin} this implies \eqref{ex1.03} for any such $k\leq l$ with $\alpha=1$, {\BBB and} $C_{n,k}=C_nk^{6+\frac{10}{n}}g_n(k)^{-1}$.

\BBB
\begin{remark}\label{rk:oneside}
A consequence of Theorem  \ref{main_linvect} is the following one-sided linear control: for any $\Om\in\A$ such that $\lambda_1(\Om)\leq 2\lambda_1(B)$, if $k\geq 2$ is such that $\lb_k(B)$ is multiple, we have for some $C_{n,k}>0$:
\begin{align*}
\text{ if }\lambda_k(B)<\lambda_{k+1}(B),\text{ then }&\lambda_k(\Om)-\lambda_k(B)\geq -C_{n,k}(\lambda_1(\Om)-\lambda_1(B)),\\
\text{ if }\lambda_{k-1}(B) <\lambda_k(B),\text{ then }&\lambda_k(\Om)-\lambda_k(B)\leq C_{n,k}(\lambda_1(\Om)-\lambda_1(B)).
\end{align*}
The second one generalizes inequality \eqref{eq:Ash-Ben} which was observed for $k=2$.
\end{remark}
\EEE

As a consequence of Theorems \ref{main_sqrt} and \ref{main_linvect},  we can state a general result on the stability of the  Saint-Venant (and Faber-Krahn) inequality through perturbation by a spectral functional which has enough symmetries:

\begin{theorem}\label{main_perturb}
Let $k\in\N^*$ be such that $\lambda_{k}(B)<\lambda_{k+1}(B)$. Let $F\in\C^2(\R_+^{*k},\R)$ \BBB satisfy \EEE 
\BBB
\begin{enumerate}
\item [$\bullet$] $|F(\lambda)|\leq C(1+|\lambda|),  \text{ for some }C>0$,
\item  [$\bullet$] $\forall i,j\in\{1,\hdots,k\}, \text{ with } \lambda_i(B)=\lambda_j(B),\ \frac{\partial F}{\partial\lambda_i}=\frac{\partial F}{\partial\lambda_j}\text{ at }(\lambda_1(B),\hdots,\lambda_k(B)).$
\end{enumerate}
\EEE
Then there exists $\delta_F>0$ such that the functional
\begin{equation}\label{ex1.06}
\Om\in\A\mapsto T(\Om)^{-1}+\delta F(\lambda_1(\Om),\hdots,\lambda_k(\Om))
\end{equation}
is minimal on the ball for any $\delta\in\R$ such that $|\delta|<\delta_F$.
\end{theorem}
In particular, the full reverse Kohler-Jobin inequality holds:

\begin{corollary}\label{cor_reversekohlerjobin}
There exists $p_n>1$ such that $\A\ni \Om \mapsto T(\Om)\lambda_1(\Om)^{\frac{1}{p_n}}$ is maximal on the ball.
\end{corollary} 

{\BBB Notice finally that under suitable assumptions, we are also able to deal with functionals involving the whole spectrum $(\lambda_k(\Om))_{k\geq 1}$, see Proposition \ref{ex1.10}.}

\subsection{Outline of the paper}

In Section \ref{sect:prelim}, we {\BBB recall} some classical estimates of Dirichlet eigenvalues and eigenfunctions. In Section \ref{sec_sqrt}, we prove Theorem \ref{main_sqrt} as well as several useful lemmas on eigenfunctions and the torsion function. This is established by combining estimates from \cite{B03} on eigenvalues of nested domains, some estimates with explicit test functions and the quantitative Faber-Krahn inequality of \cite{BDV15}.

The next two sections are devoted to the proof of Theorems \ref{main_lin} and \ref{main_linvect}; while the second result is strictly stronger than the first, for expository reasons we shall first give a full proof of Theorem \ref{main_lin} in Section \ref{sec_lin} and then adapt this proof to the vectorial case in Section \ref{sec_linvect}, pointing out the differences.

Precisely, in  Section \ref{sec_lin} we start by  restating Theorem \ref{main_lin} as a shape optimization problem in the spirit of \eqref{ex1.04}. In a first step, we prove the existence of a relaxed minimizer among capacitary measures and, in a second step {\BBB we show} that this measure corresponds to  an open set which   is a smooth perturbation of the ball in an increasingly stronger sense. The key passage from an open set to a $\C^{1,\gamma}$ set is obtained by relating our problem  to a vectorial  Alt-Caffarelli problem as in \cite{KL18} or as in the more recent result \cite{MTV21}, depending on the sign of the perturbation. We finally conclude through second order shape derivative arguments for small perturbations of the ball.

Section \ref{sec_linvect} is a summary of the steps of the previous section for the vectorial problem, with in addition a careful examination of the dependency of the constants in terms of the multiplicity of the eigenspace,  obtained by following the proof of \cite{KL18} on the one hand and through a full proof of a vectorial version of \cite{MTV21} on the other hand. 

The last section is devoted to the discussion of the consequences, namely the proof of Theorem \ref{main_perturb} and of the  reverse Kohler-Jobin inequality, Corollary \ref{cor_reversekohlerjobin}.

\section{Some preliminary estimates}\label{sect:prelim}

Here we summarize some results on eigenvalues and eigenfunctions which we will use throughout the paper.  Although these results are not original, for the readability of the paper we give short proofs when possible or, at least, we comment \BBB on \EEE the proofs.

\smallskip 
\noindent {\bf Eigenvalues of the ball.}
For any $d\in\N$ we define $\mathbb{H}_{n,d}$ the space of harmonic homogeneous polynomials of degree $d$ in $n$ variables $x_1,\hdots,x_n$. For any $\alpha>0$ we denote by $J_\alpha$ the $\alpha$-th Bessel function
\[J_\alpha(x)=\sum_{p\geq 0}\frac{(-1)^{p}}{p!\Gamma(p+\alpha+1)}\left(\frac{x}{2}\right)^{2p+\alpha},\]
where $\Gamma$ is the standard Gamma function and we call $j_{\alpha,p}$ the $p$-th positive zero of $J_\alpha$, which is well-defined for every $p\in\N^*$. Then for every eigenvalue $\lambda_k(B)$, there exists a unique $(d,p)\in\N\times\N^*$, such that 
\[\lambda_k(B)=j_{d+\frac{n-2}{2},p}^2\]
and, conversely, for every $(d,p)\in\N\times\N^*$, $j_{d+\frac{n-2}{2},p}^2$ is an eigenvalue of $B$ associated to the eigenspace
\[\left\{x\mapsto\frac{ J_{d+\frac{n-2}{2}}\left(j_{d+\frac{n-2}{2},p}|x|\right)}{|x|^{d+\frac{n-2}{2}}}P\left(x\right),\ P\in\mathbb{H}_{n,d}\right\}\]
which has dimension
\[\text{dim}(\mathbb{H}_{n,d})=\begin{cases}1&\text{ if }d=0,\\
2&\text{ if }d>0,n=2,\\
(2d+n-2)\frac{(d+n-3)!}{d!(n-2)!}&\text{ if }d\geq 0,n\geq 3.\end{cases}\]
In particular an eigenvalue $\lambda_k(B)$ is either simple with a radial eigenfunction, or degenerate with only non-radial eigenfunctions. This particular fact (and more generally the fact that any eigenvalue corresponds to a unique couple $(d,p)$) is a result due to Siegel \cite{S29}. In the literature, it is also called ``Bourget's hypothesis'' since it has been mentioned in \cite{B1866}, with an incomplete proof.\newline

\smallskip
\noindent{\bf Eigenvalues and eigenfunctions estimates on general domains.} We start {\BBB by} recalling the following inequalities.

\begin{proposition}\label{prop:growth_egv}
Let $\Om\in\A$, $k\in\N^*$, then
\begin{equation}\label{est_eigen}
\left(\frac{n}{n+2}\right)\frac{{4\pi^2}}{\om_n^{4/n}}k^\frac{2}{n}\leq\lambda_k(\Om)\leq \left(1+\frac{4}{n}\right)\lambda_1(\Om)k^\frac{2}{n},
\end{equation}
\[\lambda_1(\Om)T(\Om)\leq \om_n.\]
\end{proposition}
The lower bound in the first inequality is due to Li and Yau in \cite[Corollary 1]{LY83}, while the upper bound was obtained by Chen and Yang in \cite[Theorem 3.1]{CY07}. On the other hand inequality $\lambda_1(\Om)T(\Om)\leq \om_n$ follows directly from using the torsion function as a competitor in the Rayleigh quotient defining $\lambda_1$.

\begin{lemma}\label{lem_prelim}
Let $\Om\in\A$, $k\geq 1$, and let $w$ be the torsion function of $\Om$ and $u_k$ some $L^2$-normalized eigenfunction. Then

\[w\leq \frac{1}{2n},\qquad |u_k|\leq e^\frac{1}{8\pi}\lambda_k(\Om)^\frac{n}{4},\qquad  |u_k|\leq e^{\frac{1}{8\pi}}\lambda_k(\Om)^{1+\frac n 4}w\qquad \text{ in } \Om,\]
\[\sup_\Om|\nabla u_k|^2\leq \left(\frac{1}{n}+\sup_{\Om}|\nabla w|^2\right)e^{\frac{1}{4\pi}}\lambda_k(\Om)^{2+\frac{n}{2}}.\]
\end{lemma}
\begin{proof}
{\BBB Talenti's inequality (see \cite[Theorem 1 (iv)]{Tal76}) implies that} the supremum of the torsion function is maximal on the unit ball, on which the torsion function has the explicit expression $w(x)=\frac{1-|x|^2}{2n}${\BBB, hence the first estimate}. Then classical heat kernel estimates (see for instance \cite[Ex. 2.1.8]{D89}) give $|u_k|\leq e^{\frac{1}{8\pi}}\lambda_k(\Om)^\frac{n}{4}$ so
\[\Delta \left(\pm u_k- e^{\frac{1}{8\pi}}\lambda_k(\Om)^{1+\frac{n}{4}}w\right)=-(\pm)\lambda_k(\Om)u_k +e^{\frac{1}{8\pi}}\lambda_k(\Om)^{1+\frac{n}{4}}\geq 0,\]
therefore $|u_k|\leq e^{\frac{1}{8\pi}}\lambda_k(\Om)^{1+\frac{n}{4}}w$ by \BBB the \EEE maximum principle. 

For the gradient bound, we suppose that $\nabla w$ is bounded{\BBB, the estimate being trivial if $\sup_\Om|\nabla w|=+\infty$}. By direct computation we have that $\Delta(|\nabla a|^2)\geq 2\nabla a \cdot\nabla(\Delta a)$ for a smooth function $a:\R^n\rightarrow\R$. \BBB Using the bounds on $w$ and $u_k$ we have \EEE
\[\Delta(|\nabla u_k|^2+\BBB\lambda_k(\Om)\EEE u_k^2)\geq -2\lambda_k(\Om)^2u_k^2\geq -2e^{\frac{1}{4\pi}}\lambda_k(\Om)^{2+\frac{n}{2}}\text{ in }\Om,\]
thus giving
\[\Delta\left(|\nabla u_k|^2+\lambda_k(\Om)u_k^2-2e^{\frac{1}{4\pi}}\lambda_k(\Om)^{2+\frac{n}{2}}w\right)\geq 0\text{ in }\Om.\]
Suppose first that $\Om$ is a $\C^\infty$ domain, then $\nabla u_k$ and $\nabla w$ extend continuously to the boundary and the inequality $|u_k|\leq e^{\frac{1}{8\pi}}\lambda_k(\Om)^{1+\frac{n}{4}}w$ (\BBB together with $u_k=w=0$ on $\partial\Om$\EEE) ensures
\[\forall x\in\partial\Om,\ |\nabla u_k(x)|\leq e^\frac{1}{8\pi}\lambda_k(\Om)^{1+\frac{n}{4}}|\nabla w(x)|,\]
and so by maximum principle:
\begin{align*}
\sup_\Om|\nabla u_k|^2&\leq 2e^{\frac{1}{4\pi}}\lambda_k(\Om)^{2+\frac{n}{2}}w+\sup_{\partial\Om}|\nabla u_k|^2\\
&\leq \frac{e^{\frac{1}{4\pi}}}{n}\lambda_k(\Om)^{2+\frac{n}{2}}+e^{\frac{1}{4\pi}}\lambda_k(\Om)^{2+\frac{n}{2}}\sup_{\Om}|\nabla w|^2.
\end{align*}
In the general case, by Sard's Theorem and since $w$ is smooth inside $\Om$ we may find arbitrarily small regular values $\eps>0$ such that $\{w=\eps\}=\partial \{w>\eps\}$ is a smooth hypersurface. Denote $\Om^\eps=\{w>\eps\}$ and $w^\eps,(u_k^\eps)_{k\in\N^*}$ the associated torsion function and eigenfunctions. Note that $w^\eps=(w-\eps)_+$, {\BBB hence $\Vert w_\Om-w_{\Om_\eps}\Vert_{L^2(\Rn)}\leq \om_n \eps \to 0$}, so that $\Om^\eps$ $\gamma$-converges to $\Om$ (see  for example \cite{BB05} for the definition and properties of $\gamma$-convergence). In particular, for all $k\geq 1$, $\lambda_k(\Om^\eps)\to \lambda_k(\Om)$ thanks to \cite[Corollaries 3 and 4, pp. 1089-1090]{DS88}. Now, \BBB since \EEE $u_k^\eps$ is bounded in $H^1_0(\Om)$ we can assume (up to extraction) that $u_k^\eps$ converges strongly in $L^2(\Om)$ and weakly in $H^1_0(\Om)$ to some limit $\BBB u_k^0\EEE$. Passing to the limit in the sense of distributions in $-\Delta u_k^\eps=\lambda_k(\Om^\eps)u_k^\eps$ we obtain that $(u_k^0)_{k\in\N}$ is an orthonormal \BBB eigenbasis \EEE for $\Om$.

Now, since $w^\eps=(w-\eps)_+$ we have 
\[\sup_{\Om^\eps}|\nabla u_k^\eps|^2\leq \frac{e^{\frac{1}{4\pi}}}{n}\lambda_k(\Om^\eps)^{2+\frac{n}{2}}+e^{\frac{1}{4\pi}}\lambda_k(\Om^\eps)^{2+\frac{n}{2}}\sup_{\Om}|\nabla w|^2\]
Using the $L^\infty$ bound \BBB on $u_k^\eps\left(\om_n^{-1/n}|\Om^\eps|^{1/n}\cdot\right)$ (so that its support has measure $\om_n$), we get $|u_k^\eps|\leq (1+o_{\eps\to 0}(1)) e^{\frac{1}{8\pi}}\lambda_k(\Om^\eps)^\frac{n}{4}\leq 2e^{\frac{1}{8\pi}}\lambda_k(\Om)^\frac{n}{4}$\EEE for small $\eps$. Thus $u_k^\eps$ is bounded in $W^{1,\infty}$ as $\eps\to 0$, so that it converges (up to subsequence) locally uniformly to \BBB $u_k^0$.\EEE The uniform gradient bound on $\nabla u_k^\eps$ transfers to $\nabla u_k$, thus concluding the proof. 
\end{proof}

As the next result shows, one can control the difference of eigenvalues by the difference of torsions for two nested domains $\om\subset\Om$.
\begin{lemma}\label{lem_bucur}
Let $\om\subset\Om$ be two open sets of finite measure, then
\[\frac{1}{\lambda_k(\Om)}-\frac{1}{\lambda_k(\om)}\leq e^{\frac{1}{4\pi}}k\lambda_k(\Om)^\frac{n}{2}\left[T(\Om)-T(\om)\right]\]
\end{lemma}
\begin{proof}
This result is proved in \cite[Theorem 3.4]{B03}, where one has to follow the proof to keep track of the constants (using for instance the $L^\infty$ bound $|u_k|\leq e^\frac{1}{8\pi}\lambda_k(\Om)^{\frac{n}{4}}$).
\end{proof}

\smallskip
\noindent{\bf The quantitative Faber-Krahn inequality.} 
The Fraenkel asymmetry $\F$, defined as
\[\F(\Om)=\inf_{x\in\R^n}|(B+x)\Delta\Om|,\]
plays a crucial role in the following quantitative Faber-Krahn inequality obtained in \cite{BDV15}.  \BBB Note that since the set $\Omega$ is of finite measure, the infimum is always attained, hence equality occurs for some $x \in \R^n$.\EEE
\begin{theorem}
There exists $c_n>0$ such that for any $\Om\in\A$,
\begin{equation}\label{eq:QSV}
T(\Om)^{-1}\geq T(B)^{-1}+c_n\F(\Om)^2,
\end{equation}
\begin{equation}\label{eq:QFK}\lambda_1(\Om)\geq \lambda_1(B)+c_n\F(\Om)^2.\end{equation}
\end{theorem}

\section{Proof of Theorem \ref{main_sqrt}: the square root bound}\label{sec_sqrt}
The proof of \ref{main_sqrt} is obtained as a consequence of the quantitative \BBB Saint-Venant \EEE inequality \eqref{eq:QFK}, growth estimates over $\lambda_k(\Om)$ from \eqref{est_eigen} and the next proposition, 
which, we believe, is of independent interest.
\begin{proposition}\label{prop_auxiliaire}
Let $\Om\in\A$, then
\[\left|\frac{1}{\lambda_k(\Om)}-\frac{1}{\lambda_k(B)}\right|\leq \left(1+\frac{4}{n}\right)^\frac{n}{2}e^{\frac{1}{4\pi}}k^2\lambda_1(\Om)^\frac{n}{2}\left[T(B)-T(\Om)+\left(\frac{1}{n}+\frac{1}{n^2}\right)|\Om\Delta B|\right].\]
\end{proposition}

{\BBB To show this result,} we first prove the following lemma.
\begin{lemma}\label{lem_torsioninter}
Let $\Om\in\A$, then
\[T(\Om)-T(\Om\cap B)\leq \left(\frac{1}{n}+\frac{1}{n^2}\right)|\Om\setminus B|\]
\end{lemma}
\begin{proof}
We write $w:=w_\Om$ and $v=w_B$. 
Then let $\tilde{w}=w\wedge v$: we have $\tilde{w}\in H^1_0(\Om\cap B)$ so
\[T(\Om\cap B)\geq \int_{ \Om\cap B}\left(2\tilde{w}-|\nabla\tilde{w}|^2\right)=\int_{ \Om\cap B}\left(2(w\wedge v)-|\nabla (w\wedge v)|^2\right),\]
{\BBB and}
\begin{align*}
T(\Om)-T(\Om\cap B)+\int_{\Om\setminus B}|\nabla w|^2&\leq \int_{\Om\setminus B}2w+\int_{ B\cap\Om}\left(2(w- \tilde{w})+ |\nabla (w\wedge v)|^2-|\nabla w|^2\right)\\
&\leq \int_{\Om\setminus B}2w+\int_{ B\cap\Om}\left(2(w- v)_+ +2\nabla (w\wedge v)\cdot\nabla (w\wedge  v - w)\right)\\
&= \int_{\Om\setminus B}2w+\int_{ B\cap\Om}\left(2(w- v)_+ -2\nabla (w\wedge v)\cdot\nabla (w- v)_+\right).
\end{align*}
Notice that $\nabla (w\wedge v)\cdot\nabla (w- v)_+=\nabla  v\cdot\nabla (w- v)_+=\nabla\cdot((w- v)_+\nabla v)+(w- v)_+$  in $\Om\cap B$ so by Stokes' formula,
\begin{align*}
T(\Om)-T(\Om\cap B)+\int_{\Om\setminus B}|\nabla w|^2&\leq \int_{\Om\setminus B}2w-2 v'(1)\int_{\partial B}w
\end{align*}
Notice that $-v'(1)=\frac{1}{n}$ and using the trace estimate $\int_{\partial B}w\leq \int_{\Rn\setminus B}|\nabla w|$ we get
\[ 2(-v'(1))\int_{\partial B}w\leq \frac{2}{n}\int_{\Om\setminus B}|\nabla w|\leq \frac{1}{n^2}|\Om\setminus B|+\int_{\Om\setminus B}|\nabla w|^2,\]
so
\[T(\Om)-T(\Om\cap B)\leq \int_{\Om\setminus B}2w+\frac{1}{n^2}|\Om\setminus B|\leq \left(\frac{1}{n}+\frac{1}{n^2}\right)|\Om\setminus B|.\]
\end{proof}
We may now prove Proposition \ref{prop_auxiliaire}.
\begin{proof}[Proof of Proposition \ref{prop_auxiliaire}]
\BBB Applying \EEE  the bound from Lemma \ref{lem_bucur} to $(\Om\cap B,B)$ and $(\Om\cap B,\Om)$, we have the two inequalities
\begin{align*}
\frac{1}{\lambda_k(\Om)}-\frac{1}{\lambda_k(\Om\cap B)}&\leq e^{\frac{1}{4\pi}}k\lambda_k(\Om)^\frac{n}{2}\left[T(\Om)-T(\Om\cap B)\right]\\
&\leq \left(1+\frac{4}{n}\right)^\frac{n}{2}e^{\frac{1}{4\pi}}k^{2}\lambda_1(\Om)^\frac{n}{2}\left[T(\Om)-T(\Om\cap B)\right],
\end{align*}
\begin{align*}
\frac{1}{\lambda_k(B)}-\frac{1}{\lambda_k(\Om\cap B)}&\leq e^{\frac{1}{4\pi}}k\lambda_k(B)^\frac{n}{2}\left[T(B)-T(\Om\cap B)\right]\\
&\leq \left(1+\frac{4}{n}\right)^\frac{n}{2}e^{\frac{1}{4\pi}}k^{2}\lambda_1(\Om)^\frac{n}{2}\left[T(B)-T(\Om\cap B)\right].
\end{align*}
So combining them, we get
\begin{align*}
\left|\frac{1}{\lambda_k(\Om)}-\frac{1}{\lambda_k(B)}\right|&\leq \left(1+\frac{4}{n}\right)^\frac{n}{2}e^{\frac{1}{4\pi}}k^{2}\lambda_1(\Om)^\frac{n}{2}\left[T(\Om)+T(B)-2T(\Om\cap B)\right]\\
&= \left(1+\frac{4}{n}\right)^\frac{n}{2}e^{\frac{1}{4\pi}}k^{2}\lambda_1(\Om)^\frac{n}{2}\left[(T(B)-T(\Om))+2(T(\Om)-T(\Om\cap B))\right].
\end{align*}
Using Lemma \ref{lem_torsioninter},  
\begin{align*}
\left|\frac{1}{\lambda_k(\Om)}-\frac{1}{\lambda_k(B)}\right|&\leq \left(1+\frac{4}{n}\right)^\frac{n}{2}e^{\frac{1}{4\pi}}k^{2}\lambda_1(\Om)^\frac{n}{2}\left[(T(B)-T(\Om))+2\left(\frac{1}{n}+\frac{1}{n^2}\right)|\Om\setminus B|\right]
\end{align*}
which is the result  as $|\Om\setminus B|=\frac12|\Om\Delta B|$.
\end{proof}

We can now prove Theorem \ref{main_sqrt}.
\begin{proof}[Proof of Theorem \ref{main_sqrt}]
We may take the result of Proposition \ref{prop_auxiliaire} and apply, up to a translation of $\Om$, {\BBB \eqref{eq:QSV} as well as \eqref{est_eigen}}, 
\begin{align*}
&\left|\lambda_k(\Om)-\lambda_k(B)\right|\\
&\hskip 1cm \leq\left(1+\frac{4}{n}\right)^\frac{n}{2}e^{\frac{1}{4\pi}}k^{2}\lambda_1(\Om)^\frac{n}{2}\lambda_k(B)\lambda_k(\Om)\left[(T(B)-T(\Om))+{2\left(\frac{1}{n}+\frac{1}{n^2}\right)}|\Om\setminus B|\right]\\
&\hskip 1cm  \leq C_nk^{2+\frac{4}{n}}\lambda_1(\Om)^{1+\frac{n}{2}}\left[(T(B)-T(\Om))+C_n\sqrt{T(\Om)^{-1}-T(B)^{-1}}\right].
\end{align*}
This gives the result when $T(\Om)\geq \frac{1}{2}T(B)$ (notice that in this case, thanks to $\lambda_1(\Om)\leq\omega_nT(\Om)^{-1}$ from Proposition \ref{prop:growth_egv}, we bound $\lambda_1(\Om)\leq C_n'$ for some $C_n'>0$). When $T(\Om)\leq \frac{1}{2}T(B)$ we may write more direcly
\begin{align*}
|\lambda_k(\Om)-\lambda_k(B)|&\leq \left(1+\frac{4}{n}\right)k^{\frac{2}{n}}\left(\lambda_1(\Om)+\lambda_1(B)\right)\leq  2\left(1+\frac{4}{n}\right)k^{\frac{2}{n}}\lambda_1(\Om)\\
&\leq 2\sqrt{2\om_n}\left(1+\frac{4}{n}\right)k^{\frac{2}{n}}\lambda_1(\Om)^\frac{1}{2}\left(T(\Om)^{-1}-T(B)^{-1}\right)^\frac{1}{2}
\end{align*}
where we used the estimates from Proposition \ref{prop:growth_egv}. In both cases  the result holds for some constant $C_n>0$.
\end{proof}

\section{Proof of Theorem \ref{main_lin}: the linear bound}\label{sec_lin}
Let us fix $k\ge 1$ as in Theorem \ref{main_lin}, {\BBB which is} such that $\lambda_k(B)$ is simple (we  also include $k=1$, as it will give non-trivial results). In order to prove Theorem \ref{main_lin}, {\BBB we will show an equivalent statement, namely} that when $\delta\in\R$ is  close to $0$ (depending on $n$ and $k$) the ball is a minimizer of the functional
\begin{equation}\label{eq_func}
\Om\in\A\mapsto T(\Om)^{-1}+\delta \lambda_k(\Om).
\end{equation}
{\BBB Precisely, the goal of this section is to prove the following result, which directly implies Theorem \ref{main_lin}:}
\begin{proposition}\label{main2_lin}
There exists $c_n>0$ such that for any $\delta\in\R$ with $|\delta|\leq c_n  k^{-\left(4+\frac{8}{n}\right)}g_n(k)$ the ball is the unique minimizer of \eqref{eq_func}.
\end{proposition}

\begin{remark}\rm  We remind the reader that $g_{n}(k)$ has been defined in \eqref{eq:gnk}. As far as we know, 
there is no explicit lower bound of $g_n(k)$; it was proved by Siegel (see \cite{S29} or \cite[15.28]{W44}, referred to as ``Bourget's hypothesis'') that zeroes of different Bessel functions are distinct, but with no quantified separation between successive zeroes. We conjecture that there exists an exponent $\kappa>1$ such that for any $m,p,q\in\N^*$, $\mu\in\N/2$, it holds that
\[|j_{\mu,p}-j_{\mu+m,q}|\geq j_{\mu,p}^{-\kappa}.\]
The validity of this conjecture would improve the quality of our bounds.
\end{remark}

In Proposition \ref{main2_lin}  we  have to consider the case when $\delta$ is positive {\it and } negative in order to get a proof of Theorem \ref{main_lin} 		and to obtain a bound of $(\lambda_k(\Om)-\lambda_k(B))$ on both sides. Then, by  \eqref{eq_kohlerjobin}, {\BBB which is a consequence of Kohler-Jobin's inequality}, Theorem \ref{main_lin} directly implies  inequality \eqref{ex1.03} as a consequence{\BBB, as explained in the introduction}. 

The plan of proof of Proposition \ref{main2_lin}  is the following.
\begin{itemize}
\item For  $\delta$ close enough to $0$ we prove the existence of a minimizer \BBB $\Omega$ for \eqref{eq_func} \EEE. The case $\delta <0$ raises extra difficulties: we obtain existence through careful concentration-compactness methods, first as a capacity measure and second as a quasi-open set{\BBB , see Proposition \ref{prop_existence} (we recall that a quasi-open set is by definition the level set $\{w>0\}$ of an $H^1$-function $w$; see for example \cite{HP18} for more details)}; {\BBB we then show that the torsion function $w$ of $\Om$}  verifies some uniform bounds, {\BBB namely the Lipschitz bound} $|\nabla w|\leq C_{n}$, and a non degeneracy condition: for all $x\in\Om$, $r\in (0,1)$
\[\fint_{\partial B_{x,r}}w\geq c_{n}r.\]
{\BBB In particular, the global continuity of $w$ provides the existence of an open set for \eqref{eq_func}. The estimates (see Lemma \ref{lem_reg}) are obtained by perturbing $\Om$ and controlling the variation of $\lambda_k$ by the variations of the torsion $T$. All of these results are detailed} in  Section \ref{sect:exmin}.
\item In Section \ref{sect:blowup}, we prove that if $\Om$ solves \eqref{eq_func}, then its  torsion function $w$ and \BBB its \EEE $L^2$-normalized $k$-th eigenfunction $u_k$ verify, in some sense that will be made precise, the equations
\begin{equation}\label{eq:optimality}\begin{cases}
-\Delta w=1,\quad -\Delta u_k=\lambda_k(\Om)u_k & \mbox{in } \Om\\[2mm]
|\nabla w|^2+T(\Om)^2\delta |\nabla u_k|^2=Q & \mbox{on } \partial\Om
\end{cases}
\end{equation}
where $Q{\BBB=Q(n,k,\delta)}>0$ is a constant which is arbitrarily close to $\frac{1}{n^2}$ when $\delta\to 0$. This part of the proof uses blow-up methods similar to \cite{KL18} and \cite{CS05}. 
\item We {\BBB then use \eqref{eq:optimality} to improve the regularity properties of $\Om$}, and show in Section \ref{sect:min_are_NS} that, as $\delta\to 0$, $\partial\Om$ is \BBB an arbitrarily small $\C^{3,\gamma}$ graph on $\partial B$ \EEE (up to translation). The case $\delta>0$ relies on the results from \cite{KL18} while the case $\delta<0$ is obtained by applying the results from \cite{MTV21}. 
\item Finally, in Section \ref{subsec_NS} we prove  a Fuglede-type result, namely that the ball is optimal for \eqref{eq_func} among small $\C^{3,\gamma}$ perturbations of the ball, through a second shape derivative estimate which follows the method of \cite{DL19}. {\BBB Combined with the previous step, this allows to conclude that the ball is the unique solution to \eqref{eq_func} for small $\delta$, which was our goal.}
\end{itemize}

Throughout the proof we extensively use the two following notations:
\begin{itemize}
\item $a\lesssim b$ when $a\leq C_n b$ for some (possibly) large $C_n>0$ which only depends on the dimension $n$.
\item $a\ll b$ when $a\leq c_n b$ for some $c_n>0$  that can be made as small as we want, and only depends on the dimension $n$.
\end{itemize}
In both cases the notation does not involve a dependence on the order of the eigenvalue $k$.
\subsection{Existence of a minimizer}\label{sect:exmin}
To prove existence we first prove some a priori estimates for sets  whose energy $T^{-1}+\delta\lambda_{k}$ is bounded from above by the one of the ball, which we may suppose to be verified without loss of generality for any element of a minimizing sequence. This is the object of the next Lemma.

\begin{lemma}\label{lem_apriori} Let $\Om\in\A$ be   such that 
\[T(\Om)^{-1}+\delta\lambda_k(\Om)\leq T(B)^{-1}+\delta\lambda_k(B),\]
and suppose $\Om$ is translated such that $\F(\Om)=|\Om\Delta B|$. Then if $|\delta|\ll k^{-\frac{2}{n}}$   the following inequalities hold
\BBB
\begin{itemize}
\item [$\bullet$] 
$|\Om\Delta B|\lesssim k^\frac{1}{n}|\delta|^\frac{1}{2}$, \; \; $T(\Om)^{-1}\lesssim 1$, \; \; $\lambda_k(\Om)\lesssim k^{\frac{2}{n}}$, 
\smallskip
\item [$\bullet$] $
T(\Om)^{-1}-T(B)^{-1}\lesssim k^\frac{2}{n}|\delta|\ \text{ and for all } i\in \mathbb{N}^*, |\lambda_i(\Om)-\lambda_i(B)|\lesssim i^{2+\frac{4}{n}}k^\frac{1}{n}|\delta|^\frac{1}{2}, $
\smallskip
\item [$\bullet$] $\Vert w_\Om-w_B\Vert_{L^1(\Rn)}\lesssim k^\frac{1}{n}|\delta|^\frac{1}{2}.$
\end{itemize}\EEE
\end{lemma}
\begin{proof}
Thanks to the upper bound from \eqref{est_eigen}, we have
\[T(\Om)^{-1}-T(B)^{-1}\leq\delta \left(\lambda_k(B)-\lambda_k(\Om)\right)\lesssim k^\frac{2}{n}\lambda_1(\Om)|\delta|\lesssim  k^\frac{2}{n}T(\Om)^{-1}|\delta|\]
so when $|\delta|\ll k^{-\frac{2}{n}}$ we get that $T(\Om)^{-1}\lesssim 1$, and using the same series of inequalities together with $T(\Om)\leq T(B)$ we deduce
\[T(B)-T(\Om)\lesssim k^\frac{2}{n}|\delta|.\]
Applying the quantitative Saint-Venant inequality \eqref{eq:QSV},  we get
\[|\Om\Delta B|\lesssim k^\frac{1}{n}|\delta|^\frac{1}{2}\text{ when }|\delta|\ll k^{-\frac{2}{n}},\]
from which we also deduce, using Theorem \ref{main_sqrt}, that for any $i\in\mathbb{N}^*$ it holds
\[|\lambda_i(\Om)-\lambda_i(B)|\lesssim i^{2+\frac{4}{n}}k^\frac{1}{n}|\delta|^\frac{1}{2}.\]
For the third item, we write
\begin{align*}
\Vert w_{\Om}-w_B\Vert_{L^1(\Rn)}&\leq \Vert w_{\Om}-w_{\Om\cap B}\Vert_{L^1(\Rn)}+\Vert w_{B}-w_{\Om\cap B}\Vert_{L^1(\Rn)}\\
&=T(B)-T(\Om)+2(T(\Om)-T(\Om\cap B))\\
&\leq T(B)-T(\Om)+\left(\frac{1}{n}+\frac{1}{n^2}\right)|\Om\Delta B|,
\end{align*}
where we used Lemma \ref{lem_torsioninter} and $|\Om\Delta B|=2|\Om\setminus B|$ in the last line. We obtain the last result by recalling that $|\Om\Delta B|\lesssim k^\frac{1}{n}|\delta|^\frac{1}{2}$.
\end{proof}
{\BBB Using the bound from Theorem \ref{main_sqrt} we can also } improve the decay of the quantities from the
previous lemma, in terms of $\delta$, \BBB to the price of having larger polynomial growth in $k$\EEE.
{\BBB We deduce from these estimates that if $\delta$ is small enough (precisely if $|\delta|\ll k^{-(4+\frac8n)}g_{n}(k)$ as in Proposition \ref{main2_lin})}, then any $\Om$ satisfying the hypothesis of Lemma \ref{lem_apriori} will be such that $\lambda_k(\Om)$ is simple.
\begin{corollary}\label{rem_estimate}
Let $\Om\in\A$ and $\delta$ be chosen as in the previous lemma \ref{lem_apriori}. Then
\begin{align*}
T(\Om)^{-1}-T(B)^{-1}&\lesssim k^{4+\frac{8}{n}}|\delta|^2,\\
\F(\Om)&\lesssim k^{2+\frac{4}{n}}|\delta|,\\
\forall i\in\N^*,\ |\lambda_i(\Om)-\lambda_i(B)|&\lesssim i^{2+\frac{4}{n}}k^{2+\frac{4}{n}}|\delta|.
\end{align*}
As a consequence, if $|\delta|\ll k^{-(4+\frac8n)}g_{n}(k)$ then $\lambda_k(\Om)$ is simple.
\end{corollary}

\begin{proof}
{\BBB By Theorem \ref{main_sqrt} we have}
\[T(\Om)^{-1}-T(B)^{-1}\leq\delta \left(\lambda_k(B)-\lambda_k(\Om)\right)\lesssim k^{2+\frac{4}{n}}|\delta|\left(T(\Om)^{-1}-T(B)^{-1}\right)^{\frac{1}{2}},\]
{\BBB which} gives the first estimate. The second estimate follows from the quantitative Saint-Venant inequality and the third estimate from applying Theorem \ref{main_sqrt} again.  We deduce that $\lambda_k(\Om)$ is simple by applying separately $|\lambda_i(\Om)-\lambda_i(B)|\lesssim i^{2+\frac{4}{n}}k^{2+\frac{4}{n}}|\delta|$ for $i=k-1,k,k+1$.
\end{proof}

\begin{proposition}\label{prop_existence}
If $|\delta|\ll k^{-\left(2+\frac{4}{n}\right)}$, then the functional \eqref{eq_func} has a minimizer in the class of quasi-open sets of measure $\om_n$.
\end{proposition}

To prove this proposition,  we will use the setting of capacitary measures (see for instance \cite{BDM93}). A capacitary measure is a nonnegative Borel measure  $\mu$, possibly infinite valued,  such that $\mu(E)=0$ as soon as $E$ has zero capacity. We typically assign to any quasi-open set $A$ the capacitary measure 
\[\infty_{\Rn\setminus A}(E):=\begin{cases}+\infty&\text{ if }\text{Cap}(E\setminus A)>0\\ 0&\text{ else}.\end{cases}\] 
Given a capacitary measure $\mu$, we define the regular set of $\mu$, denoted $A_\mu$ as the union of all finely open sets of finite $\mu$-measure. If $A_\mu$ has finite Lebesgue measure, 
we define the torsion $T(\mu)$ and the eigenvalues $\lambda_k(\mu)$ of a capacitary measure as follows ($\mathscr{L}^n$ denotes the $n$-dimensional Lebesgue measure):
\begin{align*}
T(\mu)&:=\sup_{u\in H^1(\Rn)\cap L^2(\mu)}\int_{\Rn}\left(2u-|\nabla u|^2\right)d\mathscr{L}^n-\int_{\Rn}u^2d\mu\\
&=\sup_{H^1(\Rn)\cap L^2(\mu)}\frac{\left(\int_{\Rn}ud\mathscr{L}^n\right)^2}{\int_{\Rn}|\nabla u|^2d\mathscr{L}^n+\int_{\Rn}u^2d\mu}= \int_{\Rn}w_\mu,\\
\lambda_k(\mu)&:=\inf\left\{\sup_{v\in V}\frac{\int_{\Rn}|\nabla v|^2d\mathscr{L}^n+\int_{\Rn}v^2d\mu}{\int_{\Rn}v^2d\mathscr{L}^n},\ V\subset H^1(\Rn)\cap L^2(\mu)\text{ of dimension }k\right\}\\
&=\int_{\Rn}|\nabla u_{k,\mu}|^2d\mathscr{L}^n+\int_{\Rn}u_{k,\mu}^2d\mu.
\end{align*}
Above, $w_\mu$ is the torsion function associated to $\mu$ and is a variational solution of
\begin{equation}\label{ex1.07}
\begin{cases}-\Delta w_\mu+\mu w_\mu=1 \text{ in } [H^1(\Rn)\cap L^2(\Rn,\mu)]',\\[2mm]
w_\mu\in H^1(\Rn)\cap L^2(\Rn,\mu),\end{cases}
\end{equation}
and $(u_{k,\mu})_{k\in\N^*}$ is a choice of an $L^2$-orthonormal basis of eigenfunctions associated to $\mu$, that verify
\[\begin{cases}-\Delta u_{k,\mu}+\mu u_{k,\mu}=\lambda_k(\mu)u_{k,\mu} \text{ in } [H^1(\Rn)\cap L^2(\Rn,\mu)]',\\[2mm]
u_{k,\mu}\in H^1(\Rn)\cap L^2(\Rn,\mu).\end{cases}\]
Note that $\mu\mapsto T(\mu)$ and $\mu\mapsto\lambda_k(\mu)$ are continuous for the $L^1(\Rn)$ distance between the associated torsion functions $w_\mu$ (which is called $\gamma$-distance, see \cite{BDM93}). Moreover $A_\mu=\{w_\mu >0\}$, up to a set of zero capacity.

This setting is in fact only necessary when $\delta<0$, since the case $\delta>0$ could be solved using only the lower semicontinuity of the functionals $T^{-1},\lambda_k$, however we choose a unified approach to both problems.

\begin{proof}[\BBB Proof of Proposition \ref{prop_existence}]
A first remark is that when $|\delta|$ is small enough the measure constraint $|\Om|=\om_n$ may be relaxed into $|\Om|\leq \om_n$, since any set $\Om$ which does not saturate the constraint $|\Om|\leq \om_n$ may be dilated into a set with lower energy. Indeed {\BBB if $|\Om|=:(1-t)\om_n$} for some $t{\BBB \in(0,1)}$, then $(1-t)^{-\frac{1}{n}}\Om$ is still admissible and using $\lambda_k(\Om)\lesssim k^{\frac{2}{n}}T(\Om)^{-1}$  (from Lemma \ref{prop:growth_egv}) we get
\begin{equation}\label{eq_dilation}
\begin{split}
T((1-t)^{-\frac{1}{n}}\Om)^{-1}+\delta\lambda_k((1-t)^{-\frac{1}{n}}\Om)&=
(1-t)^{\frac{n+2}{n}}T(\Om)^{-1}+(1-t)^{\frac{2}{n}}\delta\lambda_k(\Om)\\
&\leq  (1-t)T(\Om)^{-1}+\delta\lambda_k(\Om)+|\delta|t\lambda_k(\Om)\\
&\leq T(\Om)^{-1}+\delta\lambda_k(\Om)-t\left( T(\Om)^{-1}-C_nk^{\frac{2}{n}}|\delta|T(\Om)^{-1}\right)\\
&< T(\Om)^{-1}+\delta\lambda_k(\Om)\text{ when }|\delta|\ll k^{-\frac{2}{n}}.
\end{split}
\end{equation}

Let $(\Om_p)_{p\in\N}$ be a minimizing sequence of $T^{-1}+\delta\lambda_k$ in $\mathcal{A}$. By replacing $\Om_p$ by $B$ if needed we can assume without loss of generality that $T^{-1}(\Om_p)+\delta\lambda_k(\Om_p)\leq T^{-1}(B)+\delta\lambda_k(B)$ for {\BBB every} $p$, so that $\Om_p$ satisfies the hypothesis of Lemma \ref{lem_apriori}. Then by Lemma \ref{lem_apriori} we have a bound on the Fraenkel asymmetry $\F(\Om_p)\lesssim k^\frac{1}{n}|\delta|^\frac{1}{2}$ so, up to translation, we {\BBB may} suppose 
\[|\Om_p\Delta B|\lesssim k^\frac{1}{n}|\delta|^\frac{1}{2}.\]
Let us first prove that this sequence $\gamma$-converges to a capacitary measure $\mu$ (meaning that the associated torsion functions $w_{\Om_p}$ converges weakly in $H^1(\Rn)$ to the function $w_\mu$ given by \eqref{ex1.07})
and {\BBB that} $A_\mu =\{w_\mu>0\}$ verifies $|A_\mu|{\BBB \leq} \omega_n$.\newline

By concentration-compactness for sequences of open sets of bounded measure (see \cite[Th 2.2.]{B00}), in order to prove that convergence occurs we must exclude the dichotomy behaviour.

We thus assume by contradiction that we are in the latter situation, meaning that one can find $\tilde{\Om}_p=\BBB \Om^1_p\cup \Om^2_p\EEE \subset\Om_p$ with $\text{dist}(\Om_p^1,\Om_p^2)\to\infty$, $\liminf_{p\to\infty} |\Om_p^i|>0$ {\BBB for $i\in\{1,2\}$}, and verifying $\|w_{\Om_p}-w_{\tilde{\Om}_p}\|_{L^2(\R^n)}\rightarrow0$. As a consequence $|T(\tilde{\Om}_p)-T(\Om_p)|\rightarrow0$ and $|\lambda_k(\tilde{\Om}_p)-\lambda_k(\Om_p)|\rightarrow0$, and $\tilde{\Om}_p$ is therefore still a minimizing sequence. By Lemma \ref{lem_apriori} we have $T(\Om_p)^{-1}-T(B)^{-1}\lesssim k^\frac{2}{n}|\delta|$, hence we also have $T(\tilde{\Om}_p)^{-1}-T(B)^{-1}\lesssim k^\frac{2}{n}|\delta|$ and this ensures $|\Om_p\setminus \tilde{\Om}_p|\lesssim  k^\frac{2}{n}|\delta|$ using the Saint-Venant inequality. We therefore have $|\tilde{\Om}_p\Delta B|\leq |\Om_p\Delta B|+|\Om_p\setminus\tilde{\Om}_p|\lesssim  k^\frac{1}{n}|\delta|^\frac{1}{2}$ for $|\delta|\ll k^{-\frac{2}{n}}$. Furthermore, since $d(\Om^1_p,\Om^2_p)\rightarrow+\infty$ we have (say) that $|\Om^1_p\Delta B|\lesssim k^\frac{1}{n}|\delta|^\frac{1}{2}$ and $|\Om^2_p|\lesssim k^\frac{1}{n}|\delta|^\frac{1}{2}$.

We claim that $\lambda_k(\tilde{\Om}_p)=\lambda_k(\Om_p^1)$ when $|\delta|\ll k^{-2-\frac{2}{n}}$. Indeed on the one hand by Faber-Krahn inequality we have
\[\lambda_1(\Om_p^2)\gtrsim |\Om_p^2|^{-\frac{2}{n}}\gtrsim k^{-\frac{2}{n^2}}|\delta|^{-\frac{1}{n}},\]
and on the other hand thanks to Proposition \ref{prop:growth_egv}
\[ \lambda_k(\Om_p^1)\lesssim k^\frac{2}{n}T(\Om_p^1)^{-1}\lesssim k^\frac{2}{n}.\]
To justify the last inequality we note that by Saint-Venant it holds $T(\Om^2_p)\lesssim k^{\frac{n+2}{n^2}}|\delta|^{\frac{n+2}{2n}}$ and   thanks to the a priori estimates from Lemma \ref{lem_apriori} we have \[T(\Om_p^1)=T(\tilde{\Om}_p)-T(\Om_p^2)\geq T(B)-C_n\left(k^\frac{2}{n}|\delta|+ k^\frac{(n+2)}{n^2}|\delta|^\frac{n+2}{2n}\right)\geq \frac{1}{2}T(B)\text{ for  }|\delta|\ll k^{-\frac{2}{n}}.\] Hence $\lambda_1(\Om_p^2)\geq\lambda_k(\Om_p^1)$ {\BBB if} $|\delta|\ll k^{-2-\frac{2}{n}}$ and therefore $\lambda_k(\tilde{\Om}_p)=\lambda_k(\Om_p^1)$ {\BBB under the same condition on $\delta$}; as a consequence
\[T(\tilde{\Om}_p)^{-1}+\delta\lambda_k(\tilde{\Om}_p)=\left(T(\Om_p^1)+T(\Om_p^2)\right)^{-1}+\delta\lambda_k(\Om_p^1).\]

Set $t_p\in(0,1)$ such that $|\Om_p^2|=t_p\om_n$, so $t_p\lesssim k^\frac{1}{n}|\delta|^\frac{1}{2}$ and $\liminf_{p\to\infty}t_p>0$. We now argue that $(1-t_p)^{-\frac{1}{n}}\Om_p^1$ is a strictly better minimizing sequence. Indeed,  since $T(\Om_p^2)\lesssim t_p^{\frac{n+2}{n}}\lesssim k^{\frac{n+2}{n^2}}|\delta|^{\frac{n+2}{2n}}$ and $T(\Om_p^1)\gtrsim 1$, one has $\frac{1}{T(\Om_p^1)}\leq \frac{1}{T\left(\Om_p^1\right)+T(\Om_p^2)}+C_nt_p^{\frac{n+2}{n}}$ for some $C_n>0$ so that
\begin{align*}
\frac{1}{T\left((1-t_p)^{-\frac{1}{n}}\Om_p^1\right)}&+\delta \lambda_k\left((1-t_p)^{-\frac{1}{n}}\Om_p^1\right)=(1-t_p)^\frac{n+2}{n}\frac{1}{T\left(\Om_p^1\right)}+(1-t_p)^{\frac{2}{n}}\delta \lambda_k\left(\Om_p^1\right)\\
&\leq \frac{1}{T\left(\Om_p^1\right)}+\delta \lambda_k\left(\Om_p^1\right)-t_p\left(\frac{1}{T(\Om_p^1)}-|\delta|\lambda_k(\Om_p^1)\right)\\
&\leq \frac{1}{T\left(\Om_p^1\right)+T(\Om_p^2)}+\delta \lambda_k\left(\Om_p^1\right)-t_p\left(\frac{1}{T(\Om_p^1)}- C_nt_p^{\frac{2}{n}}-|\delta|\lambda_k(\Om_p^1)\right)\\
&\leq \frac{1}{T\left(\Om_p^1\right)+T(\Om_p^2)}+\delta \lambda_k\left(\Om_p^1\right)-t_p\left(\frac{{1}}{T(B)}- C_n k^{\frac{2}{n^2}}|\delta|^{\frac{1}{n}}-C_nk^{\frac{2}{n}}|\delta|\right).
\end{align*}
Since $\liminf_{p\to\infty}t_p>0$, this provides a strictly better minimizing sequence when $|\delta|\ll k^{-\frac{2}{n}}$, providing a contradiction and thus proving that dichotomy does not occur.

Thanks to \cite[Th 2.2.]{B00} we deduce that there exists some capacitary measure $\mu$ such that after extraction (and translation of the $\Om_p$) the sequence $(\Om_p)_p$ $\gamma$-converges to $\mu$.  In particular one has convergence of the torsional rigidity and eigenvalues. We let $w_\mu$ be the associated torsion function and $A_\mu=\{w_\mu>0\}$ the (quasi-open) associated domain. Notice that we can assume without loss of generality that the limiting measure verifies $\mu(E)=\infty$ whenever $\text{Cap}(E\setminus A_\mu)>0$, since this leaves $w_\mu$ unchanged.

{\BBB We first notice that }$|A_\mu|\leq \omega_n$ by a.e. pointwise convergence of the torsion functions. {\BBB To conclude, we now have to show} that $\mu$ corresponds to some quasi-open domain, precisely that $\mu=\infty_{\R^n\setminus A_\mu}$. The idea is to use \BBB the \EEE optimality of $\mu$ to prove that the torsions of $\mu$ and $A_\mu$ are equal. We denote by $w_{A_\mu}$ the torsion function of $A_\mu$. We have (see \cite{BDM93})
\[\begin{cases}-\Delta w_\mu+\mu w_\mu\le1 \text{ in } {\mathcal D}'(\Rn)\\
w_\mu\in H^1(\Rn)\cap L^2(\mu)\end{cases},
\text{ while } 
\begin{cases}-\Delta w_{A_\mu}=1 \text{ in } A_\mu,\\
w_{A_\mu}\in H^1_0(A_\mu).\end{cases}\]

By the maximum principle, it holds that $w_{A_\mu}\geq w_\mu$, implying in particular $T(A_\mu)\geq T(\mu)$.  Thanks to \cite[Lemma 3.1]{BBG22}, the sequence $\Om_p\cap A_\mu$ still $\gamma$-converges to $\mu$. As a consequence, by applying Lemma \ref{lem_bucur} to $A_\mu$ and $\Om_p\cap A_\mu$ and passing to the limit one gets
\[0\leq\frac{1}{\lambda_k(A_\mu)}-\frac{1}{\lambda_k(\mu)}\leq e^{\frac{1}{4\pi}}k\lambda_k(A_\mu)^\frac{n}{2}\left[T(A_\mu)-T(\mu)\right],\]
which we rewrite, using that $\lambda_k(A_\mu)\leq \lambda_k(\mu)\lesssim k^{\frac{2}{n}}$ (using monotonicity and Lemma \ref{lem_apriori}), as
\[0\leq\lambda_k(\mu)-\lambda_k(A_\mu)\lesssim k^{2+\frac{4}{n}}\left[T(\mu)^{-1}-T(A_\mu)^{-1}\right].\]
Then by minimality of $\mu$,
\begin{align*}
T(\mu)^{-1}+\delta\lambda_k(\mu)&\leq T(A_\mu)^{-1}+\delta\lambda_k(A_\mu)\\
&\leq T(\mu)^{-1}+\delta\lambda_k(\mu)+\left(1-C_n k^{2+\frac{4}{n}}|\delta|\right)\left(T(A_\mu)^{-1}-T(\mu)^{-1}\right).
\end{align*}
When $|\delta|\ll k^{-\left(2+\frac{4}{n}\right)}$ this gives $T(A_\mu)\leq T(\mu)$, hence $T(A_\mu)=T(\mu)$. Now, since $w_\mu\in H^1_0(A_\mu)$ we deduce
\[\int_{\Rn}\left(2w_\mu-|\nabla w_\mu|^2\right)d\mathscr{L}^n\leq T(A_\mu)= T(\mu)\leq \int_{\Rn}\left(2w_\mu-|\nabla w_\mu|^2\right)d\mathscr{L}^n-\int_{\Rn}w_\mu^2d\mu\]
thus implying $\int_{\Rn}w_\mu^2d\mu=0$. As a consequence $\mu=0$ in $A_\mu$, meaning $\mu=\infty_{\Rn\setminus A_\mu}$. Hence $A_\mu$ is a minimizer of the functional \eqref{eq_func} in the class of quasi-open sets of measure $\omega_n$, thus concluding the proof. 
\end{proof}
\BBB
\begin{remark}f\label{rem:scale_free}
In order to prove regularity estimates, a scale-invariant functional is easier to handle than a measure constraint: we recall \EEE that any minimizer of \eqref{eq_func} in the class of open (or quasi-open) sets of measure $\om_n$, is also a minimizer of the scale-invariant functional

\begin{equation}\label{eq:scale_free}
A\in\{\text{(quasi-)open sets}\}\mapsto |A|^\frac{2}{n}\left(\frac{|A|}{\om_nT(A)}+\delta\lambda_k(A)\right).
\end{equation}
\end{remark}
\EEE

{\BBB Now that we have shown the existence of a solution for \eqref{eq_func} in the class of quasi-open sets, we show some weak regularity properties of those solutions, implying in particular the existence of a solution in the class of open sets.}

\begin{lemma}\label{lem_reg}
Let $\Om$ be a   minimizer of \eqref{eq_func} in the class of quasi-open sets of measure $\om_n$, and suppose that $|\delta|\ll k^{-\left(2+\frac{4}{n}\right)}$. Then $\Om$ is bounded and there exists $c_{n},C_{n}>0$ such that 
{\BBB\[\|\nabla w_\Om\|_{L^\infty(\Rn)}\leq C_{n},\qquad \|\nabla u_k\|_{L^\infty(\Rn)}\leq C_nk^{\frac{2}{n}+\frac{1}{2}},\qquad \text{diam}(\Om)\leq C_n\]} and for all $x\in\Rn$, $r\in {\BBB(0,1]}$,
\begin{equation}\label{eq:nondeg}
\fint_{\partial B_{x,r}}w_\Om< c_{n}r\text{ implies } w_{\Om}|_{B_{x,r/2}}= 0.
\end{equation}
In particular, the open set $\{w_\Om>0\}${\BBB (equal to $\Om$ up to a set of 0-capacity)} is an open minimizer of \eqref{eq_func}.
\end{lemma}
Property \eqref{eq:nondeg} will be referred to as non-degeneracy, as it accounts in a weak sense for the fact that $|\nabla w_{\Om}|$ stays away from 0 near $\partial\Om$.
\begin{proof}

\textbf{Lipschitz regularity.} Let us first prove the Lipschitz regularity of $w_\Om$, which will imply the Lipschitz regularity of the eigenfunctions $u_k$ by the estimates of Lemma \ref{lem_prelim}.  To prove the Lipschitz regularity it is enough to prove the following property on the torsional rigidity: for any open set $\tilde{\Om}$ that contains $\Om$ {\BBB and} such that $|\tilde{\Om}\setminus\Om|$ is small enough, we have
\begin{equation}\label{eq_extminimality}
T(\Om)^{-1}\leq T(\tilde{\Om})^{-1}+\Lambda_n|\tilde{\Om}\setminus \Om|.
\end{equation}

Indeed, \BBB if this property is true, consider some ball $B_{x,r}$  with $r$ small enough  \EEE  such that this inequality applies to $\tilde{\Om}=\Om\cup B_{x,r}$; for any $\tilde{w}\in H^1_0(\Om\cup B_{x,r})$ coinciding with $w$ outside of $B_{x,r}$, writing $T(\Om\cup B_{x,r})\geq \int_{\R^n}2\tilde{w}-|\nabla \tilde{w}|^2$ we get by rearranging \eqref{eq_extminimality}:
\[\int_{B_{x,r}}\left(|\nabla w|^2-\frac{1}{2}w\right)\leq \int_{B_{x,r}}\left(|\nabla\tilde{w}|^2-\frac{1}{2}\tilde{w}\right)+\Lambda_n' r^n\]
for some $\Lambda_n'>0$. This corresponds to the notion of quasi-minimizer of \cite[Definition 3.1]{BMPV15} for $f=1$, so we may apply \cite[Theorem 3.3]{BMPV15} to get a uniform Lipschitz bound.

Let us therefore prove claim \eqref{eq_extminimality}.
Let $\tilde{\Om}$ be an open set that contains $\Om$ with $|\tilde{\Om}|\leq 2|\Om|$.
We separate the case $\delta>0$ and $\delta<0$ for clarity:
\begin{itemize}[label=\textbullet]
\item Case $\delta>0$. By monotonicity of $\lambda_k$ we have $\lambda_k(\Om)\geq \lambda_k(\tilde{\Om})$, so using minimality of $\Om$ against the competitor $\left(\frac{|\Om|}{|\tilde{\Om}|}\right)^{\frac{1}{n}}\tilde{\Om}$ we have
\[
T(\Om)^{-1}\leq \left(\frac{|\tilde{\Om}|}{|\Om|}\right)^{\frac{n+2}{n}}T(\tilde{\Om})^{-1}+\delta\left[\left(\frac{|\tilde{\Om}|}{|\Om|}\right)^{\frac{2}{n}}-1\right]\lambda_k(\Om)\]
which implies
\[T(\Om)^{-1}\leq T(\tilde{\Om})^{-1}+C_n\left(1+|\delta|k^\frac{2}{n}\right)|\tilde{\Om}\setminus\Om|\]
thanks to $T(\tilde{\Om})^{-1}\leq T(\Om)^{-1}\lesssim 1$ and $\lambda_k(\Om)\lesssim k^{\frac{2}{n}}$ from Lemma \ref{lem_apriori}. We thus get \eqref{eq_extminimality} as soon as $|\delta|k^\frac{2}{n}\lesssim 1$.
\item Case $\delta<0$. {\BBB In this case, we use the same competitor as in the positive case, but we have to use Lemma \ref{lem_bucur} instead of the monotonicity of $\lambda_k$}. Comparing the energy of $\Om$ to the energy of $\left(\frac{|\Om|}{|\tilde{\Om}|}\right)^{\frac{1}{n}}\tilde{\Om}$ gives
\begin{align*}
	T(\Om)^{-1}-T(\tilde{\Om})^{-1}&\leq \delta\left(\lambda_k(\tilde{\Om})-\lambda_k(\Om)\right)+\left(\left(\frac{\tilde{|\Om|}}{|\Om|}\right)^\frac{n+2}{n}-1\right)T(\tilde{\Om})^{-1}\\
	&\leq |\delta|e^\frac{1}{4\pi}k\lambda_k(\Om)\lambda_k(\tilde{\Om})^{1+\frac{n}{2}}\left[T(\tilde{\Om})-T(\Om)\right]+C_n|\tilde{\Om}\setminus\Om|\\
	&\leq C_n'|\delta|k^{2+\frac{4}{n}}\left[T(\Om)^{-1}-T(\tilde{\Om})^{-1}\right]+C_n|\tilde{\Om}\setminus\Om|
\end{align*}
for some $C_n,C_n'>0$. When $|\delta|\ll k^{-\left(2+\frac{4}{n}\right)}$ we get \eqref{eq_extminimality}.
\end{itemize}
\textbf{Non-degeneracy property.} The non-degeneracy is obtained by similar arguments, {\BBB now} choosing sets $\tilde{\Om}$ such that $\tilde{\Om}\subset\Om$. Let us prove that for any $\tilde{\Om}\subset\Om$ with $|\tilde{\Om}]\geq\frac{1}{2}|\Om|$ it holds
\begin{equation}\label{eq_intminimality}
T(\Om)^{-1}+\Lambda_n|\Om\setminus \tilde{\Om}|\leq T(\tilde{\Om})^{-1}.
\end{equation}
This is enough to obtain the nondegeneracy property,  thanks to \cite[Lemma 1]{Bu2012}. 

This time, it is the case $\delta>0$ which requires a more careful analysis{\BBB , so we start with the negative case}.
\begin{itemize}
\item Case $\delta<0$. Consider any open set $\tilde{\Om}$ contained in $\Om$ with $|\tilde{\Om}|\geq \frac{1}{2}|\Om|$. By monotonicity $\lambda_k(\Om)\leq \lambda_k(\tilde{\Om})$, hence testing minimality of $\Om$ against $\left(\frac{|\Om|}{|\tilde{\Om}|}\right)^{\frac{1}{n}}\tilde{\Om}$ we have 
\[
T(\Om)^{-1}-\delta\left[\left(\frac{|\tilde{\Om}|}{|\Om|}\right)^{\frac{2}{n}}-1\right]\lambda_k(\Om)\leq \left(\frac{|\tilde{\Om}|}{|\Om|}\right)^{\frac{n+2}{n}}T(\tilde{\Om})^{-1} \leq T(\tilde{\Om})^{-1}\]
which provides \eqref{eq_intminimality} for $|\delta|\ll k^{-\frac{2}{n}}$, using $\lambda_k(\Om)\lesssim k^{\frac{2}{n}}$ from Lemma \ref{lem_apriori}.
\item Case $\delta>0$. We proceed as before, using in addition {\BBB  Lemma} \ref{lem_bucur}.
\BBB Testing \EEE the minimality of $\Om$ against the competitor $\left(\frac{|\Om|}{|\tilde{\Om}|}\right)^{\frac{1}{n}}\tilde{\Om}$ we get
\begin{align*}
	T(\Om)^{-1}-T(\tilde{\Om})^{-1}&\leq \delta\left(\lambda_k(\tilde{\Om})-\lambda_k(\Om)\right)+\left(\left(\frac{\tilde{|\Om|}}{|\Om|}\right)^\frac{n+2}{n}-1\right)T(\tilde{\Om})^{-1}\\
	&\leq \delta e^\frac{1}{4\pi}k\lambda_k(\Om)^{1+\frac{n}{2}}\lambda_k(\tilde{\Om})\left[T( \Om)-T( \tilde{\Om})\right]-C_n|\tilde{\Om}\setminus\Om|\\
	&\leq C_n'\delta k^{2+\frac{4}{n}}\left[T( \tilde{\Om})^{-1}-T( \Om)^{-1}\right]-C_n|\tilde{\Om}\setminus\Om|
\end{align*}
for some $C_n,C_n'>0$,  using $T(\tilde{\Om})^{-1}\gtrsim T(B)^{-1}$ and also $\lambda_k(\Om)\lesssim k^{\frac{2}{n}}$ from Lemma \ref{lem_apriori}. We get \eqref{eq_intminimality} for $|\delta|\ll k^{-\left(2+\frac{4}{n}\right)}$.
\end{itemize} 
\textbf{Bound on $\text{diam}(\Om)$.}  We have shown above that $\|\nabla w_{\Om}\|_{L^\infty(\R^n)}\lesssim1$. Hence by the Gagliardo-Nirenberg inequality,
\begin{equation}\label{eq:wOm_wB_Linfty}\Vert w_\Om-w_B\Vert_{\C^{0}(\Rn)}\lesssim \Vert \nabla (w_\Om-w_B)\Vert_{L^\infty(\Rn)}^\frac{n}{n+1} \Vert w_\Om-w_B\Vert_{L^1(\Rn)}^\frac{1}{n+1}\lesssim \Vert w_\Om-w_B\Vert_{L^1(\Rn)}^\frac{1}{n+1}   \lesssim  \left(k^\frac{1}{n}|\delta|^\frac{1}{2}\right)^{\frac{1}{n+1}}\end{equation}
where we also used $\Vert w_\Om-w_B\Vert_{L^1(\Rn)}\lesssim k^\frac{1}{n}|\delta|^\frac{1}{2}$ from Lemma \ref{lem_apriori}. Let now $c_n$ denote the non-degeneracy constants found above. Then for any $x\in \Rn\setminus B_{2}$ we have
\[\fint_{\partial B_{x,1}}w_\Om =\fint_{\partial B_{x,1}}(w_\Om-w_B)\lesssim \left(k^\frac{1}{n}|\delta|^\frac{1}{2}\right)^{\frac{1}{n+1}}\]
and this is strictly less than $c_n$ for $|\delta|\ll k^{-\frac{2}{n}}$. Hence $w_\Om(x)=0$ for \BBB any $x\in\R^n\setminus B_2$\EEE, so that we find $\Om\subset B_{2}$. This gives the desired upper bound on $\text{diam}(\Om)$, thus concluding the proof.

\end{proof}

\subsection{Blow-ups and viscosity solutions}\label{sect:blowup}
We recall \BBB (see Remark \eqref{rem:scale_free}) \EEE that if $\Om$ is a minimizer of \eqref{eq_func}, then it is also a minimizer (among open sets of any measure) of

\[ J:A\mapsto |A|^\frac{2}{n}\left(\frac{|A|}{\om_nT(A)}+\delta\lambda_k(A)\right).\]
Consider a smooth vector field $\xi\in\C^\infty_c(\mathbb{R}^n,\mathbb{R}^n)$.  Assuming enough regularity on $\Om$, the shape derivative 
\[\left.\frac{d}{dt}\right|_{t=0}J((\text{Id}+t\xi)(\Om))\]
of this functional at $\Om$ in the direction $\xi$ is given by
\[\int_{\partial\Om}\left[\om_n^{\frac{2}{n}-1}\left(\frac{n+2}{nT(\Om)}+\frac{2}{n}\delta\lambda_k(\Om)\right)-\frac{\om_n^{\frac{2}{n}}}{T(\Om)^2}|\nabla w|^2-\om_n^{\frac{2}{n}}\delta |\nabla u_k|^2\right](\xi\cdot \nu_\Om)d\mathscr{H}^{n-1},\]
where $\nu_\Om$ is the outward unit normal vector of $\Om$  (see for instance \cite{HP18} for the expressions of the shape derivatives of $|\cdot|$, $T$ and $\lambda_k$). So letting
\begin{equation}\label{eq:defQ}Q:=\frac{T(\Om)^2}{\om_n}\left(\frac{n+2}{nT(\Om)}+\frac{2}{n}\delta\lambda_k(\Om)\right)\end{equation}
we expect an overdetermined boundary condition 
\[|\nabla w|^2+T(\Om)^2\delta |\nabla u_k|^2=Q\text{ on }\Om.\]
Note that when $\delta\to 0$, using that $T(\Om)^{-1}-T(B)^{-1}\lesssim |\delta|k^{\frac{2}{n}}$ and $\lambda_k(\Om)\lesssim k^{\frac{2}{n}}$ (see Lemma \ref{lem_apriori}) we find
\[Q\to \frac{1}{n^2},\]
which is expected because it corresponds to the value of $|\nabla w_B|^2_{|\partial B}$. We may {\BBB even} estimate its rate of convergence 
\begin{equation}\label{eq_estQ}
\left|Q-\frac{1}{n^2}\right|\lesssim k^{\frac{2}{n}}|\delta|.
\end{equation}

Let us now prove that these informal considerations hold true in the viscosity sense; for any ``contact point'' of the boundary as defined below, we may define the gradient of the function in a weak sense through a characterization of blow-ups for such points: this is the object of Lemma \ref{lem_blowup} below.

\begin{definition}\label{def_contact}
Let $\Om\subset\Rn$ be an open set and $x\in\partial\Om$. We say that $x$ is a contact point of $\Om$ if there exists $R>0$ and $\nu\in\mathbb{S}^{n-1}$ verifying
\[B_{x+R\nu, R}\subset\Om \text{ or }B_{x-R\nu,R}\subset \Rn\setminus\Om.\]
The vector $\nu$ is called the  inward ``normal'' vector of $\Om$ at $x$.
\end{definition}

For a function $w:\R^n\rightarrow\R$ and $z\in\Rn$, $r>0$ we will denote $(w)_{z,r}$ the rescaled function $(w)_{z,r}:\Rn\rightarrow\R$ defined by
\[(w)_{z,r}(x):=\frac{w(z+rx)}{r}.\]

Let us \BBB recall \EEE a classical lemma of the one-phase free boundary problem. \BBB A function $w$ which satisfies the following property  
\begin{equation}\label{eq_property_blowup}
|\nabla w|\leq C,\quad\text{ and }\quad \forall r\in  (0,1),x\in\Rn,\quad \fint_{\partial B_{x,r}}w<cr\text{ implies }w|_{B_{x,r/2}}=0\ 
\end{equation}
for some constants $c,C>0$, enjoys the following blow-up behaviour.\EEE
\begin{lemma}\label{lem_prelim_blowup}
Let $w\in \C^0(\Rn,\R_+)$ \BBB that satisfy \EEE $w(0)=0$,  property \eqref{eq_property_blowup} for some constants $c$ and $C$, and $|\Delta w|1_{\{w>0\}}\in L^\infty_\loc(\Rn)$. Then there exists a   \BBB sequence  \EEE $r_i\to 0$ and a function $\ov{w}\in\C^0(\Rn,\R_+)$ verifying $\ov{w}(0)=0$, property \eqref{eq_property_blowup} for the same constants $c,C$, $\Delta\ov{w}=0$ on $\{\ov{w}>0\}$ such that
\begin{align*}
w_{0,r_i}\underset{\C^0_\loc(\Rn)}{\cvg}\ov{w},\qquad &w_{0,r_i}\underset{\C^1_\loc(\{\ov{w}>0\})}{\cvg}\ov{w},\\
\{w_{0,r_i}>0\}\underset{loc.Hausd.}{\cvg}\{\ov{w}>0\},\qquad &\{w_{0,r_i}=0\}\underset{loc.Hausd.}{\cvg}\{\ov{w}=0\}.
\end{align*}

\end{lemma}
\begin{proof}
\BBB The \EEE functions $(w_{0,r})$ all verify the same Lipschitz bound so there exists a  \BBB sequence  \EEE $(w_{0,r_i})$ (for $r_i\to 0$) that converges locally uniformly to some $\ov{w}\in\C^0(\Rn,\R_+)$. Property \eqref{eq_property_blowup} directly transfers to $\ov{w}$ for the same constants. Letting now $U$ be some open set compactly included in $\{\ov{w}>0\}$, we have that $U\subset\{w_{0,r_i}>0\}$ for any large enough $i$ with $|\Delta w_{0,r_i}|\leq \|\Delta w\|_{L^\infty(U)}r_i$ on $U$, giving both that $\Delta \ov{w}=0$ in $\{\ov{w}>0\}$ and the local $\C^1$ convergence in the support. Finally, the local Hausdorff convergence of the supports and their complements is obtained by non-degeneracy and (near-)harmonicity of $w$ and $(w_{0,r_i})$, see for instance \cite[Section 6]{V19}.
\end{proof}

\begin{lemma}\label{lem_blowup}
Let $\Om\subset\Rn$ be a minimizer of \eqref{eq_func}. Let $z\in\partial\Om$ be a contact point of $\Om$ with inward {\BBB normal} vector $\nu$. Then provided $|\delta|\ll k^{-\left(4+\frac{8}{n}\right)}g_n(k)$, there exists $\alpha>0$, $\beta\in\R$ and a positive sequence $s_i\to 0$ such that
\begin{equation}\label{eq_blowup}
\begin{split}
	(w)_{z,s_i}\underset{{ \C^0_\loc(\Rn)}}{\longrightarrow} &\alpha (x\cdot\nu)_+,\\
	(u_k)_{z,s_i}\underset{{ \C^0_\loc(\Rn)}}{\longrightarrow} &\beta (x\cdot\nu)_+,\\
\end{split}
\end{equation}
as $i\to\infty$, and 
\begin{equation}\label{eq_alphabeta}
\alpha^2+T(\Om)^2\delta\beta^2=Q
\end{equation}
where $Q$ is defined in \eqref{eq:defQ}.
\end{lemma}
\BBB Before the proof of Lemma \ref{lem_blowup}, we state the following technical result{\BBB, in which we use the notation $\Hn=\{x\in\Rn:x_n>0\}$.}
\begin{lemma}\label{lem:db01}
There exists $\theta_n>0$ and $\eps_n>0$ such that the following property holds: for any $u\in H^1(B_1,\R_+)$, such that $\Vert u- x_n^+\Vert_{L^\infty(B_1\cap\Hn)}<\eps_n$, we have
\[\int_{B_1}|\nabla u|^2\geq \theta_n+\int_{B_1}|\nabla \mathcal{H}u|^2.\]
{\BBB where $\mathcal{H}u$ is the harmonic extension of $u_{|\partial B_1}$ in $B_1$.}
\end{lemma}
\begin{proof}
Suppose that $\|u-x_n^+\|_{L^\infty(B_1\cap\Hn)}\leq\eps$ for some $\eps>0$ and let us adjust $\eps$ so that the conclusion holds. We have 
\begin{align*}
	\int_{B}(|\nabla u|^2-|\nabla \mathcal{H}u|^2)&=\int_{B}|\nabla (u-\mathcal{H}u)|^2\geq \lambda_1(B)\int_{B}|u-\mathcal{H}u|^2\\ &\geq \frac{2\lambda_1(B)}{\om_n}\left(\int_{B\cap\Hn}(\mathcal{H}u-u)_+\right)^2,
\end{align*}
where we used Faber-Krahn and Cauchy-Schwarz inequalities. Since
\[\int_{B\cap\Hn}(\mathcal{H}u-u)_+\geq \int_{B\cap\Hn}(\mathcal{H}x_n^+ -x_n^+)-\om_n \eps
\]
then by taking $\eps=\eps_n :=\frac{1}{4}\fint_{B\cap\Hn}(\mathcal{H}x_n^+ -x_n^+)$ and $\theta_n:=\frac{1}{2}\int_{B\cap\Hn}(\mathcal{H}x_n^+ -x_n^+)$ the conclusion follows.
\end{proof}
\EEE

\begin{proof}[\BBB Proof of Lemma \ref{lem_blowup}]\BBB The proof is divided in three steps. \EEE Up to a displacement we {\BBB may} assume $z=0$, $\nu=e_n$. We also write 
\[w_r(x):=\frac{w(rx)}{r},\ u_{r}(x):=\frac{u_k(rx)}{r}\]
We start by proving that $(w,u_k)$ admits $(\alpha x_n^+,\beta x_n^+)$ as a blow-up (for some subsequence) at $0$ for some $\alpha>0$, $\beta\in\R$.

\smallskip
\noindent {\bf Step 1.} {\bf Blow-up for an exterior contact sphere}. 
Supposing that there is an exterior contact sphere $ \B:=B_{-Re_n,R}$, we prove that $w(x)=\alpha x_n^+ +o(|x|)$ and $u_k(x)=\beta x_n^++o(|x|)$ for some $\alpha>0$, $\beta\in\R$, thus getting \eqref{eq_blowup}. We follow the method of \cite[Lemma 11.17]{CS05}, using the non-degeneracy and Lipschitz bounds on $w$ and $u_k$ from Lemma \ref{lem_reg}.\newline
Let us first prove the expansion for $w$: set \[G(x)=\begin{cases} \left(R\log\left(\frac{|x+Re_n|}{R}\right)\right)_+&\text{ if }n=2\\ \frac{\left(R^{2-n}-|x+Re_n|^{2-n}\right)_+}{(n-2)R^{1-n}}&\text{ if }n\geq 3\end{cases},\quad \forall m\in\N, \ \alpha_m:=\inf\left\{\alpha\geq 0:w\leq \alpha G\text{ in }B_{2^{-m}}\right\}.\]
Above, $\alpha_m$ is well-defined and finite since $w(x)\lesssim d(x,\B)$ whereas $G(x)\gtrsim d(x,\B)$ on $B_1$. It is also bounded from below by a positive constant due to the non-degeneracy property. The sequence $(\alpha_m)$ decreases and therefore we can set $\alpha=\lim_{m\to\infty}\alpha_m$. We claim that $w(x)=\alpha G(x)+o(|x|)$, which is sufficient for proving the expansion, considering that $G(x)=x_n^++o(|x|)$.\newline

Suppose it is not the case, meaning there is some sequence of points $(x^p)_{p\in\N}$ in $\Rn\setminus\B$ converging to $0$ and some $\eps\in (0,1]$ such that
\[\forall p\in\N,\ w(x^p)<\alpha G(x^p)-\eps |x^p|.\]
We let $L$ be \BBB the \EEE Lipschitz constant of $w-\alpha G$, and we will suppose without loss of generality that $\eps\ll L$. Letting $y^p=x^p+\frac{\eps}{2L}|x^p|e_n$, then $\frac{1}{2}|x^p|\leq |y^p|\leq 2|x^p|$ and by the Lipschitz bounds we have
\[\forall p,\ w(y^p)-\alpha G(y^p)<-\frac{\eps}{4} |y^p|,\]
as well as
\[y^p_n= \frac{\eps}{2L}|x^p|+x_n\geq \frac{\eps}{4L}|y^p|-\frac{1}{2R}|x^p|^2\geq \frac{\eps}{8L}|y^p|,\]
where the last inequality holds for any large enough $p$ and we have used that $x_n^p\geq \frac{-1}{2R}|x^p|^2$ (since $x^p\notin \B$). We now let $r^p=|y^p|$ and $B^p$ the ball of center $y^p/r^p$ and of radius $\frac{\eps}{16L}$ on which we have, still by the Lipschitz bound,
\[(w-\alpha G)_{r^p}\leq -\frac{\eps}{8}\]
and for which $\text{dist}(B^p,\partial (B_2\setminus \B))\geq \frac{\eps}{16L}$. Let now $\varphi^p$ be the continuous function defined by
\[\begin{cases}
\varphi^p=\eta&\text{ in }B^p\\
\varphi^p=0 &\text{ on }\partial (B_2\setminus \BBB(\B/r^p)\EEE)\\
\Delta \varphi^p = 2r^p&\text{ in }(B_2\setminus \BBB(\B/r^p)\EEE)\setminus B^p
\end{cases}\]
where $\eta>0$ is fixed small ; if $\eta\leq \frac{\eps}{8}$ is small enough we have $\varphi^p\leq \alpha G_{r^p}$ in $B_2$ {\BBB  for all $p$,} by {\BBB the} maximum principle. Then for a large enough $p$ we have $\varphi^p\geq 0$ and $\varphi^p\geq cG_{r_p}$ in $B_{1/2}$ for some $c>0$ by Hopf's lemma. We claim that for a large enough $p$ in this case
\[\left(w-\left(\alpha+\frac{c}{2}\right)G\right)_{r^p}\leq -\varphi^p\]
by maximum principle in the domain $\om^p:=\{w_{r^p}>0\}\cap(B_2\setminus (\BBB(\B/r^p)\EEE\cup B^p))$. Indeed suppose $p$ is large enough such that $w_{r^p}\leq \left(\alpha+\frac{c}{2}\right)G_{r^p}$ in $B_2$, then  
\[\Delta\left(w-\left(\alpha+\frac{c}{2}\right)G\right)_{r^p}=-r^p>-2r^p=-\Delta\varphi^p\text{ on }\om^p\]
and the inequality is verified on $\partial\om^p$:
\begin{align*}
\left(w-\left(\alpha+\frac{c}{2}\right)G\right)_{r^p}+\varphi^p&\leq \varphi^p-\alpha G_{r^p}\leq 0\text{ on }\{w_{r^p}=0\},\\
\left(w-\left(\alpha+\frac{c}{2}\right)G\right)_{r^p}+\varphi^p&= w_{r^p}-\left(\alpha+\frac{c}{2}\right)G_{r^p}\leq 0\text{ on }\partial B_2,\\
\left(w-\left(\alpha+\frac{c}{2}\right)G\right)_{r^p}+\varphi^p&\leq -\frac{\eps}{8}+\varphi^p\leq 0\text{ on }B^p.\\
\end{align*}
This implies $w(x)\leq \left(\alpha-\frac{c}{2}\right)G(x)$ in some neighbourhood of the origin, which contradicts the definition of $\alpha=\inf_m\alpha_m$. This gives the announced expansion for $w$ and hence \eqref{eq_blowup} for $w$.\newline
The \BBB exact \EEE same reasoning can then be done for $w+c u_k$ for any $c$ chosen such that $w+c u_k$ is positive in its support and \BBB $\Delta(w+c u_k)\geq -2$ (which holds for $c$ small enough )\EEE, thus getting the existence of $\beta\in\R$ such that \eqref{eq_blowup} holds true for $u_k$. This finishes the proof of \eqref{eq_blowup} in the case of an exterior contact sphere. 

\smallskip
\noindent
\textbf{Step 2. Blow-up for an interior contact sphere}. Assume now that there is an interior contact sphere $B_{Re_n,R}\subset \Om$; in particular for any blow-up $(\ov{w},\ov{u}_k)$ of $(w,u_k)$ at $0$ we have $\Hn\subset\{\ov{w}>0\}$. We apply \cite[Lemma 11.17 and Remark 11.18]{CS05} to $\ov{w}$ (and $\ov{w}+c\ov{u_k}$ for a small enough $c\ll k^{-\left(\frac{1}{2}+\frac{2}{n}\right)}$); this gives $\ov{w}(x)=\alpha x_n^++o(|x|)$ and $ \ov{u_k}(x)=\beta x_n^++o(|x|)$ in $\Hn$ for some $\alpha>0$, $\beta\in\R$.\newline

We \BBB recall \EEE that a blow-up (at 0) of a blow-up (at 0) of $w$ is still a blow-up of $w$: indeed if $w_{0,r_i}\to \ov{w}$ and $\ov{w}_{0,s_i}\to\tilde{w}$ then there is some extraction $\varphi(i)$ such that $w_{0,r_{\varphi(i)}s_i}\to\tilde{w}$.

As a consequence, there is a blow-up of $w,u_k$ at $0$ (that we still denote $\ov{w},\ov{u_k}$) such that $\ov{w}(x)=\alpha x_n^+$ and $\ov{u_k}(x)=\beta x_n^+$ in $\Hn$.

We now prove that $\ov{w}(x)=o(|x|)$ on $\Rn\setminus \Hn$, which is enough to conclude since $(\alpha x_n^+,\beta x_n^+)$ is then a blow-up of $(w,u_k)$ at $0$. Arguing by contradiction we assume that $\ov{w}(x)=o(|x|)$ is not verified on $\Rn\setminus\Hn$, so that in particular $\{\ov{w}>0\}\cap(\Rn\setminus \ov{\Hn})$ is a non-empty open set which accumulates at $0$ and since $\ov{w}_{|\partial\Hn}\equiv 0$, then $\ov{w}1_{\Rn\setminus \Hn}$ is continuous and admits $B_{e_n,1}$ as an exterior contact sphere at $0$. We can therefore proceed as in the exterior sphere condition case to deduce that there exists $\gamma\geq 0$ such that $\ov{w}(x)=\gamma x_n^-+o(|x|)$ on $\{\ov{w}>0\}\cap(\Rn\setminus \Hn)$.

Thanks to the contradiction hypothesis we must have $\gamma>0$. In particular the density of $\{\ov{w}=0\}$ at the origin is zero. We \BBB recall \EEE that $\ov{w}$ is a blow-up of $w$ at $0$ for some sequence $r_i\to 0$, so
\begin{equation}\label{eq_density}
\lim_{\tau\to 0}\lim_{i\to\infty}\frac{|B_{\tau r_i}\setminus \Om|}{(\tau r_i)^n}=\lim_{\tau\to 0}\frac{|B_\tau\setminus \{\ov{w}>0\}|}{|B_\tau|}= 0.
\end{equation}

Let then $s_i=\tau r_i$ for some $\tau>0$ to be fixed later. 
We arrive \BBB at \EEE a contradiction by proving that the energy of $\tilde{\Om}_i:=\left(\frac{\om_n}{|\Om\cup B_{s_i}|}\right)^\frac{1}{n}\Om\cup B_{s_i}$ is strictly lower. \BBB We rely on Lemma \ref{lem:db01} to build a good competitor for $T(\Om\cup B_{s_i})$. Using \EEE the harmonic extension of $w$ in $B_{s_i}$ as a test function for $T(\Om\cup B_{s_i})$, and using the fact that $\frac{1}{\alpha}w_{s_i}\underset{\C^0(B\cap\Hn)}{\longrightarrow}x_n^+$, we find that {\BBB for} any large enough $i$,
\begin{align*}
T(\Om\cup B_{s_i})-T(\Om)&\geq s_i^n\int_{B_1}\left(2(\mathcal{H}u_{s_i} -u_{s_i})-|\nabla \mathcal{H}u_{s_i}|^2+|\nabla u_{s_i}|^2\right)\geq \theta_n \alpha s_i^n
\end{align*}
so $T(\Om\cup B_{s_i})^{-1}-T(\Om)^{-1}\gtrsim s_i^n$.
At the same time we have by Lemma \ref{lem_bucur}
\[\lambda_k(\Om)-\lambda_k(\Om\cup B_{s_i})\lesssim k^{2+\frac{4}{n}}(T(\Om\cup B_{s_i})-T(\Om))\lesssim  k^{2+\frac{4}{n}}\left(T(\Om)^{-1}-T(\Om\cup B_{s_i})^{-1}\right)\]
where we also used Lemma \ref{lem_apriori} to write $\lambda_k(\Om\cup B_{s_i})\leq\lambda_k(\Om)\lesssim  k^{\frac{2}{n}}$. We now compare the energy of $\Om$ and $\tilde{\Om}_i:=\left(\frac{\om_n}{|\Om\cup B_{s_i}|}\right)^\frac{1}{n}\Om\cup B_{s_i}$:
\begin{align*}
0&\leq \left(T^{-1}+\delta\lambda_k\right)\left(\tilde{\Om}_i\right)-
\left(T^{-1}+\delta\lambda_k\right)\left(\Om\right)\\
&= \left(\frac{|\Om\cup B_{s_i}|}{\omega_n}\right)^{\frac{n+2}{n}}T(\Om\cup B_{s_i})^{-1}-T(\Om)^{-1}+\delta\left(\left(\frac{|\Om\cup B_{s_i}|}{\omega_n}\right)^{\frac{2}{n}}\lambda_k(\Om\cup B_{s_i})-\lambda_k(\Om)\right)\\
&\leq \left(1-C_n|\delta|k^{2+\frac{4}{n}}\right)\left(T(\Om\cup B_{s_i})^{-1}-T(\Om)^{-1}\right)+(C_n'+k^\frac{2}{n}C_n'')|B_{s_i}\setminus \Om|
\end{align*}
for some constants $C_n,C_n',C_n''>0$: as a consequence, when $|\delta|\ll k^{-2-\frac{4}{n}}$ we get
\[|B_{s_i}\setminus \Om|\gtrsim T(\Om)^{-1}-T(\Om\cup B_{s_i})^{-1}\gtrsim s_i^n.\]
Finally, we get $|B_{s_i}\setminus\Om|\geq c_n s_i^n$ for all $i\in\N$, for some constant $c_n>0$, which is in contradiction with \eqref{eq_density} for large $i\in\N$ when $\tau$ is chosen small enough. As a consequence we have $\ov{w}(x)=o(|x|)$ on $\Hn$, thus finishing the proof of the interior sphere case. 

\smallskip
\noindent
\textbf{Step 3. Relation between $\alpha$ and $\beta$}. Let $\zeta\in\C^\infty_c(\Rn,\Rn)$ and $\zeta^t=\Id+t\zeta$, which is a diffeomorphism for any \BBB $t\in\R$ small enough\EEE. Since for $|\delta|\ll k^{-\left(4+\frac{8}{n}\right)}g_n(k)$ we have that $\lambda_k(\Om)$ is simple (by Corollary \ref{rem_estimate}){\BBB, so} we may compute the shape derivatives of $T$, $\lambda_k$ and $|\cdot|$ at the bounded open set $\Om$ (see respectively \cite[Proposition 6]{Lau20} and \cite[Theorem 2.6 (iii)]{LamLan06} for the derivatives of $T$ and $\lambda_k$). We have
\begin{align*}
\left.\frac{d}{dt}\right|_{t=0}\left|\zeta^t(\Om)\right|&=\int_{\Om}\nabla\cdot\zeta,\\
\left.\frac{d}{dt}\right|_{t=0}T\left(\zeta^t(\Om)\right)&=\int_{\Om}\left[\left(2w-|\nabla w|^2\right)\nabla\cdot\zeta +2\nabla w\cdot D\zeta\cdot\nabla w\right],\\
\left.\frac{d}{dt}\right|_{t=0}\lambda_k\left(\zeta^t(\Om)\right)&=\int_{\Om}\left[\left(|\nabla u_k|^2-\lambda_k(\Om)u_k^2\right)\nabla\cdot\zeta -2\nabla u_k\cdot D\zeta\cdot\nabla u_k\right].
\end{align*}
Thanks to $\Om$ being a minimizer of the scale-free functional \eqref{eq:scale_free}, the optimality condition writes
\[\left.\frac{d}{dt}\right|_{t=0}\left[\left|\zeta^t(\Om)\right|^\frac{2}{n} \left(\frac{\left|\zeta^t(\Om)\right|}{\om_n T\left(\zeta^t(\Om)\right)}+\delta \lambda_k\left(\zeta^t(\Om)\right)\right)\right]=0.\]
It gives, after simplification,
\begin{align*}
&\int_{\Om}\left[\left(|\nabla w|^2+T(\Om)^2\delta |\nabla u_k|^2+Q\right)\nabla\cdot\zeta -2\left(\nabla w\cdot D\zeta\cdot\nabla w+T(\Om)^2\delta \nabla u_k\cdot D\zeta\cdot\nabla u_k\right)\right]\\
&\hskip 4cm =\int_{\Om}\left( 2w+T(\Om)^2\delta\lambda_k(\Om)u_k^2\right)\nabla\cdot\zeta.
\end{align*}
We now replace $\zeta$ with $\zeta_i(x):=\zeta(x/s_i)$, where $s_i$ is a positive sequence for which \eqref{eq_blowup} holds, and we rescale the previous equality to obtain
$$
\int_{s_i^{-1}\Om} \Big(\left(|\nabla w_{s_i}|^2+T(\Om)^2\delta |\nabla u_{k,s_i}|^2+Q\right)\nabla\cdot\zeta \hskip 4cm $$
$$\hskip 3cm -2\left(\nabla w_{s_i}\cdot D\zeta\cdot\nabla w_{s_i}+T(\Om)^2\delta \nabla u_{k,s_i}\cdot D\zeta\cdot\nabla u_{k,s_i}\right) \Big)$$
$$
\hskip 2cm =s_i\int_{s_i^{-1}\Om}\left( 2w_{s_i}+s_iT(\Om)^2\delta\lambda_k(\Om)(u_{k,s_i})^2\right)\nabla\cdot\zeta.
$$
Now, we have shown in the first part that $w_{s_i}\to \alpha x_n^+$ and $u_{k,s_i}\to \beta x_n^+$ in the $\C^0_\loc(\Rn)\cap \C^1_\loc(\Hn)$ sense. Using also the $L^\infty$ and Lipschitz bounds on $w_{s_i}$ and $u_{k,s_i}$, and recalling that $s_i^{-1}\Om=\{w_{s_i}>0\}$ converges locally in $\Rn$ in the Hausdorff sense to $\Hn$ thanks to Lemma \ref{lem_prelim_blowup}, every term above passes to the limit and we get

\[\int_{\Hn}\left[\left(\alpha^2+T(\Om)^2\delta\beta^2 +Q\right)\nabla\cdot\zeta-2\left(\alpha^2+T(\Om)^2\delta\beta^2)\partial_n\zeta_n\right) \right]=0.\]
Applying Stoke's theorem, this gives $\int_{\partial\Hn}\left(\alpha^2+T(\Om)^2\delta\beta^2-Q\right)(\zeta\cdot e_n)=0$. Since $\zeta\in\C^\infty_c(\Rn,\Rn)$ was chosen arbitrarily, we conclude that $\alpha^2+T(\Om)^2\delta\beta^2=Q$.\end{proof}

\subsection{Minimizers are nearly spherical}\label{sect:min_are_NS}

In this section we {\BBB improve the regularity properties of $\Om$ solution to \eqref{eq_func} and} prove that under sufficient smallness of $\delta$, minimizers are nearly spherical sets. In this context, ``nearly spherical'' means that
\[\Om=the author:=\left\{s(1+h(x))x,\ s\in [0,1),\ x\in\partial B\right\},\]
where $h\in \C^{3,\gamma}\left(\partial B,[-\frac{1}{2},\frac{1}{2}]\right)$ with a bound  $\Vert h \Vert_{ \C^{3,\gamma}}\lesssim 1$. This is achieved {\BBB in two main steps: first in Lemma \ref{lem_C1gamma} we show a {\BBB $\C^{1,\gamma}$} regularity estimate by relying on recent results from free boundary theory, using \cite{KL18} for the $\delta>0$ case and \cite{MTV21} for the $\delta<0$ one. Then 
in Lemma \ref{lem_C2gamma}, we go further to obtain {\BBB $\C^{3,\gamma}$}-regularity}. 

It will be useful for us to consider \textit{centered} sets, where we say $\Om\in\A$ is centered when $\text{bar}(\Om):=\fint_{\Om}xdx$ is well-defined and equals zero.  \BBB We know from Lemma \ref{lem_reg} that minimizers are bounded \BBB, so that their barycenters are well defined. \EEE Note that since the functional under study $T^{-1}+\delta\lambda_k$ is translation invariant, there is no loss of generality in assuming that a given minimizer is centered.

\begin{lemma}\label{lem_convw}
Let $\Om$ be a centered minimizer of \eqref{eq_func} for $|\delta|\ll k^{-\left(2+\frac{4}{n}\right)}$. Then we have
\begin{align*}
\Vert w_\Om-w_B\Vert_{\C^{0}(\Rn)}&\lesssim \left(k^\frac{1}{n}|\delta|^\frac{1}{2}\right)^{\frac{1}{n+1}},\\
|\Om\Delta B|&\lesssim \mathcal{F}(\Om).
\end{align*}

\end{lemma}
\begin{proof}
Suppose that $\Om$ is translated into $\tilde{\Om}$ so that $\F(\tilde{\Om})=|\tilde{\Om}\Delta B|$ and $\Om=\tilde{\Om}-\text{bar}(\tilde{\Om})$. If $|\delta|\ll k^{-\left(2+\frac{4}{n}\right)}$ then we have shown in \eqref{eq:wOm_wB_Linfty} that $\Vert w_{\tilde{\Om}}-w_B\Vert_{\C^{0}(\Rn)}\lesssim \left(k^\frac{1}{n}|\delta|^\frac{1}{2}\right)^{\frac{1}{n+1}}$.
Now, since {\BBB the diameter of} $\tilde{\Om}$ is bounded by a dimensional constant thanks to Lemma \ref{lem_reg}, we have  $|\text{bar}(\tilde{\Om})|=|\text{bar}(\tilde{\Om})-\text{bar}(B)|\lesssim |\tilde{\Om}\Delta B|\lesssim k^{\frac{1}{n}}|\delta|^\frac{1}{2}$ using also Lemma \ref{lem_apriori}. As a consequence, we deduce
\[\Vert w_{\Om}-w_B\Vert_{\C^0(\Rn)}\leq \Vert w_{\tilde{\Om}}-w_B\Vert_{\C^0(\Rn)} +\Vert w_{B-\text{bar}(\tilde{\Om})}-w_B\Vert_{\C^0(\Rn)}\lesssim \left(k^\frac{1}{n}|\delta|^\frac{1}{2}\right)^{\frac{1}{n+1}}\]
as well as
\[|\Om\Delta B|\leq |\tilde{\Om}\Delta B|+|(B+\text{bar}(\tilde{\Om}))\Delta B|\lesssim\mathcal{F}(\tilde{\Om}).\]
\end{proof}

\begin{lemma}\label{lem_C1gamma}
Let $\Om$ be a centered minimizer of \eqref{eq_func}. If $|\delta|\ll k^{-\left(4+\frac{8}{n}\right)}g_n(k)$ then we have
\BBB \[\Om=\left\{s(1+h(x))x,\ s\in[0,1), \ x\in\partial B\right\},\] \EEE
where $h\in\C^{1,\gamma}\left(\partial B,[-\frac{1}{2},\frac{1}{2}]\right)$ for some $\gamma=\gamma_n\in (0,1)$  depending only on $n$, and $\Vert  h\Vert_{\C^{1,\gamma}(\partial B)}\lesssim 1$.
\end{lemma}
\BBB In the rest of the section, the value of  $\gamma$ will be fixed as given from this Lemma.\EEE

\begin{proof}
We {\BBB again} separate the cases $\delta<0$ and $\delta>0$. 

\smallskip
\noindent
\textbf{Case $\delta<0$.} Set 
\[\ov{w}=Q^{-\frac{1}{2}}\left(w+T(\Om)\sqrt{-\delta}u_k\right) \text{ and } \un{w}=Q^{-\frac{1}{2}}\left(w-T(\Om)\sqrt{-\delta}u_k\right).\] 
For $|\delta|\ll k^{-1-\frac{4}{n}}$ the functions $\ov{w}$ and $\un{w}$ are positive on their support, due to the estimates from Lemma \ref{lem_prelim}. Since $|\delta|\ll k^{-4-\frac{8}{n}}g_n(k)$, Lemma \ref{lem_blowup} applies to ensure that the couple $(\ov{w},\un{w})$ is a viscosity solution in the sense of \cite[Definition 2.4]{MTV21} of the system
\[\begin{cases}
-\Delta \ov{w}=Q^{-\frac{1}{2}}\left(1+T(\Om)\sqrt{-\delta}\lambda_k(\Om)u_k\right) & {\BBB \textrm{ in }}\Om,\\
-\Delta \un{w}=Q^{-\frac{1}{2}}\left(1-T(\Om)\sqrt{-\delta}\lambda_k(\Om)u_k\right) & {\BBB \textrm{ in }}\Om,\\
\ov{w},\un{w}>0 & {\BBB \textrm{ in }}\Om,\\
\ov{w}=\un{w}=0 & {\BBB \textrm{ on }}\partial\Om,\\
\partial_\nu \ov{w}\cdot\partial_\nu \un{w}=1 & {\BBB \textrm{ on }}\partial\Om.
\end{cases}\]
Note that $\ov{w}$ and $\un{w}$ both converge uniformly to $\left(\frac{1-|x|^2}{2}\right)_+$ as $k^{1+\frac{4}{n}}|\delta|\to 0$. In fact, by respectively Lemma \ref{lem_convw}, inequality \eqref{eq_estQ} and Lemma \ref{lem_prelim} we have
\begin{align*}
\Vert w_\Om-w_B\Vert_{\C^0(\Rn)}\lesssim \left(k^\frac{1}{n}|\delta|^\frac{1}{2}\right)^{\frac{1}{n+1}},\qquad \left|Q-\frac{1}{n^2}\right|\lesssim k^\frac{1}{n}|\delta|^\frac{1}{2},\qquad \sqrt{-\delta}|u_k|\lesssim k^{\frac{1}{2}}|\delta|^\frac{1}{2},
\end{align*}
so that
\[\left\Vert \ov{w}-\left(\frac{1-|x|^2}{2}\right)_{+}\right\Vert_{\C^0(\Rn)}+\left\Vert \un{w}- \left(\frac{1-|x|^2}{2}\right)_{+}\right\Vert_{\C^0(\Rn)}\lesssim  \left(k^{\frac{1}{n}}|\delta|^\frac{1}{2}\right)^\frac{1}{n+1}.\]
Our goal is now to apply the $\C^{1,\gamma}$ regularity theorem \cite[Theorem 3.1]{MTV21} for balls $B_{x,r}$ with $x\in \partial B$ and sufficiently small $r>0$. To simplify {\BBB notations} we assume that $x=-e_n$ and \BBB we \EEE let $r>0$ be a radius which will be fixed later. Since  $\|(w_B)_{-e_n,r}-x_n^+\|_{\C^0(B_1)}\lesssim r$ we deduce from the convergence of $\ov{w}$ and $\un{w}$
\begin{align*}
&\Vert \ov{w}_{-e_n,r}-x_n^+\Vert_{\C^0(B_1)}+\Vert \ov{w}_{-e_n,r}-x_n^+\Vert_{\C^0(B_1)}\lesssim r+r^{-1}\left(k^\frac{1}{n}|\delta|^\frac{1}{2}\right)^{\frac{1}{n+1}},\\[2mm]
&|\Delta \ov{w}_{-e_n,r}|+|\Delta \un{w}_{-e_n,r}|\lesssim r\text{ in }\{w_{-e_n,r}>0\},
\end{align*}
where we also used  $|\Delta \ov{w}_{-e_n,r}|,|\Delta \un{w}_{-e_n,r}|\lesssim r(1+|\delta|^{\frac{1}{2}}k^{\frac{1}{2}+\frac{4}{n}})\lesssim r$ (in $\Om_{-e_n,r}$) thanks to $\lambda_k(\Om)\lesssim k^{\frac{2}{n}}$ (Lemma \ref{lem_apriori}) and the choice of $\delta$.

We let $\eps:=\sqrt{r}$ and choose $r$ small enough so that the $\eps$-regularity Theorem \cite[Theorem 3.1]{MTV21} applies to $\eps$ (note that our inequalities are up to a dimensional constant, so the choice of $\eps$ may also differ up to a constant). Then when $\left(k^{\frac{1}{n}}|\delta|^\frac{1}{2}\right)^{\frac{1}{n+1}}\ll r^2$ the couple $(\ov{w}_{-e_n,r}, \un{w}_{-e_n,r})$ is $\eps$-flat so by \cite[Theorem 3.1]{MTV21}, $\partial\{\ov{w}_{-e_n,r}>0\}\cap B_{\frac{1}{2}}$ is a $\C^{1,\gamma}$ graph with controlled $\C^{1,\gamma}$ norm (for some dimensional constant $\gamma=\gamma_n\in (0,1)$), meaning that there exists $g:\left[-\frac{1}{2},\frac{1}{2}\right]\to [-1,1]$ such that
\[\{\ov{w}_{-e_n,r}>0\}\cap B_{\frac{1}{2}}=\left\{(x',x_n)\in\R^{n-1}\times\R,\ |x'|\BBB < \EEE\frac{1}{2},\ x_n{\BBB >} g(x)\right\},\ \Vert g\Vert_{\C^{1,\gamma}\left(\left[-\frac{1}{2},\frac{1}{2}\right]\right)}\lesssim 1.\]
This translates into the fact that $\partial\Om\cap B_{-e_n,r/2}$ is a $\C^{1,\gamma}$ graph with graph function $\tilde{g}:=r(-1+g)$. Covering $\partial \Om$ with a finite number of such balls of radius $r/2$ we find $h:\partial B\to \R$ such that $\partial\Om=\{(1+h(x))x,x\in\partial B\}$ with $\Vert h\Vert_{\C^{1,\gamma}(\partial B)}\lesssim 1$.

\smallskip
\noindent\textbf{Case $\delta>0$.} We analogously set 
\[\ov{w}=Q^{-\frac{1}{2}}\left(w+T(\Om)\sqrt{\delta}u_k\right) \text{ and } \un{w}=Q^{-\frac{1}{2}}\left(w-T(\Om)\sqrt{\delta}u_k\right),\] 
This time the couple $(\ov{w},\un{w})$ is a viscosity solution with boundary condition $\frac{(\partial_\nu \ov{w})^2+(\partial_\nu \un{w})^2}{2}=1$ in the sense given by \cite[Definition 4.1]{KL18}. By the exact same argument as in the previous case we may apply \cite[Theorem { 7.2}]{KL18} to the couple $(\ov{w}_{x,r}, \un{w}_{x,r})$ for $x\in\partial B$ and some dimensional $r$, thus providing again the existence of $h:\partial B\to \R$ such that $\partial\Om=\{(1+h(x))x,x\in\partial B\}$ with $\Vert h\Vert_{\C^{1,\gamma}(\partial B)}\lesssim 1$ (with a possibly different dimensional constant $\gamma$).

\end{proof}

\begin{lemma}\label{lem_C2gamma}
Let $\Om$ be a centered minimizer of \eqref{eq_func}. If $|\delta|\ll k^{-\left(4+\frac{8}{n}\right)}g_n(k)$ then
\BBB \[\Om=\left\{s(1+h(x))x,\ s\in[0,1), \ x\in\partial B\right\},\] \EEE
where  \BBB $h\in \C^{3,\gamma}\left(\partial B,[-\frac{1}{2},\frac{1}{2}]\right)$ and  $\Vert  h\Vert_{ \C^{3,\gamma}(\partial B)}\leq D_n$ for some dimensional $\gamma=\gamma_n\in (0,1)$ \BBB introduced in lemma \ref{lem_C1gamma} \EEE and $D_n>0$.
\end{lemma}

\begin{proof}   \footnote{\BBB A former version of this Lemma, albeit enough for the purposes of this paper, was initially proved by the authors \textit{via} a partial hodograph transform, by considering the torsion function $w_\Om$ as a coordinate: this type of approach is detailed for instance in \cite{KL18,CSY18}. Based on a suggestion of one of the reviewers (to whom we extend our thanks), we give a shorter and less computational proof based on the boundary Harnack inequality \cite[Theorem 2.4]{DS14}. \EEE}
Since $|\delta|\ll k^{-4-\frac{8}{n}}g_n(k)$ we may apply Lemma \ref{lem_C1gamma} to $\Om$, thus giving that $\Om$ is the  graph on the sphere of a function $h\in\C^{1,\gamma}(\partial B,\left[-\frac{1}{2},\frac{1}{2}\right])$ with $\|h\|_{\C^{1,\gamma}(\partial B)}\lesssim 1$.  By classical $\C^{1,\gamma}$ elliptic regularity  (see for instance \cite[Th. 3.13]{HL11})  and the $W^{1,\infty}$ bounds from Lemmas \ref{lem_prelim} and \ref{lem_reg} we have $ \Vert w_\Om\Vert_{\C^{1,\gamma}(\ov{\Om})}\lesssim 1$ and $\|u_k\|_{\C^{1,\gamma}(\ov{\Om})}\lesssim k^{\frac{1}{2}+\frac{4}{n}}$. 

We remind the optimality condition that is verified (now in the classical sense) on $\partial\Om$:
\[|\nabla w_\Om|^2+T(\Om)^2\delta|\nabla u_k|^2=Q\]
where $Q$ is the constant defined in equation \eqref{eq:defQ}, and satisfies $\left|Q-\frac{1}{n^2}\right|\lesssim k^\frac{2}{n}|\delta|$. In other words, \BBB noting that $|\nabla w_\Om|$ does not vanish on $\partial\Om$ thanks to Lemma \ref{lem_prelim} and $\lambda_k(\Om)\lesssim k^{\frac{2}{n}}$ (by Lemma \ref{lem_apriori}) \BBB, we have

\begin{equation}\label{eq_proofnearlycircular_optcondition}
|\nabla w_\Om|^2=\frac{Q}{1+T(\Om)^2\delta \frac{|\nabla u_k|^2}{|\nabla w_\Om|^2}}\text{ on }\partial\Om
\end{equation}

We claim that $\left\Vert \frac{|\nabla u_k|}{|\nabla w_\Om|}\right\Vert_{\C^{1,\gamma}(\partial\Om)}\lesssim k^{\frac{1}{2}+\frac{4}{n}}$, as a consequence of \cite[Th. 2.4]{DS14} applied to the ratio $\frac{u_k}{w_\Om}$. In \cite{DS14} this is stated for a harmonic denominator, so we introduce the auxiliary function 
$$v:\Om\setminus B_{1/2}\to \R$$
to be the harmonic extension of the boundary data $1$ on $\partial B_{1/2}$ (which does not meet $\partial\Om$) and $0$ on $\partial\Om$. By classical elliptic regularity we have
$$\Vert v\Vert_{\C^{1,\gamma}(\Om\setminus B_{1/2})}\lesssim 1.$$
By \cite[Th. 2.4]{DS14} applied to $(w_\Om,v)$ and $(u_k,v)$ \BBB in every ball $B_{p,1/2}$ for $p\in\partial B$\BBB, we have 
\begin{align*}
\left\Vert\frac{w_\Om}{v}\right\Vert_{\C^{1,\gamma}(\Om\setminus B_{3/4})}&\lesssim \Vert w_\Om\Vert_{L^\infty(\Om\setminus B_{1/2})}+\Vert \Delta w_\Om\Vert_{\C^{0,\gamma}(\Om\setminus B_{1/2})}\lesssim 1,\\
\left\Vert\frac{u_k}{v}\right\Vert_{\C^{1,\gamma}(\Om\setminus B_{3/4})}&\lesssim \Vert u_k\Vert_{L^\infty(\Om\setminus B_{1/2})}+\Vert \Delta u_k\Vert_{\C^{0,\gamma}(\Om\setminus B_{1/2})},\\
&\lesssim \Vert u_k\Vert_{L^\infty(\Om)}+\Vert \lambda_k u_k\Vert_{W^{1,\infty}(\Om)}\lesssim k^{\frac{1}{2}+\frac{4}{n}}.
\end{align*}
By maximum principle, since $v\lesssim w_\Om$ on $\partial B_{1/2}$ and both functions vanish on $\partial \Om$, then for some dimensional constant $c_n>0$ we have $\frac{w_\Om}{v}\geq c_n$ on $\Om\setminus B_{1/2}$, so by writing $\frac{u_k}{w_\Om}=\frac{u_k/v}{w_\Om/v}$ we have 
\[\left\Vert\frac{u_k}{w_\Om}\right\Vert_{\C^{1,\gamma}(\Om\setminus B_{3/4})}\lesssim k^{\frac{1}{2}+\frac{4}{n}}.\]
Since $u_k,w_\Om$ both vanish on $\partial\Om$, then the trace of $\frac{u_k}{w_\Om}$ on $\partial\Om$ is equal to the trace of $\frac{|\nabla u_k|}{|\nabla w_\Om|}$, which proves our earlier claim.

So when $|\delta|\ll k^{-1-\frac{8}{n}}$ (which is implied by our hypothesis), we obtain a bound
\[\Vert \nabla w_\Om\Vert_{\C^{1,\gamma}(\partial\Om)}\lesssim 1.\]

As a consequence, $w_\Om$ verifies in a distributional sense $\Delta w_\Om=-\chi_{\{w_\Om>0\}}+q\mathcal{H}^{n-1}\lfloor \partial\Om$ for some $q\in\C^{1,\gamma}(\R^n)$: by \cite[Th. 2]{KN77} (see also \cite[Th. 8.4]{AC81}, \cite[Lem. A.3]{KL18}), $\partial\Om$ is bounded in $\C^{2,\gamma}$, in the sense that \BBB $\Om$ is the graph on the sphere of some $h\in C^{2,\gamma}(\partial B)$ with in addition \BBB
\[\Vert h\Vert_{\C^{2,\gamma}(\partial B)}\lesssim 1.\]
Note that in both references, this is stated for harmonic functions instead of torsion functions, but the proof by hodograph transform adapts seamlessly. Then $\Delta w_\Om=-\chi_{\{w_\Om>0\}}+q\mathcal{H}^{n-1}\lfloor \partial\Om$ for some $q\in\C^{2,\gamma}(\R^n)$, and iterating \cite[Th. 2]{KN77} again gives that for some dimensional constant $D_n>0$ we have
\[\Vert h\Vert_{\C^{3,\gamma}(\partial B)}\leq D_n.\]

\end{proof}
\EEE

\subsection{Minimality of the ball among nearly spherical sets.}\label{subsec_NS}

This section is dedicated to the proof that {\BBB the ball is the unique local minimizer for \eqref{eq_func} in the class of nearly spherical sets}.  In the following definition we let $\gamma\in (0,1)$ and $D_n>0$ be given by Lemma \ref{lem_C2gamma} \BBB (or by Lemma \ref{lem_C2gammavect} later)\EEE.
\begin{definition}\label{def:NS}
\BBB An open set $\Om\subset\Rn$ is said to be nearly spherical if $|\Om|=|B|$ and there is a function $h\in \C^{3,\gamma}(\partial B,\left[-\frac{1}{2},\frac{1}{2}\right])$ with  $\Vert h\Vert_{ \C^{3,\gamma}(\partial B)}\leq D_n$  and such that $\Om=B_h$ \EEE, where
\[B_h:=\{s(1+h(x))x,s\in [0,1),x\in\partial B\}.\]
By convention, $h\nu_B$ will be extended as a vector field from $\Rn$ to $\Rn$ by the expression
\begin{equation}\label{eq:ext_hzeta}\zeta(x)=\varphi(|x|)h\left(\frac{x}{|x|}\right)\frac{x}{|x|},\end{equation}
where $\varphi\in\C_c^\infty(\R_+^*,[0,1])$ is such that $\varphi\equiv 1$ on $[1/2,3/2]$, $\varphi\equiv0$ on $[0,1/4]$ and $\varphi$ is nondecreasing on $[0,1/2]$. This way $\zeta^t(x)=x+t\zeta(x)$ is a \BBB $ \C^{3,\gamma}$ diffeomorphism from $B$ to $B_{th}$ \EEE for all $|t|\leq 1$.\\
Finally, we \BBB recall \EEE that $\Om$ is said to be \textit{centered} when its barycenter is at the origin.
\end{definition} 

To be consistent with the notation $B_r$ of the centered ball of radius $r$ it would probably be more natural to denote instead $B_{1+h}$ the nearly spherical set, but we will however carry on with the notation $B_h$ through the whole section for sake of simplicity. Note also that the values of $\gamma,D_n$, which {\BBB are} taken as in lemma \ref{lem_C2gamma}, do not matter as they could be replaced {\BBB in this section} by any $\gamma'\in (0,1)$, $D_n'>0$ (so long as they are fixed).\newline

The local minimality result is the following.
\begin{proposition}\label{prop_minloc}
Let $B_h$ be a nearly spherical centered set such that $\Vert h \Vert_{L^1(\partial B)}\ll k^{-1-\frac{4}{n}}g_n(k)$. Then when  $|\delta|\ll k^{-\left(1+\frac{8}{n}\right)}  g_n(k)$   we have
\[T(B_h)^{-1}+\delta\lambda_k(B_h)\geq T(B)^{-1}+\delta\lambda_k(B)\]
with equality if and only if $h\equiv 0$.
\end{proposition}

This will be obtained by performing a second-order Taylor expansion of the functional $T^{-1}+\delta\lambda_k$. The rough idea is the following: on the one hand, the first shape derivative (taken among measure-preserving variation) of $T^{-1}$ and $\lambda_k$ vanish, while on the other hand the second shape derivative of $T^{-1}$ is coercive in $H^{1/2}(\partial B)$ (in some sense that takes into account the invariance by translation){\BBB , see \eqref{eq:tors_coercive},} and the second shape derivative of $\lambda_k$ is bounded in $H^{1/2}(\partial B)$. This will be enough to get the local minimality of $T^{-1}+\delta \lambda_k$ at the ball.\newline

We begin \BBB with \EEE a Lemma which states that the eigenvalues and eigenfunctions may be \BBB approximated \EEE smoothly. It includes the case of degenerate eigenvalues, which will also be useful in the next section.
\begin{lemma}\label{lem:Kato} Let $h\in \C^{2,\gamma}\left(\partial B,\left[-\frac{1}{2},\frac{1}{2}\right]\right)$ and $\zeta$ the corresponding vector field (in accordance with \eqref{eq:ext_hzeta}). Then there exists real analytic functions 
\[t\in [-1,1]\mapsto \mu_i(t)\in \R, \  t\in [-1,1]\mapsto \widehat{u_i}(t)\in H_0^1(B)\] 
for every $i\in\N^*$, such that $\mu_{i}(0)=\lambda_{i}(B)$, and denoting $u_i(t)=\widehat{u}_i(t)\circ (\zeta^t)^{-1}$, the functions $(u_i(t))_{i\in\N^*}$ form an orthonormal basis of (non-ordered) eigenfunctions of $B_{th}$ associated to $(\mu_i(t))_{i\in\N^*}$ and \[\ t\in [-1,1]\mapsto u_i(t)\in L^2(\R^n)\]
is differentiable with $u_i'(t)\in H^1(B_{th})$. Moreover, we have the expressions
\begin{align*}
\mu_i'(t)&=-\int_{\partial B_{th}}|\nabla u_i(t)|^2 (\zeta\cdot \nu_t),\\
\mu_i''(t)&=\int_{\partial B_{th}}|\nabla u_i(t)|^2\left(H_t(\zeta\cdot \nu_t)^2-b_t( \zeta_{\tau_t},\zeta_{\tau_t})+2\zeta_{\tau_t}\cdot \nabla_{|\partial B_{th}}(\zeta\cdot\nu_t)\right)\\
&\hspace{1cm}+2\int_{B_{th}}\left(|\nabla u_i'(t)|^2-\mu_i(t)|u_i'(t)|^2\right),
\end{align*}
where $b_t$ is the second fundamental form of $\partial B_{th}$, $H_t$ its (outward) mean curvature, $\nu_t$ its (outward) normal vector and $\zeta_{\tau_t}:=\zeta-(\zeta\cdot\nu_t)\nu_t$. Finally, $u_i'(t)$ verifies
\begin{equation}\label{eq:u_i'}\begin{cases}
	-\Delta u_i'(t)-\mu_i(t) u_i'(t)=\mu_i'(t)u_i(t) & {\BBB \textrm{ in }}B_{th},\\[2mm] u_i'(t)=-(\zeta\cdot\nu_t)\partial_{\nu_t} u_i(t) & {\BBB \textrm{ on }}\partial B_{th}, \\[2mm] \forall j\in \N^*,\ \int_{B_{th}}(u_i'(t)u_j(t)+u_i(t)u_j'(t))=0.
\end{cases}\end{equation}
\end{lemma}

\begin{proof}
For each $i\in\N^*$ and $|t|\leq1$, $u_i$ is an eigenfunction on $B_{th}$ associated to $\lambda_i(B_{th})$ if and only if $\widehat{u_i}:=u_i\circ\zeta^t$ verifies 
\[\nabla\cdot\left[J_t(D\zeta^t)^{-1} ((D\zeta^t)^{-1})^*\nabla \widehat{u_i}\right]= \lambda_i( B_{th})J_t\widehat{u_i}, \]
where $J_t:=\text{det}(D\zeta^t)$.
Letting $\widehat{u_i}:=\frac{\widehat{v_i}}{\sqrt{J_t}}$,  the family $\left((\widehat{v_i},\lambda_i(B_{th}))\right)_{i\in\N^*}$  \BBB consists of the \EEE  eigenelements of the self-adjoint operator
\[T_tv :=-\frac{1}{\sqrt{J_t}}\nabla\cdot\left[J_t(D\zeta^t)^{-1} ((D\zeta^t)^{-1})^*\nabla \frac{v}{\sqrt{J_t}}\right].\]

We apply the result \cite[VII.3.5. Theorem 3.9]{Kat95} to the family of self-adjoint operators $T_t$ as defined above, over $L^2(B)$ with fixed domain $D(T_t)=D_0:=H^2(B)\cap H_0^1(B)$. This provides the existence of real analytic rearrangements $t\in [-1,1]\mapsto \mu_i(t)$ and $t\in[-1,1]\mapsto \widehat{u_i}(t)\in L^2(B)$ of respectively eigenvalues and orthonormal eigenfunctions for the operator $T_t$ such that $\mu_i(0)=\lambda_i(B)$ for all $i$. Writing \[\widehat{u_i}(t)=R_t\left(\mu_i(t)\widehat{u_i}(t)\right),\] 
where $R_t$ is the resolvent of $T_t$, and using that $t\in[-1,1]\mapsto R_t\in \mathcal{L}(H^{-1}(B),H_0^1(B))$ is real analytic (by \BBB the \EEE implicit function theorem), we improve the analyticity of the eigenfunctions into $t\in[-1,1]\mapsto \widehat{u_i}(t)\in H_0^1(B)$.\newline

Now, by construction we have that 
\[(u_i(t),\mu_i(t)):=\left(\frac{\widehat{u_i}(t)\circ (\zeta^t)^{-1}}{\sqrt{{ J_t}\circ(\zeta^t)^{-1}}},\mu_i(t)\right)\]
are eigenvalues and (orthonormal) eigenfunctions of the Dirichlet Laplacian over $B_{th}$. Since $t\mapsto \widehat{u_i}(t)\in H_0^1(B)$ is differentiable, one proves as in \cite[Theorem 5.3.1]{HP18} that the map $t\in[-1,1]\mapsto u_i(t)\in L^2(\R^n)$ is differentiable. The expressions of the first and second derivative are then classical formulas which we derive as in \cite[Section 5.9.3]{HP18}. 
Let us \BBB recall \EEE how these expressions are found.

\smallskip
\noindent
\textbf{First derivative.} The map $t\in[-1,1]\mapsto u_i(t)\in L^2(\R^n)$ is differentiable with derivative $u_i'(t)$ verifying $u_i'(t)+\nabla u_i(t)\cdot \zeta\in H_0^1(B_{th})$. One can therefore differentiate
\[\begin{cases}-\Delta u_i(t)=\mu_i(t)u_i(t),& {\BBB \textrm{ in }}B_{th}\\
\int_{B_{th}}u_i(t)u_j(t)=\delta_{ij}\end{cases} \]
to deduce that $u_i'(t)$ verifies the equation  and the boundary conditions of \eqref{eq:u_i'}.
Integrating by parts (see \cite[(5.87), (5.88)]{HP18}) we get the expression
\[\mu_i'(t)=-\int_{\partial B_{th}}(\partial_{\nu_t} u_i(t))^2 (\zeta\cdot \nu_t) \]

\smallskip
\noindent
\textbf{Second derivative.} We write the first derivative as an integral on the interior
\[\mu_i'(t)=-\int_{B_{th}}\nabla \cdot\left(|\nabla u_i(t)|^2\zeta\right)\] 
and apply the differentiation formula \cite[Corollary 5.2.8]{HP18}. The same computations as in \cite[Section 5.9.6]{HP18} lead to an analogous expression to the case of a simple eigenvalue:
\begin{align*}\mu_i''(t)&=2\int_{\partial B_{th}}u_i'(t)\partial_{\nu_t}u_i'(t)\\&\hspace{1cm}+\int_{\partial B_{th}}(\partial_{\nu_t}u_i(t))^2\left[H_t(\zeta\cdot \nu_t)^2-b_t((\zeta)_{\tau_t},(\zeta)_{\tau_t})+2\nabla_{\tau_t}(\zeta\cdot \nu_t)\cdot(\zeta)_{\tau_t}\right]\\
&= 2\int_{B_{th}}\left(|\nabla u_i'(t)|^2-\mu_i(t)|u_i'(t)|^2\right)\\
&\hspace{1cm}+\int_{\partial B_{th}}(\partial_{\nu_t}u_i(t))^2\left[H_t(\zeta\cdot \nu_t)^2-b_t((\zeta)_{\tau_t},(\zeta)_{\tau_t})+2\nabla_{\tau_t}(\zeta\cdot \nu_t)\cdot(\zeta)_{\tau_t}\right]\end{align*}
where  $\zeta_{\tau_t}:=\zeta-(\zeta\cdot\nu_t)\nu_t$  and $\nabla_{\tau_t}$ is the gradient over $\partial B_{th}$.

\end{proof}

\begin{proposition}\label{prop_controllambda_k}
Let $B_h$ be a nearly spherical set such that  $\Vert h\Vert_{L^1(\partial B)}\ll k^{-1-\frac{4}{n}}g_n(k)$, then 
\[|\lambda_k(B_h)-\lambda_k(B)|\leq C_n \frac{k^{1+\frac{8}{n}}}{g_n(k)}\Vert h\Vert_{H^{1/2}(\partial B)}^2.\]
\end{proposition}
Here the $H^{1/2}(\partial B)$ norm is defined as
\[\Vert h\Vert_{H^{1/2}(\partial B)}^2=\Vert h\Vert_{L^2(\partial B)}^2+\int_{B}|\nabla\mathcal{H}h|^2\]
where $\mathcal{H}h$ is the harmonic extension of $h$ in $B$. This is equivalent to the usual Gagliardo-Nirenberg norm $\Vert h\Vert_{L^2(\partial B)}^2+\iint_{(\partial B)^2}\frac{|h(x)-h(y)|^2}{|x-y|^n}$.
\begin{proof}[Proof of Proposition \ref{prop_controllambda_k}]
Recall that by definition of a nearly spherical set it holds $\|h\|_{\C^{3,\gamma}(\partial B)}\lesssim 1$.

Let $(u_i(t),\mu_i(t))_{i\in\N^*}$ be the eigenelements of $B_{th}$ as defined in Lemma \ref{lem:Kato}. We claim that for any $|t|\leq 1$ and $i\neq k$ we have $|\mu_i(t)-\mu_k(t)|\geq \frac{1}{2}g_n(k)$ when $\Vert h\Vert_{L^1(\partial B)}\ll k^{-1-\frac{4}{n}}g_n(k)$. Indeed we have $|u_i(t)|\lesssim \mu_i(t)^{\frac{n}{4}}$ by \cite[Example 2.1.8]{D89} and Proposition \ref{prop:growth_egv}, hence using classical elliptic regularity we deduce

\[\Vert u_i(t)\Vert_{\C^{1,\gamma}(B_{th})}\lesssim \Vert \mu_i(t) u_i(t)\Vert_{L^\infty(B_{th})}\lesssim \mu_i(t)^{1+\frac{n}{4}}.\]
Using the expression of $\mu_i'(t)$ from Lemma \ref{lem:Kato} we get
\[|\mu_i'(t)|\lesssim \Vert h\Vert_{L^1(\partial B)}\mu_i(t)^{2+\frac{n}{2}}.\]
Integrating this expression we get that for any $|t|\leq 1$ and $i\in\N$,
\[\left|\mu_i(t)^{-1-\frac{n}{2}}-\mu_i(0)^{-1-\frac{n}{2}}\right|\lesssim \Vert h \Vert_{L^1(\partial B)}\]
and for any $i\neq k$ we have 
\[\left|\mu_i(0)^{-1-\frac{n}{2}}-\mu_k(0)^{-1-\frac{n}{2}}\right|\gtrsim k^{-1-\frac{4}{n}}g_n(k),\]
so when $\Vert h \Vert_{L^1(\partial B)}\ll k^{-1-\frac{4}{n}}g_n(k)$ we get the claim. In particular this means that $\mu_k(t)=\lambda_k(B_{th})$ for such $h$. Again by elliptic regularity we have
\begin{equation}\label{eq:ellipticreg}
\Vert u_k(t)\Vert_{\C^{1,\gamma}(B_{th})}\lesssim k^{\frac{1}{2}+\frac{2}{n}},\ 
\Vert u_k(t)\Vert_{\C^{2,\gamma}(B_{th})}\lesssim k^{\frac{1}{2}+\frac{4}{n}}.
\end{equation}

The eigenvalue $\lambda_k(B)$ being simple, the associated eigenfunction $u_k(0)$ is radial, so that by setting $|\nabla u_k(0)|^2_{|\partial B}=:c_{n,k}(\lesssim k^{1+\frac{4}{n}})$ we have $\left.\frac{d}{ds}\right|_{s=0}\left(\lambda_k(B_{sh})+c_{n,k}|B_{sh}|\right)=0$. As a consequence, by Taylor expansion and recalling that $|B_{ h }|=|B|$, there exists some $t\in [0,1]$ such that
\[\lambda_k(B_{ h })-\lambda_k(B)=\frac{1}{2}\left(\left.\frac{d^2}{ds^2}\right|_{s=t}\lambda_k(B_{sh})+c_{n,k}\left.\frac{d^2}{ds^2}\right|_{s=t}|B_{sh}|\right).\]

To reduce notations we fix $t$ and do not write the dependency in $t$ in the rest of the proof, and set instead $\Om:=B_{th}$, $u_i:=u_i(t)$, $v:=u_k'(t)$ and write $\mu_i$, $\mu_i'$, $\mu_i''$ in place of $\mu_i(t)$, $\mu_i'(t)$, $\mu_i''(t)$. The expression of $\mu_k''$ thus reads
\[\mu_k''=\int_{\Om}2\left(|\nabla v|^2-\mu_k v^2\right)+\int_{\partial \Om}|\nabla u_k|^2\left[H(\zeta\cdot\nu)^2-b( \zeta_\tau,\zeta_\tau)+2\zeta_\tau\cdot\nabla_{|\partial\Om}(\zeta\cdot\nu)\right],\]
where $v$ verifies
\begin{equation}\label{eq:u_k'}\begin{cases}
	-\Delta v-\mu_k v=\mu_k' u_k & {\BBB \textrm{ in }}\Om,\\ v=-(\zeta\cdot\nu)\partial_\nu u_k & {\BBB \textrm{ on }}\partial\Om, \\ \int_{\Om}v u_k=0.
\end{cases}\end{equation}
On the other hand we have the following expression for the second derivative of the volume (see \cite[Theorem 2.1 and Lemma 2.8]{DL19}):
\[\left.\frac{d^2}{ds^2}\right|_{s=t}|B_{sh}|=\int_{\partial \Om}H(\zeta\cdot \nu)^2+\int_{\partial \Om}b( \zeta_\tau,\zeta_\tau)-2 \zeta_\tau \cdot\nabla_{|\partial\Om}(\zeta\cdot\nu).\]
We recall that $\zeta_{|\partial B}=h\nu_B$ is extended thanks to \eqref{eq:ext_hzeta}. 
Let us now bound each term independently.

\smallskip
\noindent 
{\bf 
Estimate of $v$.} We write $v=z+w$, where $z$ is harmonic and $w\in H^1_0(\Om)$. In particular $z$ is the harmonic extension of $-(\zeta\cdot\nu)\partial_\nu u_k$ so
\begin{align*}
\Vert z\Vert_{H^1(\Om)}&\lesssim \Vert (\zeta\cdot\nu)\partial_\nu u_k\Vert_{H^{1/2}(\partial\Om)}\lesssim  k^{\frac{1}{2}+\frac{4}{n}}\Vert h\Vert_{H^{1/2}(\partial B)},\\
\Vert z\Vert_{L^2(\Om)}&\lesssim \Vert z \Vert_{L^2(\partial\Om)}\lesssim  k^{\frac{1}{2}+\frac{2}{n}}\Vert h\Vert_{H^{1/2}(\partial B)},
\end{align*}
where we used  $\|h\|_{\C^{2,\gamma}}\lesssim 1$.
Above, we used the general property 
\[\Vert fg\Vert_{H^{1/2}(\partial B)}\lesssim \Vert f\Vert_{L^{\infty}(\partial B)}\Vert g\Vert_{H^{1/2}(\partial B)}+\Vert \nabla f\Vert_{L^\infty(\partial B)}\Vert g\Vert_{L^{2}(\partial B)},\]
which is a consequence of the inequality
\begin{align*}
\int_{B}|\nabla\mathcal{H}(fg)|^2&\leq\int_{B}|\nabla(\mathcal{H}(f)\mathcal{H}(g))|^2\leq 2\int_{B}|\nabla \mathcal{H}f|^2|\mathcal{H}g|^2+|\mathcal{H}f|^2|\nabla\mathcal{H}g|^2\\
&\leq 2\Vert \nabla f\Vert_{L^\infty(\partial B)}^2\Vert g\Vert_{L^2(\partial B)}^2+\Vert f\Vert_{L^\infty(\partial B)}^2\int_{B}|\nabla \mathcal{H}g|^2.
\end{align*}

Now $w$ verifies $-\Delta w-\mu_k w=\mu_k z+\mu_k' u_k$, with  $\int_{\Om}wu_k=\frac{\mu_k'}{\mu_k}$. {\BBB For every} $i\neq k$, this equation implies
\[(\mu_i-\mu_k)\int_{\Om}wu_i=\mu_k\int_{\Om}zu_i,\]
whereas for $i=k$ we simply have
\[\int_\Om wu_k=-\int_{\Om}zu_k.\]
Using the spectral decomposition, we have
\begin{align*}
\int_{\Om}\left(|\nabla v|^2-\mu_k v^2\right)&=\int_{\Om}\left(|\nabla z|^2+|\nabla w|^2-\mu_kv^2\right)\\
&=\int_{\Om}|\nabla z|^2+\sum_{i\in\N^*}\left(\mu_i\left(\int_{\Om}wu_i\right)^2-\mu_k\left(\int_{\Om}vu_i\right)^2\right)\\
&=\int_{\Om}|\nabla z|^2+\mu_k\left(\int_{\Om}zu_k\right)^2+\sum_{i\neq k}\left(\frac{\mu_k^2}{\mu_k-\mu_i}-\mu_k\right)\left(\int_{\Om}z u_i\right)^2,
\end{align*}
so
\begin{align*}
\left|\int_{\Om}\left(|\nabla v|^2-\mu_k v^2\right)\right|&\lesssim \Vert \nabla z\Vert_{L^2(\Om)}^2+k^\frac{2}{n}\Vert z\Vert_{L^2(\Om)}^2+g_n(k)^{-1}k^\frac{4}{n}\Vert z\Vert_{L^2(\Om)}^2\\
&\lesssim g_n(k)^{-1}k^{1+\frac{8}{n}}\Vert h\Vert_{H^{1/2}(\partial B)}^2 .
\end{align*}

\smallskip
\noindent 
{\bf  Curvature terms.} We directly have,  using again  $\|h\|_{\C^{2,\gamma}}\lesssim 1$:
\[\left|\int_{\partial\Om}|\nabla u_k|^2\left[H(\zeta\cdot\nu)^2-b( \zeta_\tau,\zeta_\tau)\right]\right|\lesssim \Vert\nabla u_k\Vert_{L^\infty(\Om)}^2\Vert\zeta\Vert_{L^2(\partial\Om)}^2\lesssim k^{1+\frac{4}{n}}\Vert h\Vert_{H^{1/2}(\partial B)}^2,\]
\[\left|\int_{\partial\Om}\left[H(\zeta\cdot\nu)^2+b( \zeta_\tau,\zeta_\tau)\right]\right|\lesssim \Vert h\Vert_{H^{1/2}(\partial B)}^2.\]

\smallskip
\noindent 
{\bf  Last term.} We have $\Vert\nabla_{|\partial\Om}\zeta\Vert_{H^{-1/2}(\partial\Om)}\lesssim \Vert\zeta\Vert_{H^{1/2}(\partial\Om)}$ and 
\begin{align*}
\Vert|\nabla u_k|^2\zeta\Vert_{H^{1/2}(\partial\Om)}&\lesssim \Vert\nabla_{|\partial\Om}|\nabla u_k|^2\Vert_{L^\infty(\partial\Om)}\Vert\zeta\Vert_{H^{1/2}(\partial\Om)}\\
&\lesssim\Vert\nabla u_k\Vert_{L^\infty(\Om)}\Vert\nabla^2 u_k\Vert_{L^\infty(\Om)}\Vert\zeta\Vert_{H^{1/2}(\partial\Om)}\lesssim k^{1+\frac{6}{n}}\Vert h\Vert_{H^{1/2}(\partial B)},
\end{align*}
so
\begin{align*}
\left|\int_{\partial\Om}|\nabla u_k|^2 \zeta_\tau\cdot \nabla _{|\partial\Om}(\zeta\cdot\nu)\right|&\lesssim k^{1+\frac{6}{n}}\Vert h\Vert_{H^{1/2}(\partial B)}^2.\end{align*}
We also have 
\[\left|\int_{\partial\Om} \zeta_{\tau}\cdot \nabla _{|\partial\Om}(\zeta\cdot\nu)\right|\lesssim \Vert  h \Vert_{H^{1/2}(\partial B)}^2.\]
\smallskip

Adding all the estimates we get \[|\mu_k''|\lesssim \frac{k^{1+\frac{8}{n}}}{g_n(k)}\Vert  h \Vert_{H^{1/2}(\partial B)}^2\lesssim \frac{k^{1+\frac{8}{n}}}{g_n(k)}\Vert  h \Vert_{H^{1/2}(\partial B)}^2\]
and
\[ \left|c_{n,k}\left.\frac{d^2}{ds^2}\right|_{s=t}|B_{sh}|\right|\lesssim k^{1+\frac{4}{n}}\|h\|_{H^{1/2}(\partial B)}^2,\]
{\BBB therefore}
\[|\mu_k(B_h)-\mu_k(B)|\leq |\mu_k''|\lesssim \frac{k^{1+\frac{8}{n}}}{g_n(k)}\Vert  h \Vert_{H^{1/2}(\partial B)}^2.\]

\end{proof}
We can now prove minimality of the ball for nearly spherical sets. 
\begin{proof}[Proof of Proposition  \ref{prop_minloc}]
It was proven in \cite[Theorem 3.3]{BDV15} that for any $B_h$ with $\text{bar}(B_h)=0$ and $\Vert h\Vert_{\C^{2,\gamma}(\partial B)}\ll 1$ it holds
\begin{equation}\label{eq:tors_coercive}
T(B_h)\leq T(B)-\frac{1}{32n^2}\Vert h\Vert_{H^{1/2}(\partial B)}^2.
\end{equation}
By interpolation, using for instance the interpolation inequalities between Hölder space 
\BBB
\[\Vert h\Vert_{\C^{2,\gamma}(\partial B)}\lesssim \Vert h\Vert_{L^1(\partial B)}^{\kappa}\Vert h\Vert_{ \C^{3,\gamma}(\partial B)}^{1-\kappa}\]
for some $\kappa_n\in (0,1)$\EEE.
Then  $\Vert h\Vert_{\C^{2,\gamma}(\partial B)}\ll \left(k^{-1-\frac{4}{n}}g_n(k)\right)^{\kappa}\ll 1$. We can therefore apply \eqref{eq:tors_coercive}, which together with Proposition \ref{prop_controllambda_k} yields
\begin{align*}T(B_h)^{-1}+\delta\lambda_k(B_h)&\geq \left(T(B)-\frac{1}{32n^2}\|h\|_{H^{1/2}(\partial B)}^2\right)^{-1}+\delta\lambda_k(B)-C_n|\delta|\frac{k^{1+\frac{8}{n}}}{g_n(k)}\|h\|_{H^{1/2}(\partial B)}^2\\
&\geq T(B)^{-1}+\delta\lambda_k(B)\end{align*}
where the last line holds provided $\delta$ is sufficiently small ($|\delta|\ll k^{-\left(1+\frac{8}{n}\right)} g_n(k)$),    with equality if and only if $h=0$. This finishes the proof.
\end{proof}

\subsection{Conclusion.}

\begin{proof}[Proof of Proposition \ref{main2_lin}]
When $|\delta|\ll k^{-4-\frac{8}{n}}g_n(k)$, Proposition \ref{prop_existence} applies and there exists a minimizer $\Om$ to the functional \eqref{eq_func}. By Lemma \ref{lem_C2gamma}, up to translation $\Om$ is a centered minimizer of the form $B_h$ with  $\Vert h\Vert_{ \C^{3,\gamma}(\partial B)}\leq D_n$, and hence is nearly spherical in the sense of Definition \ref{def:NS}. By Lemma \ref{lem_convw} and Corollary \ref{rem_estimate} we have 
\[\Vert h\Vert_{L^1(\partial B)}\lesssim |\Om\Delta B|\lesssim \mathcal{F}(\Om)\lesssim k^{2+\frac{4}{n}}|\delta|.\]

Now, for $\delta$ such that $k^{2+\frac{4}{n}}|\delta|\ll k^{-1-\frac{4}{n}}g_n(k)$ (meaning $|\delta|\ll k^{-3-\frac{8}{n}}g_n(k)$)  we can apply Proposition \ref{prop_minloc} to $\Om$, so that we obtain that $\Om$ is a ball. This finishes the proof.\end{proof}

\section{Proof of Theorem \ref{main_linvect}: linear bound on clusters of multiple eigenvalues}\label{sec_linvect}
{\BBB In this section, we} consider $1\leq k\leq l$ such that
\[\lambda_{k-1}(B)<\lambda_k(B)=\dots = \lambda_l(B)<\lambda_{l+1}(B).\]
We  denote the multiplicity of the eigenspace by $m=l-k+1\geq1$ and note that $l$ and $k$ are comparable since (see Lemma \ref{prop:growth_egv}) 
\[l^\frac{2}{n}\lesssim \lambda_l(B)=\lambda_k(B)\lesssim k^\frac{2}{n}.\]
In dimension $n=2$ it is known that the multiplicity of an eigenspace is at most $2$, {\BBB but}  in dimensions $n\geq 3$ the multiplicity of the eigenspace may get arbitrarily large, since the dimension of degree $d$ homogeneous harmonic polynomials in three variables is $2d+1$. \newline

Similarly to Section \ref{sec_lin}, in order to prove Theorem \ref{main_linvect}, we will prove {\BBB the equivalent formulation that the ball is a}  minimizer of the functional
\begin{equation}\label{eq_funcvect}
\Om\in\A\mapsto T(\Om)^{-1}+\delta\sum_{i=k}^{l}\lambda_i(\Om),
\end{equation}
for any $\delta\in\R$ sufficiently close to $0$. More precisely, we prove the following.
\begin{proposition}\label{main2_linvect}
There exists $c_{n}>0$ such that for any $\delta\in\R$ with $|\delta|\leq c_{n}  k^{-\left(6+\frac{10}{n}\right)}g_n(k)$, the ball is the unique minimizer of \eqref{eq_funcvect}.
\end{proposition} 
\begin{remark}\label{rem_mainlinvect}\rm 
This result admits the following natural generalization, following the same proof: let $k\leq l$ such that $\lambda_{k-1}(B)<\lambda_k(B)$ and $\lambda_{l}(B)<\lambda_{l+1}(B)$ (note that we do not ask $\lambda_l(B)=\lambda_k(B)$). Then for any $\Om\in\A$:
\[
\left|\sum_{i=k}^{l}\big(\lambda_i(\Om)-\lambda_i(B)\big)\right|\leq \frac{C_n l^{6+\frac{10}{n}}}{\min\left\{\lambda_{l+1}(B)-\lambda_{l}(B),\lambda_k(B)-\lambda_{k-1}(B)\right\}}\left(T(\Om)^{-1}-T(B)^{-1}\right),\]
where $C_n>0$ is some dimensional constant.
\end{remark}

The proof {\BBB of Proposition \ref{main2_linvect}} follows the same plan as in the non-degenerate case {\BBB (Section \ref{sec_lin}). Several steps of the proof will be very similar, particularly} the existence of a solution for \eqref{eq_funcvect}, the first regularity estimates and the existence of blow-ups are proved in the same way \BBB as \EEE  in Section \ref{sec_lin} (see Proposition \ref{prop_existence}, Proposition \ref{lem_reg} and Proposition \ref{lem_blowup} respectively). {\BBB We gather these results just below} and emphasize the slight differences in the proofs. Going from $\C^{1,\gamma}$ to $\C^{3,\gamma}$ is also similar{\BBB, see Lemma \ref{lem_C2gammavect}}.	

The main difference with the case when $\lambda_k(B)$ is simple concerns the {\BBB $\C^{1,\gamma}$} estimates of a minimizer $\Om${\BBB, which was the purpose of Lemma \ref{lem_C1gamma} in the previous section}. When $\delta>0$ we {\BBB are able to }again apply the results from \cite{KL18}, {\BBB but we need to obtain estimates that are} uniform in the multiplicity $m$. {\BBB However}, when $\delta<0$ we cannot directly apply \cite{MTV21} {\BBB as we did in Section \ref{sec_lin}}; instead we {\BBB have to }see \eqref{eq_funcvect} as a vectorial version of the problem studied in \cite{MTV21}, and follow the strategy of \cite{MTV21} (as in \cite{DS11} for the one-phase free boundary problem) by proving first some partial Harnack inequality (see Proposition \ref{prop_harnack}) and then get by contradiction an improvement of flatness (see Proposition \ref{prop_flatness_improv}) in order to get {\BBB $\C^{1,\gamma}$} regularity of a minimizer $\Om$. In this second case we {\BBB must} also follow carefully the dependency of the estimates in the multiplicity. {\BBB These new technical results are gathered in Sections \ref{ssect:harnack} and \ref{ssect:flatness} below, but as announced, we start by summarizing the first steps of the strategy for which the proofs are very similar to the ones in Section \ref{sec_lin}:}

\begin{lemma}\label{lem_apriorivect}Let $\Om$ be a domain such that 
\[T(\Om)^{-1}+\delta\sum_{i=k}^{l}\lambda_i(\Om)\leq T(B)^{-1}+\delta\sum_{i=k}^{l}\lambda_i(B).\]
Then if $|\delta|\ll k^{-1-\frac{2}{n}}$, we have the following properties (up to a translation of $\Om$):
\begin{itemize}[label=\textbullet]
\item $|\Om\Delta B|\lesssim k^{\frac{1}{2}+\frac{1}{n}}|\delta|^\frac{1}{2}$.
\item It holds
\begin{align*} T(\Om)^{-1}\lesssim 1, &\text{ and for }k\leq i\leq l, \ \lambda_i(\Om)\lesssim k^{\frac{2}{n}},\\ 
	T(\Om)^{-1}-T(B)^{-1}\lesssim k^{1+\frac{2}{n}}|\delta|,\ &\text{ and for all } i\in\N^*, |\lambda_i(\Om)-\lambda_i(B)|\lesssim i^{2+\frac{4}{n}}k^{\frac{1}{2}+\frac{1}{n}}|\delta|^\frac{1}{2}.\end{align*}
\item $\Vert w_\Om-w_B\Vert_{L^1(\Rn)}\lesssim k^{\frac{1}{2}+\frac{1}{n}}|\delta|^\frac{1}{2}$.
\end{itemize}
\end{lemma}

\begin{proof}\BBB
This follows the proof of Lemma \ref{lem_apriori}. Thanks to the upper bound from Lemma \eqref{est_eigen}, we have
\[T(\Om)^{-1}-T(B)^{-1}\leq\delta \sum_{i=k}^{l}\left(\lambda_i(B)-\lambda_i(\Om)\right)\lesssim k^{1+\frac{2}{n}}\lambda_1(\Om)|\delta|\lesssim  k^{1+\frac{2}{n}}T(\Om)^{-1}|\delta|\]
so when $|\delta|\ll k^{-1-\frac{2}{n}}$ we get $T(\Om)^{-1}\lesssim 1$, and using the quantitative Saint-Venant inequality \eqref{eq:QSV},  we get	\[|\Om\Delta B|\lesssim \sqrt{T(B)-T(\Om)}\lesssim k^{\frac{1}{2}+\frac{1}{n}}|\delta|^\frac{1}{2}.\]
Then Theorem \ref{main_sqrt} applied to any $i\in\mathbb{N}^*$ gives
\[|\lambda_i(\Om)-\lambda_i(B)|\lesssim i^{2+\frac{4}{n}}k^{\frac{1}{2}+\frac{1}{n}}|\delta|^\frac{1}{2}.\]
For the third item, we write as in the proof of Lemma \ref{lem_apriori}
\begin{align*}
\Vert w_{\Om}-w_B\Vert_{L^1(\Rn)}&\leq \Vert w_{\Om}-w_{\Om\cap B}\Vert_{L^1(\Rn)}+\Vert w_{B}-w_{\Om\cap B}\Vert_{L^1(\Rn)}=T(B)-T(\Om)+2(T(\Om)-T(\Om\cap B))\\
&\leq T(B)-T(\Om)+\left(\frac{1}{n}+\frac{1}{n^2}\right)|\Om\Delta B|\lesssim k^{\frac{1}{2}+\frac{1}{n}}|\delta|^\frac{1}{2}.
\end{align*}
\end{proof}

\BBB
\begin{corollary}\label{rem_estimatevect}
Let $\Om\in\A$ and $\delta$ satisfy the same hypotheses as in Lemma \ref{lem_apriorivect}. Then
\begin{align*}
T(\Om)^{-1}-T(B)^{-1}&\lesssim k^{6+\frac{8}{n}}|\delta|^2,\\
\forall i\in\N^*,\ |\lambda_i(\Om)-\lambda_i(B)|&\lesssim i^{2+\frac{4}{n}}k^{3+\frac{4}{n}}|\delta|.
\end{align*}
\end{corollary}

\begin{proof}
We follow the proof of Corollary \ref{rem_estimate}. Thanks to the hypothesis on $\Om$ and Theorem \ref{main_sqrt} we have
\[T(\Om)^{-1}-T(B)^{-1}\leq\delta \sum_{i=k}^{l}\left(\lambda_i(B)-\lambda_i(\Om)\right)\lesssim k^{3+\frac{4}{n}}|\delta|\left(T(\Om)^{-1}-T(B)^{-1}\right)^{\frac{1}{2}}.\]
This gives the first estimate. The second estimate follows by applying Theorem \ref{main_sqrt} again.
\end{proof}

\EEE

\BBB

{\BBB In the following result we adapt the results of Proposition \ref{prop_existence}, Lemma \ref{lem_reg} and Lemma \ref{lem_convw} to the case of multiple eigenvalues.}
\begin{proposition}\label{prop_existenceregvect}
If $|\delta|\ll k^{-\left(3+\frac{4}{n}\right)}$, then \BBB the \EEE functional \eqref{eq_funcvect} has a minimizer $\Om \in\A${\BBB. Besides,} there exists $c_n,C_n>0$ such that $\|\nabla w_\Om\|_{L^\infty(\Rn)}\leq C_{n}$, for all $k\leq i \leq l$, $\|\nabla u_i\|_{L^\infty(\Rn)}\leq C_n k^{\frac{1}{2}+\frac{2}{n}}$ and for all $x\in\Rn$, $r\in (0,1)$,
\[\fint_{\partial B_{x,r}}w_\Om< c_{n}r\text{ implies } w_{\Om}|_{B_{x,r/2}}= 0.\]
Moreover, \BBB $\Om$ is bounded \EEE with $\text{diam}(\Om)\lesssim 1$, and up to translating $\Om$ we have
\begin{align*}
\Vert w_\Om-w_B\Vert_{\C^{0}(\Rn)}&\lesssim \left(k^{\frac{1}{2}+\frac{1}{n}}|\delta|^\frac{1}{2}\right)^{\frac{1}{n+1}},\\
|\Om\Delta B|&\lesssim k^{3+\frac{4}{n}}|\delta|.
\end{align*}
\end{proposition}
\begin{proof}
This proof is completely similar to the proofs of \BBB Proposition \EEE  \ref{prop_existence} and \BBB Lemmas \EEE \ref{lem_reg}, \ref{lem_convw}.  The condition $|\delta|\ll k^{-\left(3+\frac{4}{n}\right)}$ in place of $|\delta|\ll k^{-\left(2+\frac{4}{n}\right)}$ in Proposition \ref{prop_existence} again comes from the multiplicity estimate $m\lesssim k$, as well as some of the estimates above.

\end{proof}
\EEE
Analogously to the case of a simple eigenvalue, we set
\begin{equation}\label{eq:defQ_mult}Q:=\frac{T(\Om)^2}{\om_n}\left(\frac{n+2}{nT(\Om)}+\frac{2}{n}\delta\sum_{i=k}^l\lambda_i(\Om)\right)\end{equation}
and we have
\begin{equation}\label{eq:Qspeed_vect}\left|Q-\frac{1}{n^2}\right|\lesssim k^{1+\frac{2}{n}}|\delta|.\end{equation}

\begin{lemma}\label{lem_blowupvect}
Let $\Om \in \A$ be a minimizer of \eqref{eq_funcvect} and suppose $|\delta|\ll k^{-\left(6+\frac{8}{n}\right)}g_n(k)$. Suppose $z\in\partial\Om$ has a contact sphere on either side with inward {\BBB normal} vector $\nu$. Then there exists $ \beta>0$, $\beta_i\in\R$, and a sequence $s_j\to 0$ such that
\begin{equation}\label{eq_blowupvect}
\begin{split}
	(w)_{z,s_j}\underset{\C^0_\loc(\Rn)}{\cvg} & \beta (x\cdot\nu)_+,\\
	(u_i)_{z,s_j}\underset{\C^0_\loc(\Rn)}{\cvg} & \beta_i (x\cdot\nu)_+\text{ for any }k\leq i\leq l
\end{split}
\end{equation}
as $j\to\infty$, and 
\begin{equation}\label{eq_alphabetavect}
\beta^2+T(\Om)^2\delta\sum_{i=k}^{l}\beta_i^2=Q,
\end{equation}
where $Q$ is defined in \eqref{eq:defQ_mult}.
\end{lemma}

\begin{proof}
Since $|\delta|\ll k^{-\left(6+\frac{8}{n}\right)}g_n(k)$ we have $\lambda_{k-1}(\Om)<\lambda_k(\Om)\leq \lambda_l(\Om)<\lambda_{l+1}(\Om)$ thanks to \BBB  Corollary \ref{rem_estimatevect}\EEE. The proof is then completely analogous to the proof of \BBB Lemma \EEE  \ref{lem_blowup}, the only difference lying in the computation of the shape derivative: while each $\lambda_i$ is not necessarily differentiable, the sum $\sum_{i=k}^l\lambda_i$ is, thanks to \cite[Theorem 2.6]{LamLan06}, and we have
\[\left.\frac{d}{dt}\right|_{t=0}\left(\sum_{i=k}^{l}\lambda_i\right)\left(\zeta^t(\Om)\right)=\sum_{i=k}^{l}\int_{\Om}\left[\left(|\nabla u_i|^2-\lambda_k(\Om)u_i^2\right)\nabla\cdot\zeta -2\nabla u_i\cdot D\zeta\cdot\nabla u_i\right],\]
where $(u_i)_{k\leq i\leq l}$ is an orthonormal basis of the eigenspaces associated to $(\lambda_i(\Om))_{k\leq i\leq l}$.
\end{proof}

\subsection{Harnack inequality}\label{ssect:harnack}

Let us start by introducing the space of viscosity solutions relevant to us, which we will be denoted by $\Sol_{m,\delta}(L)$,  $m$ being the multiplicity of the eigenspace associated to $\lambda_k(B)$.

\begin{definition}\label{def_Sol}
Let $L\geq 1$, $\delta\in\R$, $m\in\mathbb{N}^*$. We define $\Sol_{m,\delta}(L)$ to be the set of functions $\left(v,\ov{v}_1,\un{v}_1,\hdots,\ov{v}_m,\un{v}_m\right)\in H^1(B_1, \R_+)^{2m+1}$ such that
\begin{align*}
|\nabla v|&\leq L,\\
|\nabla \ov{v}_i|,|\nabla \un{v}_i|&\leq L|\delta|^\frac{1}{4},\\
0<\frac{\ov{v}_i}{v},\frac{\un{v}_i}{v}&\leq L|\delta|^\frac{1}{4} \ \text{ in }\{v>0\},\\
\frac{1}{L}<\frac{\ov{v}_i}{\un{v}_i}&\leq L \ \text{ in }\{v>0\}
\end{align*}
and for every $z\in\partial \{v>0\}$  with a contact sphere on either side with inward normal vector $\nu$, there exists numbers $\alpha>0, \ov{\alpha}_i>0,\ \un{\alpha}_i>0$
such that $x\mapsto (\alpha(x\cdot\nu)_+,\ov{\alpha}_1(x\cdot\nu)_+,\hdots,\un{\alpha}_m(x\cdot\nu)_+)$ is a blow-up of $(v,\ov{v}_1,\ldots,\un{v}_m)$ at $z$ in the sense of \eqref{eq_blowupvect} and
\begin{align*}
\alpha^2+\sum_{i=1}^{m}\frac{\ov{\alpha}_i^2+\un{\alpha}_i^2}{2}&=1&\text{ if }\delta>0,\\
\alpha^2+\sum_{i=1}^{m}\ov{\alpha}_i\un{\alpha}_i&=1&\text{ if }\delta<0.
\end{align*}
\end{definition}
Note that this last condition may be written as $(\partial_\nu v)^2+\sum_{i=1}^{m}\frac{\BBB (\partial_\nu \ov{v}_i)^2+(\partial_\nu\un{v}_i)^2}{2}=1$ {\BBB when $\delta>0$} (resp. $(\partial_\nu v)^2+\sum_{i=1}^{m}({\BBB \partial_\nu \ov{v}_i})(\partial_\nu\un{v}_i)=1$ when $\delta<0$) on $\partial\{v>0\}$ in the viscosity sense, although the traces of the gradients are not {\BBB assumed to be well-defined } here.
\begin{remark}\label{rem_estalpha}\rm 
According to this definition, $\ov{\alpha}_i$ and $\un{\alpha}_i$ are bounded by $L|\delta|^\frac{1}{4}$, so that $\alpha^2\geq 1- L^2 m|\delta|^\frac{1}{2}$. In particular $\alpha\geq \frac{1}{2}$ when $L^2 m|\delta|^\frac{1}{2}\leq \frac{1}{8}$, which is the hypothesis we will make in order to obtain Harnack inequalities.
\end{remark}

In the next Lemma we link this definition of viscosity solutions to our free boundary problem. Let $\Om$ be a minimizer of \eqref{eq_funcvect} and let $w,u_{k},\hdots,u_l$ be its torsion function and eigenfunctions associated to the eigenvalues $(\lambda_i(\Om))_{k\leq i\leq l}$. Set $m=l-k+1$, and $B_{x,r}$ a ball of $\Rn$ \BBB where $r\in(0,1]$ is arbitrary\EEE. We let
\begin{align*}
v(y)&=\frac{1}{r}Q^{-\frac{1}{2}}\left(1-mT(\Om)^2|\delta|^\frac{1}{2}\right)^\frac{1}{2}w(x+ry),\\
\ov{v}_i(y)&=\frac{1}{r}Q^{-\frac{1}{2}}T(\Om)|\delta|^\frac{1}{4}\left(w+|\delta|^\frac{1}{4}u_{k+i-1}\right)(x+ry),\;\;\;i=1,\hdots,m,\\
\un{v}_i(y)&=\frac{1}{r}Q^{-\frac{1}{2}}T(\Om)|\delta|^\frac{1}{4}\left(w-|\delta|^\frac{1}{4}u_{k+i-1}\right)(x+ry),\;\;\;i=1,\hdots,m.
\end{align*}
\begin{lemma}
There exists $C_n,L_n>0$ such that if $|\delta|\leq C_n k^{-2-\frac{8}{n}}$ then for any \BBB $x\in\R^n$ and $v,\ov{v}_1,\un{v}_1,\hdots,\ov{v}_m,\un{v}_m$ defined as above, it holds\EEE
\[\left(v,\ov{v}_1,\un{v}_1,\hdots,\ov{v}_m,\un{v}_m\right)\in \Sol_{m,\delta}\left(L_n\right).\]
\end{lemma}
\begin{proof}
Since $|u_{i}|\lesssim k^{\frac{2}{n}+\frac{1}{2}}w$ for $k\leq i\leq l$ (by Lemma \ref{lem_prelim}) then for $|\delta|\ll k^{-2-\frac{8}{n}}$ we have $0<\ov{v}_i,\un{v}_i\lesssim |\delta|^\frac{1}{4}v$ and also $1\lesssim \ov{v_i}/\un{v_i}\lesssim 1$. By Proposition \ref{prop_existenceregvect} and \eqref{eq:Qspeed_vect} we have $|\nabla w|\lesssim1$ and $|\delta|^\frac{1}{4}|\nabla u_{i}|\lesssim |\delta|^\frac{1}{4}k^{\frac{2}{n}+\frac{1}{2}}\lesssim 1$ since $|\delta|\ll k^{-2-\frac{8}{n}}$, hence there exists $L=L_n$ verifying the properties of Definition \ref{def_Sol}.  Finally, for any $z\in\partial\{v>0\}$ which has a contact sphere with inward {\BBB normal} vector $\nu$, thanks to Lemma \ref{lem_blowupvect} there exists blow-ups $(\beta (x\cdot\nu)_+,\beta_k (x\cdot\nu)_+,\hdots,\beta_l (x\cdot \nu)_+)$ of $(w,u_k,\ldots,u_l)$ at $z$ such that
\[\beta^2+T(\Om)^2\delta\sum_{i=k}^{l}\beta_i^2=Q\]
which may be rearranged as
\[\left(1-m T(\Om)^2|\delta|^\frac{1}{2}\right)\beta^2+T(\Om)^2|\delta|^\frac{1}{2}\sum_{i=k}^{l}\left(\beta^2+\text{sign}(\delta)|\delta|^\frac{1}{2}\beta_i^2\right)=Q.\]
Letting\begin{align*}
\alpha&=Q^{-\frac{1}{2}}\left(1-mT(\Om)^2|\delta|^\frac{1}{2}\right)^\frac{1}{2}\beta,\\
\ov{\alpha}_i&=Q^{-\frac{1}{2}}T(\Om)|\delta|^\frac{1}{4}\left(\beta+|\delta|^\frac{1}{4}\beta_{k+i-1}\right),\;\;\;i=1,\hdots,m,\\
\un{\alpha}_i&=Q^{-\frac{1}{2}}T(\Om)|\delta|^\frac{1}{4}\left(\beta-|\delta|^\frac{1}{4}\beta_{k+i-1}\right),\;\;\;i=1,\hdots,m,
\end{align*}
these correspond to the gradients of the blow-ups of $(v,\ov{v}_1,\hdots,\un{v}_m)$ at $z$, thus concluding the proof.
\end{proof}

As in \cite{MTV21}, \BBB in order to use a viscosity method in the vectorial setting, the key observation \EEE consists in noting that if $(v,\ov{v}_1,\un{v}_1,\hdots,\ov{v}_m,\un{v}_m)\in\Sol_{m,\delta}(L)$, then \[(v,\sqrt{\ov{v}_1\un{v}_1},\hdots,\sqrt{\ov{v}_m\un{v}_m})\]
is a supersolution of a vectorial problem of the type of \cite{KL18} because (see \cite[Lemma 2.9 and Remark 4.1]{MTV21})
\begin{equation}\label{eq_geomean}
\begin{split}
\Delta \sqrt{\ov{v}_i\un{v}_i}&\leq \frac{1}{2}\left(\sqrt{\frac{\un{v}_i}{\ov{v}_i}}\Delta \ov{v}_i+\sqrt{\frac{\ov{v}_i}{\un{v}_i}}\Delta \un{v}_i\right) \leq \sqrt{L}\frac{(\Delta \ov{v}_i)_++(\Delta \un{v}_i)_+}{2},\\
\partial_\nu \sqrt{\ov{v}_i\un{v}_i}&=\sqrt{(\partial_\nu\ov{v}_i)(\partial_\nu\un{v}_i)}\text{ at blow-ups of contact points.}
\end{split}
\end{equation}
Similarly for any positive $(c_i)_{i=1,\hdots, m}$ with $c_i\in [1/(2L),2L]$ we have that
\[\left(v,\frac{1}{2}(c_1\ov{v}_1+c_1^{-1}\un{v}_1),\hdots,\frac{1}{2}(c_m\ov{v}_m+c_m^{-1}\un{v}_m)\right)\] is a subsolution in the sense that\BBB

\begin{equation}\label{eq_armean}
\begin{split}
\Delta \frac{c_i\ov{v}_i+c_i^{-1}\un{v}_i}{2}&\geq -\frac{c_i\Delta\ov{v_i}+c_i^{-1}\Delta\un{v_i}}{2},\\
\partial_\nu \left(\frac{c_i\ov{v}_i+c_i^{-1}\un{v}_i}{2}\right)&\geq \sqrt{(\partial_\nu\ov{v}_i)(\partial_\nu\un{v}_i)}\text{ at blow-ups of contact points.}
\end{split}
\end{equation}
\EEE

\BBB
We will now prove an $\eps$-regularity result (see Corollary \ref{prop_flatness} below), following the general compactness strategy of \cite{DS11} (inspired from \cite{S07}): our goal is to prove that any sufficiently flat (\textit{i.e.} close to affine) solution becomes flatter on a smaller ball, as is stated in Lemma \ref{prop_flatness_improv}. The first step (Lemma \ref{lemma_estharnack}) is a weaker improvement, which is related to Harnack's inequality in the classical \BBB elliptic \BBB setting, regarding the total oscillation of a solution. This implies an equicontinuity property (see Lemma \ref{prop_harnack}) which then allows us to use a compactness argument to prove the flatness improvement result.

Let us start by defining a notion of flatness for solutions.\EEE
\begin{definition}\label{def_flat}
We say $(v,\ov{v}_1,\un{v}_1,\hdots,\ov{v}_m,\un{v}_m)\in\Sol_{m,\delta}(L)$ is $\eps$-flat with parameters 
\[a,b,(\alpha,\ov{\alpha}_1,\un{\alpha}_1,\hdots,\ov{\alpha}_m,\un{\alpha}_m)\]
when $|a|,|b|\leq\eps$ and
\begin{equation}\label{eq_defflatness}
\begin{split}
	0\leq b-a&\leq \eps\\
	\alpha^2+\sum_{i=1}^{m}\frac{\ov{\alpha}_i^2+\un{\alpha}_i^2}{2}= 1&\text{ if }\delta>0,\qquad
	\alpha^2+\sum_{i=1}^{m}\ov{\alpha}_i\un{\alpha}_i= 1\text{ if }\delta<0, \\
	(x_n+a)_+\leq 
	\frac{v(x)}{\alpha},\ &\frac{\ov{v}_i(x)}{\ov{\alpha}_i},\ \frac{\un{v}_i(x)}{\un{\alpha}_i} \leq  (x_n+b)_+\text{ in }B_1,\\
	\frac{|\Delta v|}{\alpha},\ \frac{|\Delta\ov{v}_i|}{\ov{\alpha}_i},\ \frac{|\Delta\un{v}_i|}{\un{\alpha}_i}&< \eps^2\text{ in }B_1 \cap \{v>0\}.\\
\end{split}
\end{equation}
\end{definition}
We remark that the second and third equations of \eqref{eq_defflatness} (evaluated at $x\to e_n$) directly imply, for a small enough $\eps$,
\begin{equation}\label{eq_estapriorialpha_i}
\frac{\ov{\alpha}_i}{\un{\alpha}_i}\in \left[\frac{1}{2L},2L\right]\text{ and } \ov{\alpha}_i,\un{\alpha}_i\leq 2L|\delta|^\frac{1}{4}.
\end{equation}

In all the following we let $\eta\in\C_c^\infty(\R,[0,1])$ such that $\eta\equiv 1$ on $[-3/5,3/5]$ and $\eta=0$ outside of $[-4/5,4/5]$. Then for any small enough $t$ (positive or negative) we set
\[H_t:=\{(x',x_n)\in \Rn:x_n> -t\eta(|x'|)\}.\]

We define three functions depending on $t$:
\[\begin{cases}
\Delta \varphi_t =0 & {\BBB \textrm{in } }B\cap H_t\\
\varphi_t = x_n & {\BBB \textrm{on } }\partial B\cap H_t\\
\varphi_t = 0 & {\BBB \textrm{on } }B\cap \partial  H_t\\
\end{cases},\ \begin{cases}
\Delta \psi_t =0 & {\BBB \textrm{in } }(B\cap H_t)\setminus B\left(\frac{1}{2}e_n,\frac{1}{4}\right)\\
\psi_t=1 & {\BBB \textrm{in } }B\left(\frac{1}{2}e_n,\frac{1}{4}\right)\\
\psi_t = 0 & {\BBB \textrm{on } }\partial(B\cap H_t)
\end{cases},\ \begin{cases}
\Delta \zeta_t =2 & {\BBB \textrm{in } }B\cap H_t\\
\zeta_t =-\chi  & {\BBB \textrm{on } }\partial(B\cap H_t)
\end{cases}\]
where $\chi\in\C^{\infty}(B_1,[0,1])$ is such that $\chi>0$ on $B_1\setminus \ov{B}_{\frac{9}{10}}$ and $\chi\equiv0$ on $B_{\frac{9}{10}}$. Note that $\zeta_t\leq 0$ and $\psi_t\geq0$ for every $t$.

\begin{lemma}\label{lemma_estharnack}
There exists $t_n\in (0,1)$, $c_n,d_n>0$ such that for any $t\in (-t_n,t_n)$,
\[\partial_n\psi_t \geq c_n,\quad | \nabla\varphi_t -e_n|\leq d_n t, \quad |\zeta_t|\leq d_n,\quad |\nabla  \zeta_t|\leq d_n, \text{ on }\partial H_t\cap\left\{|x'|\leq \frac{9}{10}\right\}\]
and
\begin{align*}
\text{if }t>0,\ \varphi_t&\geq \left(x_n+\frac{1}{32}t\right)_+ \text{ on }B_{\frac{1}{8\sqrt{n}}}\cap H_t,\\
\text{if }t<0,\ \varphi_t&\leq \left(x_n+\frac{1}{ 32}t\right)_+ \text{ on }B_{\frac{1}{8\sqrt{n}}}\cap H_t.
\end{align*}

\end{lemma}
\begin{proof}
For the \BBB four \EEE first estimates, we only explain how the estimate of $\partial_n\psi_t$ is obtained, as the three others are derived analogously. We have first that $\partial_n\psi_0\geq c>0$ for some $c>0$ by Hopf Lemma, while on the other hand by elliptic estimates $\|\partial_n(\psi_t\circ T_t-\psi_0)\|_{L^\infty(H_0\cap B_{\frac{9}{10}})}\lesssim t$ where $T_t$ is a diffeomorphism sending  $B\cap H_t$ over $B\cap H_0=B\cap\Hn$, thus giving $\partial_n \psi_t\geq c_n$ over $\partial H_t\cap\left\{|x'|\leq \frac{9}{10}\right\}$ for any $|t|\leq t_n$ for some dimensional $c_n>0$ and $t_n>0$.

For the second point, we consider for $0 \leq t\leq\frac{1}{4\sqrt{n}}$ 
\[ P_t(x):=x_n+4t\left(n\left(x_n-\frac{1}{4\sqrt{n}}\right)^2-|x'|^2\right).\]
We check that $ P_t(x)\leq x_n^+\leq \varphi_t$ on $\partial  \left(B^{n-1}_\frac{1}{2}\times\left[-t,\frac{1}{4\sqrt{n}}\right]\right)$, while $\Delta P_t=8t$, so by maximum principle we have $ P_t\leq \varphi_t$ on $B^{n-1}_\frac{1}{2}\times\left[-t,\frac{1}{4\sqrt{n}}\right]$. Since we have $ P_t(x)\geq x_n+\frac{1}{32}t$ on $B_{\frac{1}{8\sqrt{n}}}\cap H_t$ we deduce the claim in this case.

\BBB The case $t<0$ is treated similarly: this time we have $P_t\geq x_n^+\geq\varphi_t$ by maximum principle and $P_t\leq x_n+\frac{1}{32}t$ on $B_{\frac{1}{8\sqrt{n}}}\cap H_t$.\EEE

\end{proof}

\begin{proposition}\label{cor_harnack}
Let $L\geq 1$, $\delta\in\R$ and $m\in\mathbb{N}^*$ be such that $L^2m|\delta|^\frac{1}{2}\leq \frac{1}{8}$. Then there exists $c_n>0$ such that for any $\eps\ll L^{- \frac{1}{2}}$ and any $(v,\ov{v}_1,\un{v}_1,\hdots,\ov{v}_m,\un{v}_m)\in\Sol_{m,\delta}(L)$ that is $\eps$-flat in the sense of Definition \ref{def_flat} with parameters
\[a,b,(\alpha,\ov{\alpha}_1,\un{\alpha}_1,\hdots,\ov{\alpha}_m,\un{\alpha}_m),\]
then there exists $a',b'$ such that $a\leq a'\leq b'\leq b$, $b'-a'\leq (1-c_n)\eps$ and
\[(x_n+a')_+\leq \frac{v(x)}{\alpha},\ \frac{\ov{v}_i(x)}{\ov{\alpha}_i},\ \frac{\un{v}_i(x)}{\un{\alpha}_i} \leq (x_n+b')_+\text{ on }B_{\frac{1}{8\sqrt{n}}}.\]
\end{proposition}
\begin{remark}\rm 
The hypothesis $L^2m|\delta|^\frac{1}{2}\leq \frac{1}{8}$ may be replaced by $L^2m|\delta|^\frac{1}{2}\leq 1-\eta$ for any $\eta>0$, but how small $\eps$ needs to be would depend on $\eta$.
\end{remark}
\begin{proof}
First note that by the estimates \eqref{eq_estapriorialpha_i} and the hypothesis $L^2m|\delta|^\frac{1}{2}\leq\frac{1}{8}$ we get $\alpha\geq\frac{1}{2}$.\newline

We suppose without loss of generality that $b-a\geq\frac{1}{2}\eps$ otherwise we are done. As a consequence we have either $\frac{v\left(\frac{1}{2}e_n\right)}{\alpha}\geq \frac{1}{2}+a+\frac{\eps}{4}$ (Case A) or $\frac{v\left(\frac{1}{2}e_n\right)}{\alpha}\leq \frac{1}{2}+b-\frac{\eps}{4}$ (Case B).

\smallskip
\noindent
\textbf{Case A.} Without loss of generality, we can set $a=0$, meaning $\frac{v\left(\frac{1}{2}e_n\right)}{\alpha}\geq\frac{1}{2}+\frac{\eps}{4}$.

Assume first $\delta<0$. Since $|\Delta v|<\alpha\eps^2$  in $B_1 \cap \{v>0\}$, then for a small enough $\eps$ we get by the usual Harnack inequality applied to the positive function $v-\alpha x_n$ the existence of some $\sigma_n\in (0,1/4)$ such that
\begin{equation}\label{eq:Harnack_classi_v}\frac{v(x)}{\alpha}\geq x_n+ 2\sigma_n \eps\text{ on }B\left(\frac{1}{2}e_n,\frac{1}{4}\right).\end{equation}
We now consider the set of $t\geq 0$ such that the following $m+1$ inequalities are all verified on $B\cap H_t$:
\begin{equation}\label{eq_viscosity}
\begin{split}
	\frac{v}{\alpha}&\geq \varphi_t+\eps\sigma_n \psi_t +\eps^2\zeta_t,\\
	\sqrt{\frac{\ov{v}_i}{\ov{\alpha}_i}\cdot \frac{\un{v}_i}{\un{\alpha}_i}}&\geq \varphi_t +L^\frac{1}{2}\eps^2\zeta_t.
\end{split}
\end{equation}
Note first that for each $t\geq0$, 
\begin{align*}\frac{\Delta v}{\alpha}< \eps^2=\Delta(\varphi_t+\eps\sigma_n \psi_t +\eps^2\zeta_t) &\text{ in }B\cap H_t\setminus \ov{B}(\frac{1}{2}e_n,\frac{1}{4}),\\
\Delta \sqrt{\frac{\ov{v}_i}{\ov{\alpha}_i}\cdot \frac{\un{v}_i}{\un{\alpha}_i}}\leq L^\frac{1}{2}\eps^2<\Delta\left(\varphi_t +L^\frac{1}{2}\eps^2\zeta_t\right)&\text{ in }B\cap H_t,
\end{align*}
where we used \eqref{eq_geomean} in the second series of inequalities. 

By \eqref{eq:Harnack_classi_v}, it holds $v(x)/\alpha\geq x_n+\sigma_n\eps\psi_0(x)$. As a consequence the \BBB first inequality in \eqref{eq_viscosity} is verified at $t=0$ by maximum principle inside $B\cap H_0\setminus B(\frac{1}{2}e_n,\frac{1}{4})$, and the second inequality of \eqref{eq_viscosity} is verified at $t=0$ by maximum principle inside $B\cap H_0$ (using also $v(x)/\alpha, \ov{v}_i/\ov{\alpha}_i,\un{v}_i/\un{\alpha}_i\geq x_n$ in $B_1$ and noting that $\varphi_0=x_n$ and $\zeta_t\leq 0$)\EEE.
We can therefore consider the largest $t\geq0$ such that the inequalities \eqref{eq_viscosity} are verified in $B\cap H_t$. We want to prove that $t\geq \vartheta_n \eps$ for some dimensional $\vartheta_n>0$. Note that we lose no generality in supposing that $t$ is at most comparable to $\eps$ (meaning $t\ll\eps$), since the claim holds otherwise.\newline

At the maximal $t$ there is a equality in one of the inequalities \eqref{eq_viscosity} at some point $x\in \ov{B\cap H_t}$. Let us consider the possible cases. 
\begin{itemize}[label=\textbullet]
\item We cannot have $x\in \partial (H_t\cap B)\setminus \ov{B_{\frac{9}{10}}}$, since $\zeta_t<0$, $\psi_t=0$ and $\varphi_t=x_n$ over this set.
\item Suppose that $x\in B\cap H_t$. Then let us show that in this case we must have $v(x)>0$. Otherwise, we would have $v(x)=\ov{v}_i(x)=\un{v}_i(x)=0$ so that necessarily $x_n\leq0$. But $\varphi_t+L^{\frac{1}{2}}\eps^2\zeta_t>0$ (and likewise $\varphi_t+\eps\sigma_n \psi_t +\eps^2\zeta_t>0$) over $B\cap H_t\cap\{x_n\leq0\}$, which comes from $\varphi_t+L^{\frac{1}{2}}\eps^2\zeta_t=\varphi_t-L^{\frac{1}{2}}\eps^2\chi=0$ over $B\cap \partial H_t\cap\{x_n\leq0\}$ and $|\nabla(\varphi_t+L^{\frac{1}{2}}\eps^2\zeta_t)-e_n|\ll1$ thanks to Lemma \ref{lemma_estharnack}. As a consequence $v(x)>0$ and we can apply the maximum principle inside $B\cap H_t\cap\{v>0\}$ to get that the equality cannot happen \BBB for the second inequality of \eqref{eq_viscosity}\EEE. On the other hand, equality cannot happen for $v$ by maximum principle inside $\{v>0\}\cap B\cap H_t\setminus \ov{B}(\frac{1}{2}e_n,\frac{1}{4})$ and since in $ \ov{B}(\frac{1}{2}e_n,\frac{1}{4})$ one has for $t\ll \eps$:

\[\varphi_t+\eps\sigma_n \psi_t +\eps^2\zeta_t\leq \varphi_t+\eps \sigma_n\leq x_n+C_n t+\eps\sigma_n<x_n+2\sigma_n\eps.\]

\end{itemize}

As a consequence, $x\in\partial H_t\cap \ov{B_{\frac{9}{10}}}$. \BBB Since \EEE $(\varphi_t+\eps\sigma_n \psi_t +\eps^2\zeta_t)(x)=(\varphi_t +L^\frac{1}{2}\eps^2\zeta_t)(x)=0$ \BBB then \EEE $v(x)=\ov{v}_i(x)=\un{v}_i(x)=0$ and there is  equality in all the inequalities \eqref{eq_viscosity}.
Since on the other hand one has $B\cap H_t\subset \{v>0\}$, hence at any interior contact sphere for $B\cap H_t$ at $x$ there exists a blow-up of $(v,\ov{v}_1,\hdots,\un{v}_m)$ \BBB of the form $z\mapsto(\beta z\cdot\nu,\ov{\beta}_1 z\cdot\nu,\un{\beta}_1 z\cdot\nu,\hdots,\un{\beta}_m z\cdot\nu)$ \EEE as in Definition \ref{def_Sol}. 
As a consequence we have the viscosity condition:

\begin{align*}
1&=\beta^2+\sum_{i=1}^{m}\ov{\beta}_i\un{\beta}_i\geq \alpha^2|\nabla(\varphi_t+\eps \sigma_n \psi_t +\eps^2\zeta_t)(x)|^2+\sum_{i=1}^{m}\ov{\alpha}_i\un{\alpha}_i|\nabla(\varphi_t+L^\frac{1}{2}\eps^2\zeta_t)(x)|^2\\
&\geq \alpha^2\left[(\partial_n\varphi_t)(x)^2+2\partial_n\varphi_t(x)\partial_n\psi_t(x)\sigma_n\eps-C_n\eps^2\right]+\sum_{i=1}^{m}\ov{\alpha}_i\un{\alpha}_i\left[(\partial_n\varphi_t(x))^2-C_nL^\frac{1}{2}\eps^2\right]\\
&\hspace{1cm}\text{ for some large enough dimensional constant }C_n>0\\
&\geq \partial_n\varphi_t(x)^2+2\alpha^2\partial_n\varphi_t(x)\partial_n\psi_t(x)\sigma_n\eps-C_nL^\frac{1}{2}\eps^2\\
&\geq 1-2d_nt+\frac{1}{2}c_n\sigma_n\eps-C_nL^\frac{1}{2}\eps^2\text{ since }\alpha\geq \frac{1}{2}
\end{align*}
where $c_n,d_n>0$ come from Lemma \ref{lemma_estharnack}.
So when $\eps\ll L^{-\frac{1}{2}}$ we get
\[t\geq \frac{c_n\sigma_n}{8d_n}\eps =:\vartheta_n \eps\]
so that using Lemma \ref{lemma_estharnack} we find for any $y\in B_{\frac{1}{8\sqrt{n}}} \cap H_t$
\[\frac{v(y)}{\alpha},\frac{\ov{v}_i(y)}{\ov{\alpha}_i},\frac{\un{v}_i(y)}{\un{\alpha}_i}\geq \varphi_{t}(y)+\sqrt{L}\eps^2\zeta_{t}(y)\geq y_n+\frac{1}{32}\vartheta_n \eps-\sqrt{L}\Vert\zeta_{t}\Vert_\infty\eps^2\geq y_n+\frac{1}{64}\vartheta_n \eps\]
for $\eps\ll1$. For $y\in B_{\frac{1}{8\sqrt{n}}} \setminus H_t$ the above inequalities always hold, since $y_n+\frac{1}{64}\vartheta_n \eps\leq0$ and the functions are non-negative. This finishes the proof when $\delta<0$.

The case $\delta>0$ follows the same strategy, and was proven in \cite[Theorem 5.1]{KL18}{\BBB, though we notice in  addition that one can keep} track of the constants. \BBB Roughly speaking, in this case \BBB we find instead the viscosity condition

\begin{align*}
1&=\beta^2+\sum_{i=1}^{m}\frac{\ov{\beta}_i^2+\un{\beta}_i^2}{2}\geq \alpha^2|\nabla(\varphi_t+\eps \sigma_n \psi_t +\eps^2\zeta_t)(x)|^2+\sum_{i=1}^{m}\ov{\alpha}_i\un{\alpha}_i|\nabla(\varphi_t+L^\frac{1}{2}\eps^2\zeta_t)(x)|^2\\
&\geq 1-2d_nt+\frac{1}{2}c_n\sigma_n\eps-C_nL^\frac{1}{2}\eps^2\text{ since }\alpha\geq \frac{1}{2}
\end{align*}
and conclude in the same way.
\EEE

\smallskip
\noindent
\textbf{Case B.} We suppose without loss of generality that $b=0$, meaning $\frac{v\left(\frac{1}{2}e_n\right)}{\alpha}\leq \frac{1}{2}-\frac{\eps}{4}$. The proof here follows the same outline, \BBB thus we only \EEE give rough details.

First, the Harnack inequality applied to $\alpha x_n-v$ gives the existence of some $\sigma_n>0$ such that $\frac{v(x)}{\alpha}\leq x_n- 2\sigma_n \eps$ on $B\left(\frac{1}{2}e_n,\frac{1}{4}\right)$. We now consider the largest $t>0$ such that all the following inequalities are verified in $B\cap H_{-t}$:
\begin{align*}
\frac{v}{\alpha}&\leq \varphi_{-t}-\eps\sigma_n \psi_{-t} -\eps^2\zeta_{-t},\\
\frac{1}{2}\left(\sqrt{\frac{\un{\alpha}_i}{\ov{\alpha}_i}}\ov{v}_i+ \sqrt{\frac{\ov{\alpha}_i}{\un{\alpha}_i}}\un{v}_i\right)&\leq \varphi_{-t} -(2L)^\frac{1}{2}\eps^2\zeta_{-t}.
\end{align*}
It is verified at $t=0$ by the previous remark and the maximum principle. We then identify a contact point $x$ associated to the largest $t$ that we suppose small compared to $\eps$: it is not inside $B\cap H_{-t}$ by maximum principle since
\begin{align*}\frac{\Delta v}{\alpha}>\Delta(\varphi_{-t}-\eps\sigma_n \psi_{-t} +\eps^2\zeta_{-t}) &\text{ in }B\cap H_{-t}\setminus \ov{B}(\frac{1}{2}e_n,\frac{1}{4})\\
\Delta \frac{1}{2}\left(\sqrt{\frac{\un{\alpha}_i}{\ov{\alpha}_i}}\ov{v}_i+ \sqrt{\frac{\ov{\alpha}_i}{\un{\alpha}_i}}\un{v}_i\right)>\Delta\left(\varphi_{-t} -(2L)^\frac{1}{2}\eps^2\zeta_{-t}\right)&\text{ in }B\cap H_{-t},
\end{align*}
\BBB where we used the estimates \eqref{eq_estapriorialpha_i} in the last line. \EEE The contact point is not in $\partial(B_1\cap H_{-t})\setminus\ov{B_{\frac{9}{10}}}$ for the same reason as before, thus giving that it lies in $\partial H_{-t}\cap \ov{B_{\frac{9}{10}}}$. We then use the boundary condition (understood in the viscosity sense at the contact point)
\[1\leq (\partial_\nu v)^2+\sum_{i=1}^{m}\left(\partial_\nu \frac{1}{2}\left(\sqrt{\frac{\un{\alpha}_i}{\ov{\alpha}_i}}\ov{v}_i+ \sqrt{\frac{\ov{\alpha}_i}{\un{\alpha}_i}}\un{v}_i\right)\right)^2\leq 1+a_n t-b_n\eps\]
for some $a_n,b_n>0$. This yields $t\geq \vartheta_n\eps$ for some $\vartheta_n>0$ and we conclude by the last property of Lemma \ref{lemma_estharnack}. The case $\delta>0$ is similar.
\end{proof}

\begin{proposition}\label{prop_harnack}
Let $L\geq 1$, $\delta\in\R$ and $m\in\mathbb{N}^*$ be such that $L^2m|\delta|^\frac{1}{2}\leq \frac{1}{8}$. Then there exists $C_n>0$ and $\kappa_n\in (0,1)$, such that the following holds: for any $\eps\ll L^{-\frac{1}{2}}$, and for any $(v,\ov{v}_1,\un{v}_1,\hdots,\ov{v}_m,\un{v}_m)\in\mathcal{S}_{m,\delta}(L)$ that is $\eps$-flat in the sense of definition \ref{def_flat} with parameters
\[a,b,(\alpha,\ov{\alpha}_1,\un{\alpha}_1,\hdots,\ov{\alpha}_m,\un{\alpha}_m),\]
then defining
\[V(x)=\frac{v(x)-{\alpha} x_n}{\alpha \eps},\
\ov{V}_i(x)=\frac{\ov{v}_{i}(x)-\ov{\alpha}_{i} x_n}{\ov{\alpha}_{i} \eps},\
\un{V}_i(x)=\frac{\un{v}_{i}(x)-\un{\alpha}_{i} x_n}{\un{\alpha}_{i} \eps},\]
\[\ W_i^\eps(x)=\frac{\sqrt{\ov{v}_i\un{v}_i}-\sqrt{\ov{\alpha}_i\un{\alpha}_i}x_n}{\sqrt{\ov{\alpha}_i\un{\alpha}_i}\eps}\]
we have that for any $x,y\in B_{\frac{1}{2}}\cap\{v>0\}$ such that $|x-y|>\eps$,
\begin{align}
\label{eq:V_iepsHold}
\left|V(x)-V(y)\right|,\;\;\left|\ov{V}_i(x)-\ov{V}_i(y)\right|,\;\;\left|\ov{V}_i(x)-\ov{V}_i(y)\right|,\;\;\left|W_i(x)-W_i(y)\right|\leq C_n|x-y|^{\kappa_n}.
\end{align}
\end{proposition}

\begin{proof}
This is obtained by applying successively the previous Lemma, as in \cite[Corollary 3.2]{DS11} or \cite[Lemma 7.14]{V19}.
\end{proof}

\subsection{Flatness improvement}\label{ssect:flatness}
\BBB We start by stating a general result on sequences in compact metric spaces.
\begin{lemma}\label{lemma_compact}
Let $(X,d)$ be a nonempty compact metric space. Let $m_k$ be a sequence of integers such that $m_k\to\infty$ and $(x_{j}^{k})_{k\in\N^*,1\leq j\leq m_k}$ be a sequence in $X$. Then there exists a sequence of permutations $\sigma^k\in\mathfrak{S}(\llbracket 1,m_k\rrbracket)$ and a sequence $( x_{j})_{j\in \N^*}$ such that
\[\liminf_{k\to\infty}\sup_{1\leq j\leq m_k}d\left(x_{\sigma^k(j)}^k,x_j\right)=0.\]
\end{lemma}

We do not claim that this lemma is original, but since we have not found any reference in the literature we provide a short proof.
\begin{proof}
Note first that it is enough to prove the Lemma for the Cantor set $X=\{0,1\}^{\N^*}$ endowed with the dyadic metric $d(x,y)=2^{-\inf\{i\geq 1:x(i)\neq y(i)\}}$, as it surjects continuously onto any compact metric space. We write $X_N=\{0,1\}^{N}$ and $\pi_N:X\to X_N$ the projection onto the first $N$ coordinates. Let $\varphi_1:\N^*\to\N^*$ be an extraction such that the number of $0$'s and $1$'s among \[\left(\pi_1 x^{\varphi_1(k)}_{1},\hdots,\pi_1 x^{\varphi_1(k)}_{m_{\varphi_1(k)}}\right)\]
is nondecreasing in $k$. Starting from $\varphi_1$ we define recursively $\varphi_N$ in the following way: if $\varphi_{N-1}$  is given we build $\varphi_N$ as an extraction of $\varphi_{N-1}$ to guarantee that the number of occurrences of each $b\in X_N$ in
\[\left(\pi_N x^{\varphi_N(k)}_{1},\hdots,\pi_N x^{\varphi_N(k)}_{m_{\varphi_N(k)}}\right)\]
is nondecreasing in $k$. We finally set $\varphi(k):=\varphi_k(k)$. We now define a sequence of permutations $\sigma^k\in\mathfrak{S}_{m_{\varphi(k)}}$ as follows: we let $\sigma^1$ be the identity, and provided $\sigma^{k}$, we define $\sigma^{k+1}$ recursively. Since in the list
\[\left(\pi_{k} x^{\varphi(k+1)}_{1},\hdots,\pi_{k} x^{\varphi(k+1)}_{m_{\varphi(k+1)}}\right)\]
there are at least as many occurrences of each element of $X_k$ as in the list
\[\left(\pi_{k} x^{\varphi(k)}_{1},\hdots,\pi_{k} x^{\varphi(k)}_{m_{\varphi(k)}}\right).\]
Then we may fix $\sigma^{k+1}\in\mathfrak{S}_{m_{\varphi(k+1)}}$ such that for each $j\in \{1,\hdots, m_{\varphi(k)}\}$, we have
\[\pi_kx^{\varphi(k+1)}_{\sigma^{k+1}(j)}=\pi_k x^{\varphi(k)}_{\sigma^k(j)}.\]
We now define $x_j$ as the unique element of $X$ such that for every $k\in \N^*$, as soon as $m_{\varphi(k)}\geq j$ we have
\[\pi_kx_j=\pi_kx_{\sigma^k(j)}^{\varphi(k)}.\]
By construction this gives for every $k\in\N^*$, $j\in \{1,\hdots,m_{\varphi(k)}\}$:
\[
d\left(x_{\sigma^k(j)}^{\varphi(k)},x_j\right)\leq\frac{1}{2^{k+1}}.\]
Hence $(x_{\sigma^k(j)}^{\varphi(k)})\underset{k\to\infty}{\longrightarrow}(x_j)$ thus concluding the proof.
\end{proof}

\EEE

\BBB Coming back to the flatness improvement claim, we \EEE let $\tau=\tau_n\in (0,1)$ be a fixed constant depending only on $n$, such that for any harmonic function $h:B_1\to [-5,5]$ and for any $x\in B_\tau$ it holds
\begin{equation}\label{eq:harmonic_2grad_est}\left|h(x)-h(0)-x\cdot\nabla h(0)\right|\leq \frac{1}{8}\tau.\end{equation}
{\BBB This constant is used in the statement of the next result.}

\begin{proposition}\label{prop_flatness_improv}
Let $L\geq 1$, $\delta\in\R$ and $m\in\mathbb{N}^*$ with $L^2m^2|\delta|^\frac{1}{2}\leq \frac{1}{8}$. Then there exists $\eps_{n}(L)$ such that we have the following flatness reduction property for any $\eps<\eps_n(L)$. Suppose  $(v,\ov{v}_1,\un{v}_1,\hdots,\ov{v}_m,\un{v}_m)\in\mathcal{S}_{m,\delta}(L)$ is $\eps$-flat in the sense of Definition \ref{def_flat} with parameters
\[a,b,(\alpha,\ov{\alpha}_1,\un{\alpha}_1,\hdots,\ov{\alpha}_m,\un{\alpha}_m).\]
Then there exists $a'\leq b'$,  $e'\in\mathbb{S}^{n-1}$, and $\alpha',\ov{\alpha}_i',\un{\alpha}_i'$ verifying
\[{\alpha'}^2+\sum_{i=1}^{m}\frac{{\ov{\alpha}_i'}^2+{\un{\alpha}_i'}^2}{2}= 1\text{ if }\delta>0,\;\;\;
{\alpha'}^2+\sum_{i=1}^{m}\ov{\alpha}_i'\un{\alpha}_i'= 1\text{ if }\delta<0\]
such that
\[\left(e'\cdot x+a'\right)\leq\frac{v(\tau x)}{\tau \alpha'},\ \frac{\ov{v}_i(\tau x)}{\tau \ov{\alpha}_i'},\ \frac{\un{v}_i(\tau x)}{\tau \un{\alpha}_i'} \leq \left(e'\cdot x+b'\right)\text{ on }B_{1}\cap \frac{\{v>0\}}{\tau},\]
and $b'-a'\leq \frac{1}{2}\eps$, with moreover
\[|e'-e_n|,\ \left|1-\frac{\alpha'}{\alpha}\right|,\ \left|1-\frac{\ov{\alpha'}_i}{\ov{\alpha}_i}\right|,\ \left|1-\frac{\un{\alpha'}_i}{\un{\alpha}_i}\right|\lesssim \eps.\]
\end{proposition}

\begin{proof}{\BBB We follow the ideas of the proof of \cite[Theorem 6.1]{KL18}, with a different treatment when the multiplicity of the eigenspace goes to infinity (see the case $m_\eps\to\infty$ below):
we will} proceed by contradiction and compactness. Suppose there exists a sequence $\eps_p\to 0$ (we drop the index $p$ and just write $\eps\to 0$ to lighten the notations) and some sequences \[(v^\eps,\ov{v}_{1}^\eps,\un{v}_{1}^\eps,\hdots,\ov{v}_{m^\eps}^\eps,\un{v}_{m^\eps}^\eps)\in\Sol_{m^\eps,\delta^\eps}(L),\ a^\eps, b^\eps,(\alpha^\eps,\ov{\alpha}_1^\eps,\un{\alpha}_1^\eps,\hdots,\ov{\alpha}_{m^\eps}^\eps,\un{\alpha}_{m^\eps}^\eps)\]
which verify the hypotheses but not the conclusion.  This means that at least one of the functions  $v^\eps$, $\ov{v}_1^\eps$, $\un{v}_1^\eps$, $\hdots$, $\ov{v}_{m^\eps}^\eps$, $\un{v}_{m^\eps}^\eps$ does not verify the flatness improvement on $B_{\tau}$.

Consider the sequences
\[V^\eps(x)=\frac{v^\eps(x)-\alpha^\eps x_n}{\alpha^\eps \eps},\
\ov{V}_i^\eps(x)=\frac{\ov{v}_{i}^\eps(x)-\ov{\alpha}_{i}^\eps x_n}{\ov{\alpha}_{i}^\eps \eps},\
\un{V}_i^\eps(x)=\frac{\un{v}_{i}^\eps(x)-\un{\alpha}_{i}^\eps x_n}{\un{\alpha}_{i}^\eps \eps},\]
\[ W_i(x)=\frac{\sqrt{\ov{v}_i^\eps\un{v}_i^\eps}-\sqrt{\ov{\alpha}_i^\eps\un{\alpha}_i^\eps}x_n}{\sqrt{\ov{\alpha}_i^\eps\un{\alpha}_i^\eps}\eps}.\]
We also write $\Om^\eps=B_1\cap\{v^\eps>0\}$ their (common) domain of definition, which converges locally Hausdorff to $B_1\cap\Hn$ since $\{x\in B_1, x_n>\eps\}\subset\Om^\eps\subset \{x\in B_1, x_n>-\eps\}$.
Each function has values in $[-1,1]$, with \BBB Laplacian \EEE bounded by $\eps$ in $\Om^\eps$. Moreover thanks to Proposition \ref{cor_harnack}, they verify the Hölder-type property \eqref{eq:V_iepsHold} for some $\kappa_n\in(0,1)$ up to the boundary $\partial\Hn$.

After extraction in $\eps$ we have a local Hausdorff convergence of the graphs of $V^\eps$, $\ov{V}_i^\eps$, $\un{V}^\eps_i$ on $\Om^\eps$ to the graphs of functions $V,\ov{V}_i,\un{V}_i:B\cap\Hn\to [-1,1]$, which are in $\C^{0,\kappa_n}_\loc(B\cap\ov{\Hn})$ and harmonic in $B\cap\Hn$ (see for instance \cite[Lemma 7.15]{V19}). The functions $(W_i^\eps)_{\eps\to 0}$ verify the same oscillation reduction so after extraction their graphs converge in the local Hausdorff sense to a limit $W_i$, which we identify (by taking a limit for any $x\in B\cap\{x_n>\delta\}$) as
\[W_i=\frac{\ov{V}_i+\un{V}_i}{2}.\]

We now distinguish four cases depending on whether $m^\eps$ is stationary at some finite value $m\in\N^*$ or not, and whether $\delta>0$ or $\delta<0$: we detail the cases $\delta<0$ and outline the cases $\delta>0$, {\BBB for which more details may be found (though without focus on the }control of the constants in $m$) in \cite{KL18}.\bigbreak

\smallskip
\noindent
\textbf{Case $m^\eps\to m$, $\delta<0$.} We lose no generality in assuming $m^\eps=m$ for all $\eps$. Up to extraction there exists $\alpha,\ov{\alpha}_i,\un{\alpha}_i\geq0$ such that
\[\alpha^\eps\to\alpha,\ \ov{\alpha}_i^\eps\to \ov{\alpha}_i,\ \un{\alpha}_i^\eps\to \un{\alpha}_i,\]
$V-\ov{V}_i$ and $V-\un{V}_i$ verify a Dirichlet boundary condition on $B_1\cap\{x_n=0\}$ (since $V^\eps-\un{V}_i^\eps=V^\eps-\ov{V}_i^\eps=0$ on $\partial\Om^\eps$). This makes $2m$ Dirichlet boundary conditions for $2m+1$ harmonic functions, and we claim that we have an additional boundary condition
\begin{equation}\label{eq:Neum_visc}\partial_{n} h=0\text{ in }B_1\cap\{x_n=0\}, \mbox{ where } h= \left(\alpha^2V+\sum_{i=1}^{m}\ov{\alpha}_i\un{\alpha}_i\frac{\ov{V}_i+\un{V}_i}{2}\right),\end{equation}
holding in the viscosity sense.\newline

\smallskip
\noindent
\BBB \textit{Inequality $\partial_nh\leq 0$.} \EEE To prove this claim set $x^0\in B_1\cap\{x_n=0\}$, and we suppose by contradiction that there are constants $\BBB \overline p \EEE>0$, $z\in\R^{n-1}\times\{0\}$, $\sigma>0$ such that
\begin{equation}\label{eq:polynom_h}h(x)\geq h(x^0)+\BBB \overline p \EEE x_n+z\cdot(x-x^0)+\sigma\left(x_n^2-\frac{1}{n+1}\left|x-x^0\right|^2\right)=:\varphi(x),\ \forall x\in B(x^0,\rho)\cap\Hn\end{equation}
Note that we can always change $\BBB \overline p \EEE$ into $\BBB \overline p \EEE/2$, replace $\sigma$ by some arbitrarily large $\sigma'\geq \sigma$ and $\rho$ by some small enough $\rho'<\rho$ such that the equality holds only at $x=x^0$. Since the functions $V-h$, $\ov{V}_i-h$ and $\un{V}_i-h$ are harmonic and vanish on $B_1\cap\{x_n=0\}$, they are smooth over $B_1\cap\ov{\Hn}$  so that there exists $q$, $\ov{q}_i$, $\un{q}_i\in\R$ such that
\begin{align*}
(V-h)(x)&=qx_n+\mathcal{O}\left(|x-x^0|^2\right),\\
(\ov{V}_i-h)(x)&=\ov{q}_ix_n+\mathcal{O}\left(|x-x^0|^2\right),\\
(\un{V}_i-h)(x)&=\un{q}_ix_n+\mathcal{O}\left(|x-x^0|^2\right),\\
(W_i-h)(x)&=\frac{\ov{q}_i+\un{q}_i}{2}x_n+\mathcal{O}\left(|x-x^0|^2\right)=:q_ix_n+\mathcal{O}\left(|x-x^0|^2\right),
\end{align*}
\BBB and \EEE which verify in addition $\alpha^2q+\sum_{i=1}^{m}\ov{\alpha}_i\un{\alpha}_iq_i=0$. Up to reducing $\BBB \overline p \EEE$ and increasing $\sigma$, we have by uniform interior $\C^2$ estimates on the harmonic functions:
\[V(x)\geq q x_n+\varphi(x),\;\;\; W_i(x)\geq q_ix_n+\varphi(x),\]
in a neighbourhood of $x^0$ in $B\cap\Hn$, which we denote by $B_{x^0,\rho}\cap B\cap\Hn$. Then by the local uniform Hausdorff convergence of the graphs there exists $c^\eps\to 0$ such that
\[V^\eps(x)\geq q x_n+\varphi(x)-c^\eps,\;\;\; W_i^\eps(x)\geq q_ix_n+\varphi(x)-c^\eps,\]
for $x\in B_{x^0,\rho}\cap\Om^\eps$. This may be rewritten as
\begin{align*}
v^\eps(x)&\geq \alpha^\eps\left(x_n+\eps q x_n+\eps\varphi(x)-\eps c^\eps\right),\\
\sqrt{\ov{v}_i^\eps\un{v}_i^\eps}(x)&\geq \sqrt{\ov{\alpha}_i^\eps\un{\alpha}_i^\eps}\left(x_n+\eps\ov{q}_ix_n+\eps\varphi(x)-\eps c^\eps\right).
\end{align*}
Hence, up to changing $c^\eps$ into $2c^\eps+C\eps$, for some large enough constant $C$ (that does not depend on $\eps$), we have for any $x\in B_{x^0,\rho}\cap\Om^\eps$:
\begin{equation}\label{eq_aux_phi0}
v^\eps(x)\geq \alpha^\eps (1+\eps q)\psi_0(x),\qquad  \sqrt{\ov{v}_i^\eps\un{v}_i^\eps}(x)\geq \sqrt{\ov{\alpha}_i^\eps\un{\alpha}_i^\eps}(1+\eps q_i)\psi_0(x),
\end{equation}
where for any $t\geq 0$ we have defined
\[\psi_t(x)=x_n+\eps \varphi(x)- \eps c^\eps +t.\]
We can therefore consider the maximal $t^\eps\geq 0$ such that the previous set of inequalities \eqref{eq_aux_phi0} still holds \BBB with $\psi_t$ instead of $\psi_0$\EEE. There must be a contact point $x^\eps \in \ov{B_{x^0,\rho}\cap\Om^\eps}$ for one of the functions, and we may assume without loss of generality (up to changing $\BBB \overline p \EEE$ into $\BBB \overline p \EEE/2$, increasing $\sigma$ and reducing $\rho$ accordingly) that $x^\eps\to x^0$.\newline

Suppose $x^\eps\in \Om^\eps$ and that it is a contact point for $v^\eps$ (the same argument holds for the other functions $W_i^\eps$), and we then have by maximum principle
\[ (1+\eps q)\eps\frac{2\sigma n}{n+1}=(1+\eps q)\eps \Delta\varphi(x^\eps)\leq \frac{1}{\alpha^\eps}\Delta v^\eps(x^\eps)\leq \eps^2,\]
which is a contradiction for a small enough $\eps$. As a consequence $x^\eps\in\partial\Om^\eps$, so that there is equality in all of the above inequalities at $x^\eps$. Hence there exists a blow-up of $(v,\ov{v}_1,\hdots,\un{v}_m)$ at $x^\eps$, and comparing the derivatives at this blow-up \eqref{eq_geomean} we get

\[1\geq |\alpha^\eps|^2(1+\eps q)^2|\nabla\psi_{t^\eps}(x^\eps)|^2+\sum_{i=1}^{m}\ov{\alpha}_i^\eps\un{\alpha}_i^\eps(1+\eps q_i)^2|\nabla\psi_{t^\eps}(x^\eps)|^2\]
which after simplification becomes
\[|\nabla\psi_t(x^\eps)|^2\leq 1+o_{\eps\to 0}(\eps).\]
This is a contradiction since $|\nabla\psi_t(x^\eps)|^2=1+2p\eps +o_{\eps\to 0}(\eps)$. As a consequence we get $\partial_n h\leq 0$ on $B\cap\{x_n=0\}$ in the viscosity sense.\newline

\smallskip
\noindent
\BBB \textit{Inequality $\partial_n h\geq 0$.} \EEE Suppose  {\BBB by contradiction}, that this time $h(x)\leq \varphi(x)$ with $p<0$ and $\sigma<0$. As a consequence following the previous reasoning for some sequence $c^\eps\to 0$ and in some neighbourhood $B(x^0,\rho)\cap\Hn$ we have
\[v^\eps(x)\leq \alpha^\eps(1+\eps q)\psi_0(x),\;\;\;\ov{v}_i^\eps(x)\leq \ov{\alpha}_i^\eps(1+\eps\ov{q}_i)\psi_0(x),\;\;\;\un{v}_i^\eps(x)\leq \un{\alpha}_i^\eps(1+\eps\un{q}_i)\psi_0(x),\]
where $\psi_t$ is defined as previously. Consider the largest $t$ such that the inequalities
\[v^\eps(x)\leq \alpha^\eps(1+\eps q)\psi_t(x),\;\;\;\frac{1}{2}\left(\sqrt{\frac{\un{\alpha}_i^\eps}{\ov{\alpha}_i^\eps}}\ov{v}_i^\eps(x)+\sqrt{\frac{\ov{\alpha}_i^\eps}{\un{\alpha}_i^\eps}}\un{v}_i^\eps\right)\leq \sqrt{\ov{\alpha}_i^\eps\un{\alpha}_i^\eps}(1+\eps q_i)\psi_t(x)\]
are verified in $B(x^0,\rho)\cap\Hn$: at the largest $t$ there is some contact point $x^\eps$, and either $x^\eps\notin\Om^\eps$ by maximum principle as earlier (\BBB we use here the estimates \eqref{eq_estapriorialpha_i}\EEE) or $x^\eps\in\partial\Om^\eps$ and we have the viscosity condition
\[1\leq |\alpha^\eps|^2(1+\eps q)^2|\nabla\psi_{t^\eps}(x^\eps)|^2+\sum_{i=1}^{m}\ov{\alpha}_i^\eps\un{\alpha}_i^\eps(1+\eps q_i)^2|\nabla\psi_{t^\eps}(x^\eps)|^2\]
which after simplification becomes
\[|\nabla\psi_t(x^\eps)|^2(=1+2p\eps+o(\eps))\geq 1+o_{\eps\to 0}(\eps).\]
Since $p<0$ this is a contradiction for a small enough $\eps$.\newline

Now that the Neumann boundary condition \eqref{eq:Neum_visc} is verified in the viscosity sense, $h$ may be extended as a smooth harmonic function on $B_1$ by an even reflexion through $\partial \Hn$.

We may now conclude: $V$, $\ov{V}_i$, $\un{V}_i$ may be respectively extended as harmonic functions $V'$, $\ov{V}_i'$, $\un{V}_i'$ on $B_1$, with values in $[-5,5]$. Indeed we first write
\[V=\left(\alpha^2 V+\sum_{i=1}^{m}\ov{\alpha}_i\un{\alpha}_i\frac{\ov{V}_i+\un{V}_i}{2}\right)+\sum_{i=1}^{m}\ov{\alpha}_i\un{\alpha}_i\frac{(V-\ov{V}_i)+(V-\un{V}_i)}{2}.\]
The first term extends by even reflection thanks to \eqref{eq:Neum_visc}, and the second by odd reflection  (since $V-\ov{V}_i=V-\un{V}_i=0$ over $B_1\cap\{x_n=0\}$), and as a consequence $V$ extends harmonically with a bound $|V|\leq 3$. We then write for each $i$, $\ov{V}_i=V-(V-\ov{V}_i)$ and $\un{V}_i=V-(V-\un{V}_i)$, so that $\ov{V}_i,\un{V}_i$ extend harmonically in $B_1$ into functions bounded by $5$.   Recalling \eqref{eq:harmonic_2grad_est}  we find $c\in\R$, $z\in\R^{n-1}\times 0$, $q,\ov{q}_i,\un{q}_i\in\R$ such that for any $x\in B_\tau$,
\begin{align*}
\left|V(x)-c-z\cdot x'-q x_n\right|&\leq \frac{1}{8}\tau,\\
\left|\ov{V}_i(x)-c-z\cdot x'-\ov{q}_i x_n\right|&\leq \frac{1}{8}\tau,\\
\left|\un{V}_i(x)-c-z\cdot x'-\un{q}_i x_n\right|&\leq \frac{1}{8}\tau.
\end{align*}
Thus by Haudsdorff convergence of the graphs, for any small enough $\eps$ and any $x\in B_\tau$ we have
\begin{equation}\label{eq:flat_contrad}\left|\frac{v^\eps(x)-\alpha^\eps x_n}{\eps\alpha^\eps}-c-z\cdot x'-qx_n\right|\leq \frac{1}{6}\tau,\end{equation}
and the same holds accordingly for $\ov{v}_i^\eps$, $\un{v}_i^\eps$. Set now 
\[S^\eps=(1+\eps q)^2(\alpha^\eps)^2+\sum_{i=1}^{m}\ov{\alpha}_i^\eps\un{\alpha}_i^\eps(1+\eps \ov{q}_i)(1+\eps \un{q}_i).\]

Since $q\alpha^2 +\sum_{i=1}^{m}\ov{\alpha}_i^\eps\un{\alpha}_i^\eps\frac{\ov{q}_i+\un{q}_i}{2}=0$, then $|S^\eps-1|=o(\eps)$. Let now
\[{\alpha^\eps}'=\frac{(1+\eps q)\alpha^\eps}{ \sqrt{S^\eps}},\;\;\;{\ov{\alpha}_i^\eps}'=\frac{(1+\eps q_i) \ov{\alpha}_i^\eps}{ \sqrt{S^\eps}},\;\;\;{\un{\alpha}_i^\eps}'=\frac{(1+\eps q_i)\un{\alpha}_i^\eps}{ \sqrt{S^\eps}}, \]
and let
\[e'=\frac{e_n+\eps z}{\sqrt{1+\eps^2|z'|^2}},\;\;\;a'=\frac{c}{\tau}-\frac{1}{4}\eps,\;\;\;b'=\frac{c}{\tau}+\frac{1}{4}\eps.\]

Then \eqref{eq:flat_contrad} may be rewritten as

\begin{align*}
\left[(1+\eps q)e_n+\eps z'\right]\cdot x+\eps c-\frac{1}{6}\tau\leq&\frac{v^\eps(x)}{\alpha^\eps}\leq \left[(1+\eps q)e_n+\eps z'\right]\cdot x+\eps c+\frac{1}{6}\tau,\\
\left[(1+\eps \ov{q}_i)e_n+\eps z'\right]\cdot x+\eps c-\frac{1}{6}\tau\leq&\frac{\ov{v}_i^\eps(x)}{\ov{\alpha}_i^\eps}\leq \left[(1+\eps \ov{q}_i)e_n+\eps z'\right]\cdot x+\eps c+\frac{1}{6}\tau,\\
\left[(1+\eps \un{q}_i)e_n+\eps z'\right]\cdot x+\eps c-\frac{1}{6}\tau\leq&\frac{\un{v}_i^\eps(x)}{\un{\alpha}_i^\eps}\leq \left[(1+\eps \un{q}_i)e_n+\eps z'\right]\cdot x+\eps c+\frac{1}{6}\tau,\\
\end{align*}

for any $x\in B_\tau\cap\Om^\eps$, which simplifies as $\eps\to 0$ to

\[e'\cdot x+a'\leq \frac{v^\eps(x)}{{\alpha'}^\eps},\frac{\ov{v}_i^\eps(x)}{{\ov{\alpha}_i'}^\eps},\frac{\un{v}_i^\eps(x)}{{\un{\alpha}_i'}^\eps}\leq e'\cdot x+b',\ \forall x\in\Om^\eps\cap B_\tau\]
so that all functions verify the flatness improvement, which is a contradiction for small enough $\eps$. This concludes the proof in this case.

\smallskip
\noindent
\textbf{Case $m^\eps\to m$, $\delta>0$.} This case follows more closely \cite{KL18}; the only difference here is that the Neumann boundary condition verified at the limit is
\[\partial_{n}\left(\alpha^2V+\sum_{i=1}^{m}\frac{\ov{\alpha}_i^2\ov{V}_i+\un{\alpha}_i^2\un{V}_i}{2}\right)=0\text{ in }B_1\cap\{x_n=0\}.\]

\smallskip
\noindent
\textbf{Case $m^\eps\to \infty$.} In this case we treat $\delta>0$ and $\delta<0$ at once. Define $V^\eps,\ov{V}_i^\eps$ and $\un{V}_i^\eps$ as previously and thanks to Lemma \ref{lemma_compact}, we may change their order (in $i$) and assume a convergence to some limits $V,\ov{V}_i$ and $\un{V}_i$ as $\eps\to 0$ (in the sense of local Hausdorff convergence of the graphs) which is uniform in $i$. We still have the Dirichlet boundary condition $V-\ov{V}_i=V-\un{V}_i=0$ on $B_1\cap\{x_n=0\}$. Due to the estimates \eqref{eq_estapriorialpha_i} we have
\[\left|1-(\alpha^\eps)^2\right|\leq 4L^2|\delta^\eps|^\frac{1}{2}m^\eps\leq 2L^2(m^\eps)^{-1}\to 0.\]

We now prove the Neumann boundary condition (in the viscosity sense)
\begin{equation}\label{eq:Neum_V_visc}\partial_n V=0\text{ on }B\cap\{x_n=0\}.\end{equation}
We proceed as previously: letting $x^0\in B_1\cap\{x_n=0\}$, we pick a polynomial function defined as in \eqref{eq:polynom_h} with $p>0$ touching $V$ from below at a point $x^0$. Writing $\ov{V}_i(x)=V(x)+\ov{q}_i x_n+\mathcal{O}(|x-x^0|^2)$ (where the remainder term only depends on $1-|x^0|$ and $n$ but not on $i$, by uniform regularity of harmonic functions) and similarly for $\un{V}_i$, using the uniform convergence of the graphs we get that in a neighborhood $B_{x^0,\rho}\cap\Om^\eps$ of $x^0$ it holds for every $i$:
\begin{align*}
v^\eps(x)\geq \alpha \psi_0(x),&\;\;\;\ov{v}_i^\eps(x)\geq \ov{\alpha}_i(1+\eps \ov{q}_i)\psi_0(x),\;\;\;\un{v}_i^\eps(x)\geq \un{\alpha}_i(1+\eps \un{q}_i)\psi_0(x)\;\;\;\text{ if }\delta>0,\\
v^\eps(x)\geq \alpha \psi_0(x),&\;\;\;\sqrt{\ov{v}_i^\eps(x)\un{v}_i^\eps(x)}\geq \sqrt{\ov{\alpha}_i\un{\alpha}_i^\eps}\left(1+\eps \frac{\ov{q}_i+\un{q}_i}{2}\right)\psi_0(x)\;\;\;\text{ if }\delta<0,
\end{align*}
where for any $t\geq 0$ we have defined
\[\psi_t(x)=x_n+\eps \varphi(x)- \eps c^\eps +t\]
and $c^\eps\to 0^+$. Note that $(\ov{q}_i,\un{q}_i)$ are uniformly bounded by a constant $M$ only depending on $n$ and $1-|x^0|$.\newline

Taking then the largest $t^\eps\geq0$ such that these inequalities are verified over $B_{x^0,\rho}\cap\Om^\eps$ for every $i$, then there is a contact point $x^\eps$ either for the function $v^\eps$ or one of the $\ov{v}_i^\eps,\un{v}_i^\eps$, with $x^\eps\to x^0$. Then as previously by maximum principle the contact point $x^{\eps}$ lies in $\partial \Om^\eps$ when $\eps$ is small enough, and there is equality at $x^\eps$ in all the inequalities. Comparing the derivatives at $x^\eps$ we get (with the viscosity condition)
\begin{align*}
1&\geq |\alpha^\eps|^2|\nabla\psi_{t^\eps}(x^\eps)|^2+\sum_{i=1}^{m^\eps}\frac{1}{2}\left(|\ov{\alpha}_i^\eps|^2(1+\eps\ov{q}_i)^2+|\un{\alpha}_i^\eps|^2(1+\eps\un{q}_i)^2\right)|\nabla\psi_{t^\eps}(x^\eps)|^2\\
&\geq \left(1-2L^2m^\eps|\delta^\eps|^\frac{1}{2}(M\eps +M^2\eps^2)\right)|\nabla\psi_{t^\eps}(x^\eps)|^2\\
&\geq \left(1-\frac{M}{m^\eps}\eps+o_{\eps\to0}(\eps)\right)\left(1+\eps\partial_n\varphi(x^\eps)\right),
\end{align*}
which is a contradiction for small enough $\eps$ since $\partial_n\varphi(x^\eps)\to p>0$. This ensures $\partial_nV\leq0$ over $B_1\cap\{x_n=0\}$ in the viscosity sense.\bigbreak

Likewise we get $\partial_nV\geq0$ on $\{x_n=0\}$ in the viscosity sense. As a consequence $V$ verifies the Neumann boundary condition \eqref{eq:Neum_V_visc}. We may extend $V$ by an even reflexion and the $V-\ov{V}_i$, $V-\un{V}_i$ by odd reflexions, so that $V$, $\ov{V}_i,\un{V}_i$ extend as harmonic {\BBB functions} on $B_1$ with values in $[-3,3]$, and relying on \eqref{eq:harmonic_2grad_est} we obtain as previously a contradiction, finishing the proof in the case $m^\eps\to \infty$. \end{proof}

\begin{corollary}\label{prop_flatness}
Let $L\geq 1$ and $\delta\in\R$. Then there exists $\eps_{n}(L)>0$, $\gamma_n\in (0,1)$ verifying the following property. For any $m\in\mathbb{N}^*$, $\delta\in\R$ verifying $L^2m^2|\delta|^\frac{1}{2}\leq \frac{1}{8}$, and for any $\eps<\eps_n(L)$,  $(v,\ov{v}_1,\un{v}_1,\hdots,\ov{v}_m,\un{v}_m)\in\Sol_{m,\delta}(L)$ that is $\eps$-flat in the sense of Definition \ref{def_flat}, then there exists $g\in\C^{1,\gamma_n}\left(B^{n-1}_{1/2},[-\eps,\eps]\right)$ such that $\Vert g\Vert_{\C^{1,\gamma_n}\left(B^{n-1}_{1/2}\right)}\lesssim \eps$ and 
\[ \{v>0\}\cap \left(B^{n-1}_{1/2}\times \left[-1/2,1/2\right]\right)=\left\{(x',x_n)\in B^{n-1}_{1/2}\times \BBB(-1/2,1/2)\EEE:\ x_n \BBB > \EEE g(x')\right\}.\]
\end{corollary}
\begin{proof}
This step comes from iterating the flatness improvement Proposition \ref{prop_flatness_improv} as is done in \cite[Theorem 8.1]{V19}. 
\end{proof}

This implies that for $\delta$  small enough \BBB the boundary of any minimizer $\Om$ \EEE  can be written as a  $ \C^{3,\gamma}$  graph on the sphere. This is the object of next Lemma.

\begin{lemma}\label{lem_C2gammavect}
Suppose that $|\delta|\ll k^{-\left(6+\frac{8}{n}\right)}g_n(k)$ and let $\Om$ be a centered minimizer of \eqref{eq_funcvect}. Then there exists  some dimensional $\gamma=\gamma_n\in(0,1)$, $D_n>0$  and some $h\in \C^{3,\gamma}\left(\partial B,\left[-\frac{1}{2},\frac{1}{2}\right]\right)$ with  $\Vert  h\Vert_{ \C^{3,\gamma}(\partial B)}\leq D_n$  such that
\BBB \[\Om=\left\{s(1+h(x))x,\ s\in[0,1), \ x\in\partial B\right\}.\] \EEE
\end{lemma}

\begin{proof}
\BBB

For the $\C^{1,\gamma}$ estimate on $h$ we proceed exactly as in the proof of Lemma \ref{lem_C1gamma}.
\BBB The $\C^{3,\gamma}$ bound on $h$ follows again from \cite[Theorem 2.4]{DS14}: indeed this time the optimality condition is
$$|\nabla w_\Om|^2+T(\Om)^2\delta\sum_{i=k}^{l}|\nabla u_i|^2=Q$$
where $Q$ is the constant defined in equation \eqref{eq:defQ_mult}, whence
$$|\nabla w_\Om|^2=\frac{Q}{1+T(\Om)^2\delta\sum_{i=k}^{l}\frac{|\nabla u_i|^2}{|\nabla w_\Om|^2}}.$$
By \cite[Theorem 2.4]{DS14} applied to each $\frac{u_i}{w_\Om}$ we get $\Vert \nabla w_\Om\Vert_{\C^{1,\gamma}(\partial\Om)}\lesssim 1$, which implies $\Vert h\Vert_{\C^{2,\gamma}(\partial\Om)}\lesssim 1$, and then $\Vert h\Vert_{\C^{3,\gamma}(\partial\Om)}\lesssim 1$ by iterating again.
\EEE

\end{proof}

\subsection{Minimality of the ball among nearly spherical sets.} The purpose of this subsection is to show the minimality of the ball for the functional $T^{-1}+\delta\sum_{i=k}^l\lambda_i$ among nearly spherical sets in the sense of Definition \ref{def:NS}. 
This time in the definition of a nearly spherical set (see Definition \ref{def:NS}) we rather take $\gamma$ and $D_n$ as in Lemma \ref{lem_C2gammavect} instead of Lemma \ref{lem_C2gamma}, although for simplicity we do not introduce additional definition and notations. {\BBB We will keep that same $\gamma$ for the rest of the section.}
{\BBB The minimality result is the following.}

\begin{proposition}\label{prop:minloc_vec}
Let $B_h$ be a nearly spherical  centered  set such that $\Vert h\Vert_{L^1(\partial B)}\ll k^{-1-\frac{4}{n}}g_n(k)$ and suppose that $|\delta|\ll k^{-\left(2+\frac{8}{n}\right)}g_n(k)$. Then we have
\[T(B_h)^{-1}+\delta\sum_{i=k}^l\lambda_i(B_h)\geq T(B)^{-1}+\delta\sum_{i=k}^l\lambda_i(B)\]
with equality if and only if $B_h=B$.
\end{proposition}

The strategy is the same as in Subsection \ref{subsec_NS}, with some differences due to the fact that we are considering multiple eigenvalues.  We will make use of Lemma \ref{lem:Kato} which still applies to multiple eigenvalues, and as in the case of Proposition \ref{prop_minloc} {\BBB we perform} a second order Taylor expansion of the functional $T^{-1}+\sum_{i=k}^l\lambda_i$: the main difference {\BBB is that even though} the eigenvalues $\lambda_k(B)=\ldots=\lambda_l(B)$ are multiple {\BBB(in which case each individual eigenvalue is not shape differentiable), the sum $\sum_{i=k}^l\lambda_i((\text{Id}+\zeta)(B))$ is still} smooth in $\zeta$ (and even analytic, see \cite[Theorem 2.6]{LamLan06}), and has a critical point at the ball.

\begin{proposition}\label{prop_controllambda_k_vec}
Let $B_h$ be a nearly spherical set such that  $\Vert h\Vert_{L^1(\partial B)}\ll k^{-1-\frac{4}{n}}g_n(k)$. Then it holds 
\[\left|\sum_{i=k}^{l}(\lambda_i(B_h)-\lambda_i(B))\right|\leq C_n \frac{k^{2+\frac{8}{n}}}{g_n(k)}\Vert h\Vert_{H^{1/2}(\partial B)}^2.\]
\end{proposition}

\begin{proof}
We proceed as in the proof of Proposition \ref{prop_controllambda_k}: letting $(\mu_i(t),u_i(t))$ be given by Lemma \ref{lem:Kato}, since $\Vert h\Vert_{L^1(\partial B)}\ll k^{-1-\frac{4}{n}}g_n(k)$ we have by the same argument that $|\mu_i(t)-\mu_j(t)|\geq \frac{1}{2}g_n(k)$ for any $|t|\leq 1$ and $i,j$ such that $i\in\{k,\hdots, l\}$, $j\notin\{k,\hdots, l\}$. As a consequence, $u_k(t),\hdots,u_l(t)$ is an orthonormal basis of eigenfunctions of the sum of eigenspaces corresponding to $(\lambda_i(B_{th}))_{i=k,\hdots,l}$, and $(\mu_i(t))_{i=k,\hdots,l}$ is a permutation of $(\lambda_i(B_{th}))_{i=k,\hdots,l}$.\newline

We have
\[\left.\frac{d}{dt}\right|_{t=0}\sum_{i=k}^{l}\lambda_i(B_{th})=\sum_{i=k}^{l}\mu_i'(0)=-\int_{\partial B}\left(\sum_{i=k}^{l}|\nabla u_i(0)|^2\right)\zeta\cdot\nu.\]
Because of the structure of the eigenfunctions of the ball, we know that $\sum_{i=k}^{l}|\nabla u_i(0)|^2$ is a constant \BBB on the boundary of the ball (independent on $(\mu_i(t), u_i(t))$)\EEE, which we denote by $c_{n,k}$, and therefore
\[\left.\frac{d}{dt}\right|_{t=0}\left(\sum_{i=k}^{l}\lambda_i(B_{th})+c_{n,k}|B_{th}|\right)=0.\]
Also by elliptic regularity we have $c_{n,k}\lesssim k^{1+\frac{6}{n}}$.
Using a Taylor formula and since $|B|=|{ B_h}|$, there exists some $t\in [0,1]$ such that

\[\sum_{i=k}^{l}(\lambda_i({ B_h})-\lambda_i(B))=\frac{1}{2}\left(\sum_{i=k}^{l}\mu_i''(t)+c_{n,k}\left.\frac{d^2}{ds^2}\right|_{s=t}|B_{sh}|\right).\]

To reduce notations we fix $t$ and we do not write the dependency in $t$ in the rest of the proof. We thus set $\Om=B_{th}$, $u_i=u_i(t)$, $v_i:=u_i'(t)$ and write $\mu_i$, $\mu_i'$, $\mu_i''$ in place of $\mu_i(t)$, $\mu_i'(t)$, $\mu_i''(t)$. The expression of $\mu_i''$ takes the form
\begin{equation}\label{eq:expression_mu_iseconde}
\mu_i''=\int_{\Om}2\left(|\nabla v_i|^2-\mu_i v_i^2\right)+\int_{\partial \Om}|\nabla u_i|^2\left[H(\zeta\cdot\nu)^2-b( \zeta_{\tau},\zeta_{\tau})+2\zeta_\tau\cdot\nabla_{|\partial\Om}(\zeta\cdot\nu)\right]
\end{equation}
and each $v_i$ verifies
\begin{equation}\begin{cases}\label{eq:u_i'_bis}
	-\Delta v_i-\mu_i v_i=\mu_i' u_i & {\BBB\textrm{ in }}\Om,\\ v_i=-(\zeta\cdot\nu)\partial_\nu u_i & {\BBB \textrm{ on }}\partial\Om, \\ \int_{\Om}(v_i u_j+v_j u_i)=0,& \forall j\in\N^*,
\end{cases}\end{equation}
where the last line is a consequence of $\int_{B_{th}}u_i(t)u_j(t)=\delta_{ij}$ for all $t$.\newline

The ``geometric'' terms in \eqref{eq:expression_mu_iseconde} are estimated exactly as in the proof of Proposition \ref{prop_controllambda_k}. We thus have
\[\left|\int_{\partial \Om}|\nabla u_i|^2\left[H(\zeta\cdot\nu)^2- b(\zeta_{\tau},\zeta_{\tau})+2\zeta_\tau \cdot\nabla_{|\partial\Om}(\zeta\cdot\nu)\right]\right|\lesssim k^{1+\frac{6}{n}}\Vert h\Vert_{H^{1/2}(\partial B)}^2.\]

To estimate the first term of \eqref{eq:expression_mu_iseconde}, there is a difference with the case of a simple eigenvalue lying in the fact that we do not have (and do not expect) a good control of $\int_{\Om}v_iu_j$ when $i,j\in \{k,\hdots,l\}$. Refining the analysis we will see that these terms in fact cancel in the sum $\sum_{i =k}^l\mu_i''$.\newline

Set $I=\{k,\ldots,l\}$ and for each $i\in I$, write $v_i=z_i+w_i$ with $z_i$ the harmonic extension of $-(\partial_\nu u_i)\zeta\cdot\nu$ and $w_i\in H_0^1(\Om)$. The functions $z_i$ verify the same estimates as in the proof of Proposition \ref{prop_controllambda_k}:
\[\|z_i\|_{H^1(\Om)}\lesssim k^{\frac{1}{2}+\frac{4}{n}}\|h\|_{H^{1/2}(\partial B)},\ \|z_i\|_{L^2(\Om)}\lesssim k^{\frac{1}{2}+\frac{2}{n}}\|h\|_{H^{1/2}(\partial B)}.\]
The function $w_i$ verifies $-(\Delta+\mu_i) w_i=\mu_iz_i+\mu_i'u_i$ which ensures
\begin{equation}\label{eq:w_iu_j}\forall j\ne i,\ (\mu_j-\mu_i)\int_\Om w_iu_j=\mu_i\int_\Om z_iu_j\end{equation}
or, written differently,
\begin{equation}\label{eq:v_iu_j}\forall j\neq i,\ (\mu_j-\mu_i)\int_\Om v_iu_j=\mu_j\int_\Om z_iu_j.\end{equation}
We have 
\begin{align*}
\sum_{i\in I}\int_{\Om}\left(|\nabla v_i|^2-\mu_i|v_i|^2\right)&=\sum_{i\in I}\int_{\Om}\left(|\nabla z_i|^2+|\nabla w_i|^2-\mu_i v_i^2 \right)\\
&=\sum_{i\in I}\int_{\Om}|\nabla z_i|^2+\sum_{i\in I,j\in\N^*}\mu_j\left(\int_{\Om}w_iu_j\right)^2-\mu_i\left(\int_\Om v_i u_j\right)^2\\
&=\sum_{i\in I}\int_{\Om}|\nabla z_i|^2+\sum_{i\in I,j\notin I}\left(\frac{\mu_i^2}{\mu_j-\mu_i}-\mu_i\right)\left(\int_{\Om}z_iu_j\right)^2\\
&\hspace{1cm}+\sum_{i\in I,j\in I}\mu_j\left(\int_{\Om}z_iu_j\right)^2-(\mu_j-\mu_i)\left(\int_\Om v_i u_j\right)^2,
\end{align*}
where we used \eqref{eq:v_iu_j} in the last line.
Thanks to the orthogonality conditions from \eqref{eq:u_i'_bis}, we have $\sum_{i,j\in I}(\mu_i-\mu_j)\left(\int_{\Om}v_iu_j\right)^2=0$, hence we conclude
\begin{align*}
\left|\sum_{i\in I}\int_{\Om}\left(|\nabla v_i|^2-\mu_i|v_i|^2\right)\right|&\lesssim \sum_{i\in I}\Vert \nabla z_i\Vert_{L^2(\Om)}^2+g_n(k)^{-1}k^{\frac{4}{n}}\sum_{i\in I}\Vert z_i\Vert_{L^2(\Om)}^2\\
&\lesssim g_n(k)^{-1}k^{1+\frac{8}{n}}|I|\Vert h\Vert_{H^{1/2}(\partial B)}^2\\
&\lesssim g_n(k)^{-1}k^{{ 2+\frac{8}{n}}}\Vert h\Vert_{H^{1/2}(\partial B)}^2.
\end{align*}
As a consequence, $\left|\sum_{i=k}^{l}\mu_i''(t)\right|\lesssim k^{2+\frac{8}{n}}g_n(k)^{-1}\|h\|_{H^{1/2}(\partial B)}^2$ which ends the proof.

\end{proof}

\begin{proof}[Proof of Proposition \ref{prop:minloc_vec}]

This is done exactly as in the proof of Proposition \ref{prop_minloc}.

\end{proof}

\subsection{Conclusion}
\begin{proof}[Proof of Proposition \ref{main2_linvect}]
We proceed exactly as in the proof of Proposition \ref{main2_lin}. Since $|\delta|\ll k^{-\left(3+\frac{4}{n}\right)}$, by application of Proposition \ref{prop_existenceregvect} there exists a minimizer $\Om$ (which we can suppose to be centered) with $|\Om\Delta B|\lesssim k^{3+\frac{4}{n}}|\delta|$. By Lemma \ref{lem_C2gammavect}, since $|\delta|\ll k^{-6-\frac{8}{n}}g_n(k)$ we have $\Om=B_h$ where  $\Vert h\Vert_{ \C^{3,\gamma}(\partial B)}\leq D_n$, so that $\Om$ is a nearly spherical set. Since 
\[\Vert h\Vert_{L^1(\partial B)}\lesssim |\Om\Delta B|\lesssim k^{3+\frac{4}{n}}|\delta|.\]

then for $k^{3+\frac{4}{n}}|\delta|\ll k^{-2-\frac{8}{n}}g_n(k)$ (which is verified for $|\delta|\ll k^{-6-\frac{10}{n}}g_n(k)$), we can therefore apply Proposition \ref{prop:minloc_vec} to conclude that $\Om$ is a ball.

\end{proof}

\section{Discussion and consequences}\label{sec_disc}
\subsection{About the sharpness of the results}

We prove in the proposition below that the exponents $\alpha=1/2$ and $\alpha=1$ on the right-hand side of \eqref{ex1.03} given by Theorems \ref{main_sqrt} and \ref{main_lin} are sharp for every $k$ in dimension $n=2$, and that similarly \BBB Theorem \EEE \ref{main_linvect} is sharp for any $k\neq 2$. Proving it for any dimension would require a full second order analysis of the spectrum of smooth deformations of the ball in every dimension, in the spirit of the two dimensional work of Berger in \cite{Be15}.

\begin{proposition}\label{prop:sharp}
Let $n=2$, $k\geq 2$. There exists a constant $c_{k}>0$ and $\Om^\eps=\phi^\eps(B)$ a sequence of domains (in $\A$) with $\Vert\phi^\eps- \Id\Vert_{\C^1}\leq\eps$ such that
\begin{equation}\label{ex1.08}
\mbox{if } \lambda_k(B) \mbox{ is simple } |\lambda_k(\Om^\eps)-\lambda_k(B)|\geq c_{k}(\lambda_1({\BBB \Om^\eps})-\lambda_1(B)),
\end{equation}
\begin{equation}\label{ex1.09} \mbox{if } \lambda_k(B) \mbox{ is double } |\lambda_k(\Om^\eps)-\lambda_k(B)|\geq  c_{k}\left(\lambda_1({\BBB \Om^\eps})-\lambda_1(B)\right)^\frac{1}{2},
\end{equation}

\BBB
and moreover if $\lambda_k(B)$ is double and $k\neq 2$, there exists a sequence $(\Om^\eps)$ such that
\[\left|{\BBB \lambda_{k}(\Om^\eps)-\lambda_k(B)+\lambda_{k+1}(\Om^\eps)-\lambda_{k+1}(B)}\right|\geq c_k \left(\lambda_1(\Om^\eps)-\lambda_1(B)\right).\]
\end{proposition}
\begin{proof}
Suppose first $\lambda_k(B)$ is simple. Then, following \cite[Lemmas 1 and 3]{Be15} we get an explicit perturbation $\phi^\eps$ of the identity, expressed as a Fourier series, which is  preserving the area at the second order and  for which the second order term in the asymptotic developments  of both $\lb_1 (\phi^\eps (B))$ and  $\lb_k (\phi^\eps (B))$ are non vanishing, proving \eqref{ex1.08}.

If   $\lambda_k(B)$ is double, (for instance $\lambda_k(B)=\lambda_{k+1}(B)${\BBB, the other case $\lambda_k(B)=\lambda_{k-1}(B)$ being similar}), then for any vector field $\zeta\in\C^{\infty}_c(\R^2,\R^2)$ such that $\int_{\partial B}\zeta\cdot x=0$ the directional \BBB derivatives \EEE of $\lambda_k,\lambda_{k+1}$ in the direction $\zeta$ are respectively the first and second eigenvalues of the symmetric matrix
\[\begin{pmatrix}
-\int_{\partial B}|\nabla u_k|^2\zeta\cdot x & -\int_{\partial B}(\nabla u_k\cdot\nabla u_{k+1})(\zeta\cdot x) \\ -\int_{\partial B}(\nabla u_k\cdot\nabla u_{k+1})(\zeta\cdot x)  & -\int_{\partial B}|\nabla u_{k+1}|^2\zeta\cdot x 
\end{pmatrix}.\]
Moreover, since the functions $(u_k,u_{k+1})$ are not radial, we may choose a field $\zeta$ such that $\int_{\partial B}|\nabla u_k|^2\zeta\cdot x\neq 0$, which gives a non-zero matrix with two nonzero (opposite) eigenvalues. Letting $\Om^\eps=\frac{(\Id+\eps\zeta)(\Om)}{|(\Id+\eps\zeta)(\Om)|^{1/2}}$ we have $|\lambda_k(\Om^\eps)-\lambda_{k}(B)|>c\eps$ for some $c>0$ and small enough $\eps$, whereas $\lambda_1(\Om^\eps)-\lambda_1(B)<C\eps^2$.

\BBB
The last property is a consequence of the proof of \cite[Lemma 5]{Be15}: \BBB in this paper, the author defines $\Om^\eps=B_{h_{\eps}}$ with
$$h_\eps(\theta)=\sum_{n\in\mathbb{Z}}(\eps a_n+\eps^2b_n)e^{in\theta},$$
where all but a finite number of coefficients are non-zero, and the coefficients are chosen such that $h$ is real-valued. Let $k\neq 2$ such that $\lambda_k(B)=\lambda_{k+1}(B)$, then in the proof of \cite[Lemma 5]{Be15}, they prove the existence of a choice of coefficients $(a_n),(b_n)$, such that either $a_2$ or $a_3$ is non-zero and for $j\in\{k,k+1\}$ we have
$$|\Om^\eps|\lambda_j(\Om^\eps)=|B|\lambda_j(\Om)-c_k\eps^2+\mathcal{O}(\eps^3)$$
for some constant $c_k>0$. On the other hand, since $a_2$ or $a_3$ is non-zero, then for some constant $d_k>0$ we have
$$|\Om^\eps|\lambda_1(\Om^\eps)\geq |B|\lambda_1(B)+d_k\eps^2.$$
Thus, $\Om^\eps$ is an example of sequence satisfying the last property.\EEE
\end{proof}

\subsection{Proof of Corollary \ref{cor_reversekohlerjobin}: the reverse Kohler-Jobin inequality.}

The linear bound of Theorem \ref{main_lin} (or equivalently Proposition  \ref{main2_lin}) gives us a non-trivial conclusion on the reverse of the Kohler-Jobin inequality from Corollary  \ref{cor_reversekohlerjobin}. This answers, in full generality, the question raised in \cite{BBP21}.

\begin{proof}[Proof of Corollary \ref{cor_reversekohlerjobin}]
By Proposition \ref{main2_lin} there exists some $\delta_n>0$ such that $\A \ni \Om\mapsto T^{-1}(\Om)-\delta_n\lambda_1(\Om)$ is minimal on the ball. Let $p\geq 1$ and $\Om\in\A$ be such that $T(\Om)\lambda_1(\Om)^{\frac{1}{p}}>T(B)\lambda_1(B)^{\frac{1}{p}}$, then
\begin{align*}
\frac{\lambda_1(\Om)}{\lambda_1(B)}&> \left(\frac{T(B)}{T(\Om)}\right)^p\geq 1+p\left(\frac{T(B)}{T(\Om)}-1\right)\geq 1+\delta_n p T(B)\left(\lambda_1(\Om)-\lambda_1(B)\right)
\end{align*}
which implies $p<p_n:=(\delta_n T(B)\lambda_1(B))^{-1}$.
\end{proof}
Note \BBB that \EEE we do not have explicit information on the value of $p_n$ even in low dimension as the proof of Proposition \ref{main2_lin} relies on a contradiction and compactness argument at several points.

\subsection{Stability of more general spectral functionals and proof of Theorem \ref{main_perturb}}\label{ssect:stability}

The linear bound on a cluster of eigenvalues in Theorem \ref{main_linvect} is established for the sum $\sum_{i=k}^{l}\left[\lambda_i(\Om)-\lambda_i(B)\right]$; one could argue that this choice of function is arbitrary and we could, for instance, replace it by the geometric mean.  We thus consider in this section the general functionals verifying the hypotheses given by Theorem \ref{main_perturb}, and prove the stability result stated in this result.

\begin{proof} [Proof of Theorem \ref{main_perturb}]
It is useful to partition $\llbracket1,k\rrbracket$ {\BBB into} $(I_s)_{s=1,\hdots,p}$ {\BBB made of indices of} clusters of eigenvalues. Let $c_s$ be the common value of $\frac{\partial F}{\partial\lambda_i}(\lambda_1(B),\hdots,\lambda_k(B))$ for $i\in I_s$. Then there exists some $C>0$ such that for any $\lambda\in(\R_+^{*})^k$:

\[\left|F(\lambda_1,\hdots,\lambda_k)-F(\lambda_1(B),\hdots,\lambda_k(B))-\sum_{s=1}^{p}c_s\sum_{i\in I_s}\left[\lambda_i-\lambda_i(B)\right]\right|\leq C\sum_{i=k}^{l}(\lambda_i-\lambda_i(B))^2.\]
\BBB Applying \EEE Theorem \ref{main_linvect} to each $\sum_{i\in I_s}\left[\lambda_i(\Om)-\lambda_i(B)\right]$ and Theorem \ref{main_sqrt} to each $(\lambda_i(\Om)-\lambda_i(B))^2$: we get some constant $D>0$ such that

\begin{equation}\label{eq_Fnonlin}
\left|F(\lambda_1(\Om),\hdots,\lambda_k(\Om))-F(\lambda_1(B),\hdots,\lambda_k(B))\right|\leq D\left(T(\Om)^{-1}-T(B)^{-1}\right)T(\Om)^{-1}.
\end{equation}

Let $\delta\in\R$, consider $\Om\in\A$ a domain such that
\[T^{-1}(\Om)+\delta F(\lambda_1(\Om),\hdots,\lambda_k(\Om))\leq T^{-1} (B)+\delta F(\lambda_1(B),\hdots,\lambda_k(B))\]
for some $\delta\in\R$. Due to Proposition \ref{prop:growth_egv}, this gives   for some $C_{n,k}>0$  
\begin{align*}
T^{-1}(\Om)&\leq T^{-1}(B)+\delta F((\lambda_i(B))_{i=1,\hdots,k})+C|\delta|(1+|(\lambda_i(\Om))_{i=1,\hdots,k}|)\\
&\leq T^{-1}(B)+\delta F((\lambda_i(B))_{i=1,\hdots,k})+CC_{n,k}|\delta|(1+T^{-1}(\Om)).
\end{align*}
As a consequence, when $|\delta|$ is small enough then $T^{-1}(\Om)\leq 2T^{-1}(B)$. Equation \eqref{eq_Fnonlin} then gives
\[\left|F(\lambda_1(\Om),\hdots,\lambda_k(\Om))-F(\lambda_1(B),\hdots,\lambda_k(B))\right|\leq 2DT(B)^{-1}\left(T(\Om)^{-1}-T(B)^{-1}\right)\]
which gives the result when $|\delta|\leq (2D)^{-1}T(B)$.
\end{proof}
From the Kohler-Jobin  inequality, we get:
\begin{corollary}\label{main_perturb2}
Let $F$ be as {\BBB in Theorem \ref{main_perturb}}, then there exists $\delta_F>0$ such the functional
\[\Om\in\A\mapsto \lb_1(\Om) +\delta F(\lambda_1(\Om),\hdots,\lambda_k(\Om))\]
is minimal on the ball as soon as $|\delta |\le \delta_F$. 
\end{corollary}

As the reader noticed, in the whole paper we kept track of the dependence of constants in terms of $k$ the order of the involved eigenvalues; this allows us to consider spectral functionals which depend on an infinite number of eigenvalues, such as the trace of the heat kernel
\[Z_\Om(t):=\int_{\Om}K_t(x,x)dx=\sum_{k\geq 1}e^{-t\lambda_k(\Om)}.\]
In the following proposition we denote $a_n=\left(\frac{n}{n+2}\right)\frac{{4\pi^2}}{\om_n^{4/n}}$ so that $\lambda_k(\Om)\geq a_n k^\frac{2}{n}$ thanks to Proposition \ref{prop:growth_egv}.
\begin{proposition}\label{ex1.10}
Let $f\in\C^2(\R_+^*,\R)$ be a smooth function such that
\[B_n(f):=\sum_{i\geq 1}\left(i^{7+\frac{8}{n}}\sup_{\lb\geq a_n i^\frac{2}{n}}|f'(\lb)|+i^{6+\frac{10}{n}}\sup_{\lb\geq a_n i^\frac{2}{n}}|f''(\lb)|\right)<+\infty.\]
Then  there exists $C_n>0$ such that  for any $\Om\in\A$ we have
\[\sum_{i\geq 1}\left(f(\lambda_i(\Om))-f(\lambda_i(B))\right)\leq C_nB_n(f)(\lambda_1(\Om)-\lambda_1(B)).\]
In particular, this gives
\begin{equation}\label{ex1.11}
|Z_\Om(t)-Z_B(t)|\leq C_nt^{-(4n+3)}(\lambda_1(\Om)-\lambda_1(B)).
\end{equation}
\end{proposition}
Note that the hypotheses above apply  also to $f(\lambda)=\lambda^{-s}$ for any $s>4n+3$.
To prove Proposition \ref{ex1.10}, we start with the following Lemma:
\begin{lemma}\label{lem_separation_eigen}
There exists constants $c_n,D_n>0$ such that for any $k\leq l$ satisfying $\lambda_{i+1}(B)-\lambda_i(B)\leq c_ni^{\frac{2}{n}-1}$ for every $i\in\{k,\hdots,l\}$, we have $l\leq D_n k$.
\end{lemma}
\begin{proof}
Let $c_n:=\frac{a_n}{n}$, then for any such $k,l$ we have
\[a_n l^\frac{2}{n}\leq \lambda_l(B)\leq \lambda_k(B)+c_n\sum_{i=k}^{l-1}i^{\frac{2}{n}-1}\leq \left(1+\frac{4}{n}\right)\lambda_1(B)k^\frac{2}{n}+\frac{nc_n}{2}l^\frac{2}{n}\]
so with our choice of $c_n$:
\[\frac{l}{k}\leq \left(\left(1+\frac{4}{n}\right)\frac{2\lambda_1(B)}{a_n}\right)^\frac{n}{2}=:D_n.\]
\end{proof}
\begin{proof}[Proof of Proposition \ref{ex1.10}.]
Let us first prove an a priori estimate on $\lambda_1(\Om)$. Indeed by application of Theorem \ref{main_sqrt} we have for any $\Om{\BBB \in\A,}$
\begin{align*}
\sum_{i\geq 1}f(\lambda_i(\Om))-f(\lambda_i(B))&\lesssim \sum_{i\geq 1}\sup_{t\geq a_n i^\frac{2}{n}}|f'(t)| i^{2+\frac{4}{n}}\lambda_1(\Om)^\frac{1}{2}(\lambda_1(\Om)-\lambda_1(B))^\frac{1}{2}\\
&\lesssim B_n(f)(\lambda_1(\Om)-\lambda_1(B))\text{ if }\lambda_1(\Om)\geq 2\lambda_1(B)
\end{align*}
We now suppose without loss of generality that $\lambda_1(\Om)\leq 2\lambda_1(B)$. Using lemma \ref{lem_separation_eigen} we split $\N^*$ into a partition of intervals $(I_p)_{p\in\N^*}$ such that $\sup(I_p)\leq C_n\inf(I_p)$ and for any $i=\sup(I_p)$ we have $\lambda_{i}(B)<\lambda_{i+1}(B)-c_ni^{\frac{2}{n}-1}$: for this we define $i$ to be in the same interval as $i+1$ as soon as
\[\lambda_{i+1}(B)-\lambda_i(B)<c_n i^{\frac{2}{n}-1}\]
where $c_n$ is the constant lemma \ref{lem_separation_eigen}. As a consequence using Remark \ref{rem_mainlinvect} for each $I_p$ we have a stability result
\[\left|\sum_{i\in I_p}\lambda_i(\Om)-\lambda_i(B)\right|\lesssim (\inf I_p)^{7+\frac{8}{n}}(\lambda_1(\Om)-\lambda_1(B))\]
Now each $I_p$ is split into an (ordered) partition of intervals $(I_{p,s})_{s=1,\hdots,s_p}$ such that $\lambda_i(B)=\lambda_j(B)$ if and only if there exists $p,s$ such that $i,j\in I_{p,s}$. Let $J_{p,s}=\cup_{s'\geq s}I_{p,s'}$, the stability result on $J_{p,s}$ gives
\[\left|\sum_{i\in J_{p,s}}\lambda_i(\Om)-\lambda_i(B)\right|\lesssim \frac{(\inf I_p)^{6+\frac{10}{n}}(\lambda_1(\Om)-\lambda_1(B))}{\lambda_{\min J_{p,s}}(B)-\lambda_{\min J_{p,s}-1}(B)}.\]
We write
$$
\sum_{i\geq 1}\left(f(\lambda_i(\Om))-f(\lambda_i(B))\right) \leq \sum_{i\geq 1}f'(\lambda_i(B))\left(\lambda_i(\Om)-\lambda_i(B)\right)+\frac{1}{2}\sum_{i\geq 1}\sup_{t\geq a_n i^\frac{2}{n}}|f''(t)|(\lambda_i(\Om)-\lambda_i(B))^2.
$$
The second term is estimated using Theorem  \ref{main_sqrt}:
\[\frac{1}{2}\sum_{i\geq 1}\sup_{t\geq a_n i^\frac{2}{n}}|f''(t)|(\lambda_i(\Om)-\lambda_i(B))^2\lesssim \left(\sum_{i\geq 1}i^{4+\frac{8}{n}}\sup_{t\geq a_n i^\frac{2}{n}}|f''(t)|\right)(\lambda_1(\Om)-\lambda_1(B)).\]
The first term is split as $\sum_{p\geq 1}\sum_{i\in I_p}f'(\lambda_i(B))\left(\lambda_i(\Om)-\lambda_i(B)\right)$ and, denoting $i_s=\inf(I_{p,s})(=\inf(J_{p,s})$, each term is bounded by
\begin{align*}
\sum_{i\in I_p}f'(\lambda_i(B))\left(\lambda_i(\Om)-\lambda_i(B)\right)&=f'(\lambda_{i_1}(B))\sum_{i\in I_p}\left(\lambda_i(\Om)-\lambda_i(B)\right)\\
&+\sum_{s=2}^{s_p}\left(\left[f'(\lambda_{i_{s}}(B))-f'(\lambda_{i_{s-1}}(B))\right]\sum_{i\in J_{p,s}}\left(\lambda_i(\Om)-\lambda_i(B)\right)\right)\\
&\lesssim f'(\lambda_{i_1}(B))(\inf I_p)^{7+\frac{8}{n}}(\lambda_1(\Om)-\lambda_1(B))\\
&+\sum_{s=2}^{s_p}\sup_{t\geq a_n i_{s-1}^\frac{2}{n}}|f''(t)|(\lambda_{i_s}(B)-\lambda_{i_{s}-1}(B))\frac{(\inf I_p)^{6+\frac{10}{n}}(\lambda_1(\Om)-\lambda_1(B))}{\lambda_{i_s}(B)-\lambda_{i_{s}-1}(B)}\\
&\lesssim f'(\lambda_{i_1}(B))(\inf I_p)^{7+\frac{8}{n}}(\lambda_1(\Om)-\lambda_1(B))\\
&+\sum_{s=2}^{s_p}\sup_{t\geq a_n i_{s-1}^\frac{2}{n}}|f''(t)|(\inf I_p)^{6+\frac{10}{n}}(\lambda_1(\Om)-\lambda_1(B)).
\end{align*}
Summing this for $p\in\N^*$ we find $\sum_{i\geq 1}\left(f(\lambda_i(\Om))-f(\lambda_i(B))\right)\leq C_nB_n(f)(\lambda_1(\Om)-\lambda_1(B))$ for some $C_n>0$, \BBB thus proving the first claim. \EEE

\BBB The bound \eqref{ex1.11} follows by direct estimates of $B_n(f_t)$ for the function $f_t(\lb)=e^{-t\lb}$. Indeed,

\begin{align*}
B_n(f_t)&=\sum_{i\geq 1}\left(i^{7+\frac{8}{n}}t\exp\left(-a_n t i^\frac{2}{n}\right)+i^{6+\frac{10}{n}}t^2\exp\left(-a_n t i^\frac{2}{n}\right)\right)\\
&=t^{-3-4n}\sum_{i\geq 1}t^\frac{n}{2}\left(t^\frac{n}{2}i\right)^{7+\frac{8}{n}}\exp\left(-a_n \left(t^\frac{n}{2}i\right)^\frac{2}{n}\right)+t^{-4-2n}\sum_{i\geq 1}t^\frac{n}{2}\left(t^\frac{n}{2}i\right)^{6+\frac{10}{n}}\exp\left(-a_n \left(t^\frac{n}{2}i\right)^\frac{2}{n}\right)
\end{align*}
and both sums converge to a finite limit as $t\to 0$, and to $0$ when $t\to +\infty$. \BBB We can therefore apply the previous estimate to deduce \eqref{ex1.11}. \EEE

\end{proof}

\noindent{\bf Acknowledgements: }
The authors would like to thank G. Buttazzo, L. Briani and S. Guarino Lo Bianco for providing an early version of \cite{BBG22} and for fruitful discussions. This work was partially supported by the project ANR STOIQUES (ANR-24-CE40-2216) financed by the French Agence Nationale de la Recherche (ANR).

\bibliographystyle{plain}
%\bibliography{biblio}

\begin{thebibliography}{10}

\bibitem{AKN21}
Mark Allen, Dennis Kriventsov, and Robin Neumayer.
\newblock Sharp quantitative {F}aber-{K}rahn inequalities and the
  {A}lt-{C}affarelli-{F}riedman monotonicity formula.
\newblock {\em Ars Inven. Anal.}, pages Paper No. 1, 49, 2023.

\bibitem{AC81}
H.~W. Alt and L.~A. Caffarelli.
\newblock Existence and regularity for a minimum problem with free boundary.
\newblock {\em J. Reine Angew. Math.}, 325:105--144, 1981.

\bibitem{Be15}
A.~Berger.
\newblock The eigenvalues of the {L}aplacian with {D}irichlet boundary
  condition in {$\Bbb{R}^2$} are almost never minimized by disks.
\newblock {\em Ann. Global Anal. Geom.}, 47(3):285--304, 2015.

\bibitem{BC06}
J.~Bertrand and B.~Colbois.
\newblock Capacit{\'e} et in{\'e}galit{\'e} de {F}aber--{K}rahn dans $\mathbb
  {R}^n$.
\newblock {\em Journal of Functional Analysis}, 232(1):1--28, 2006.

\bibitem{B1866}
J.~Bourget.
\newblock M\'{e}moire sur le mouvement vibratoire des membranes circulaires.
\newblock {\em Ann. Sci. \'{E}cole Norm. Sup.}, 3:55--95, 1866.

\bibitem{BDV15}
L.~Brasco, G.~De~Philippis, and B.~Velichkov.
\newblock {F}aber--{K}rahn inequalities in sharp quantitative form.
\newblock {\em Duke Mathematical Journal}, 164(9):1777--1831, 2015.

\bibitem{BL09}
T.~Brian\c{c}on and J.~Lamboley.
\newblock Regularity of the optimal shape for the first eigenvalue of the
  laplacian with volume and inclusion constraints.
\newblock {\em Annales de l'I.H.P. Analyse non lin\'eaire}, 26(4):1149--1163,
  2009.

\bibitem{BBG22}
Luca Briani, Giuseppe Buttazzo, and Serena Guarino Lo~Bianco.
\newblock On a reverse {K}ohler-{J}obin inequality.
\newblock {\em Rev. Mat. Iberoam.}, 40(3):913--930, 2024.

\bibitem{B00}
D.~Bucur.
\newblock Uniform concentration-compactness for sobolev spaces on variable
  domains.
\newblock {\em Journal of Differential Equations}, 162(2):427--450, 2000.

\bibitem{B03}
D.~Bucur.
\newblock Regularity of optimal convex shapes.
\newblock {\em Journal of Convex Analysis}, 10(2):501--516, 2003.

\bibitem{Bu2012}
D.~Bucur.
\newblock Minimization of the k-th eigenvalue of the {D}irichlet {L}aplacian.
\newblock {\em Archive for Rational Mechanics and Analysis}, 206(3):1073--1083,
  2012.

\bibitem{BB05}
D.~Bucur and G.~Buttazzo.
\newblock {\em Variational Methods in Shape Optimization Problems}, volume~65
  of {\em Progress in Nonlinear Differential Equations and Their Applications}.
\newblock Birkh{\"a}user Basel, 2005.

\bibitem{BMPV15}
D.~Bucur, D.~Mazzoleni, A.~Pratelli, and B.~Velichkov.
\newblock Lipschitz regularity of the eigenfunctions on optimal domains.
\newblock {\em Archive for Rational Mechanics and Analysis}, 216(1):117--151,
  2015.

\bibitem{BDM93}
G.~Buttazzo and G.~Dal~Maso.
\newblock An existence result for a class of shape optimization problems.
\newblock {\em Arch. Rational Mech. Anal.}, 122(2):183--195, 1993.

\bibitem{CS05}
L.~A Caffarelli and S.~Salsa.
\newblock {\em A geometric approach to free boundary problems}, volume~68.
\newblock American Mathematical Soc., 2005.

\bibitem{CSY18}
L.~A. Caffarelli, H.~Shahgholian, and K.~Yeressian.
\newblock A minimization problem with free boundary related to a cooperative
  system.
\newblock {\em Duke Mathematical Journal}, 167(10):1825--1882, 2018.

\bibitem{CY07}
Q.-M. Cheng and H.~Yang.
\newblock Bounds on eigenvalues of {D}irichlet {L}aplacian.
\newblock {\em Mathematische Annalen}, 337(1):159--175, 2007.

\bibitem{CL12}
M.~Cicalese and G.-P. Leonardi.
\newblock A selection principle for the sharp quantitative isoperimetric
  inequality.
\newblock {\em Arch. Ration. Mech. Anal.}, 206(2):617--643, 2012.

\bibitem{DL19}
M.~Dambrine and J.~Lamboley.
\newblock Stability in shape optimization with second variation.
\newblock {\em Journal of Differential Equations}, 267(5):3009--3045, 2019.

\bibitem{D89}
E.~B. Davies.
\newblock {\em Heat kernels and spectral theory}, volume~92 of {\em Cambridge
  Tracts in Mathematics}.
\newblock Cambridge University Press, Cambridge, 1989.

\bibitem{DS11}
D.~De~Silva.
\newblock Free boundary regularity for a problem with right hand side.
\newblock {\em Interfaces and free boundaries}, 13(2):223--238, 2011.

\bibitem{DS14}
Daniela De~Silva and Ovidiu Savin.
\newblock A note on higher regularity boundary {Harnack} inequality.
\newblock {\em Discrete Contin. Dyn. Syst.}, 35(12):6155--6163, 2015.

\bibitem{DS88}
N.~Dunford and J.~T. Schwartz.
\newblock {\em Linear operators. {P}art {II}}.
\newblock Wiley Classics Library. John Wiley \& Sons, Inc., New York, 1988.
\newblock Spectral theory. Selfadjoint operators in Hilbert space, With the
  assistance of William G. Bade and Robert G. Bartle, Reprint of the 1963
  original, A Wiley-Interscience Publication.

\bibitem{HL11}
Q.~Han and F.~Lin.
\newblock {\em Elliptic partial differential equations}, volume~1 of {\em
  Courant Lecture Notes in Mathematics}.
\newblock Courant Institute of Mathematical Sciences, New York; American
  Mathematical Society, Providence, RI, second edition, 2011.

\bibitem{HP18}
A.~Henrot and M.~Pierre.
\newblock {\em Shape Variation and Optimization, A Geometrical Analysis},
  volume~28.
\newblock EMS Tracts in Mathematics, 2018.

\bibitem{KO13}
C.-Y. Kao and B.~Osting.
\newblock Minimal convex combinations of sequential {L}aplace-{D}irichlet
  eigenvalues.
\newblock {\em SIAM J. Sci. Comput.}, 35(3):B731--B750, 2013.

\bibitem{Kat95}
T.~Kato.
\newblock {\em Perturbation Theory for Linear Operators}.
\newblock Springer Berlin, Heidelberg, 1995.

\bibitem{KN77}
David Kinderlehrer and Louis Nirenberg.
\newblock Regularity in free boundary problems.
\newblock {\em Annali della Scuola Normale Superiore di Pisa-Classe di
  Scienze}, 4(2):373--391, 1977.

\bibitem{KM1}
H.~Kn\"upfer and C.~B. Muratov.
\newblock On an isoperimetric problem with a competing nonlocal term i: The
  planar case.
\newblock {\em Communications on Pure and Applied Mathematics},
  66(7):1129--1162, 2013.

\bibitem{KM2}
H.~Kn\"upfer and C.~B. Muratov.
\newblock On an isoperimetric problem with a competing nonlocal term ii: The
  general case.
\newblock {\em Communications on Pure and Applied Mathematics},
  67(12):1974--1994, 2014.

\bibitem{K78}
M.-T. Kohler-Jobin.
\newblock Une m\'{e}thode de comparaison isop\'{e}rim\'{e}trique de
  fonctionnelles de domaines de la physique math\'{e}matique. {II}. {C}as
  inhomog\`ene: une in\'{e}galit\'{e} isop\'{e}rim\'{e}trique entre la
  fr\'{e}quence fondamentale d'une membrane et l'\'{e}nergie d'\'{e}quilibre
  d'un probl\`eme de {P}oisson.
\newblock {\em Z. Angew. Math. Phys.}, 29(5):767--776, 1978.

\bibitem{KL18}
D.~Kriventsov and F.~Lin.
\newblock Regularity for shape optimizers: the nondegenerate case.
\newblock {\em Communications on Pure and Applied Mathematics},
  71(8):1535--1596, 2018.

\bibitem{KL19}
D.~Kriventsov and F.~Lin.
\newblock Regularity for shape optimizers: the degenerate case.
\newblock {\em Comm. Pure Appl. Math.}, 72(8):1678--1721, 2019.

\bibitem{LamLan06}
P.~D. Lamberti and M.~Lanza~de Cristoforis.
\newblock Critical points of the symmetric functions of the eigenvalues of the
  laplace operator and overdetermined problems.
\newblock {\em Journal of the Mathematical Society of Japan}, 58(1):231--245, 1
  2006.

\bibitem{Lau20}
A.~Laurain.
\newblock {Distributed and boundary expressions of first and second order shape
  derivatives in nonsmooth domains},.
\newblock {\em {J. Math. Pures Appl.}}, {134}:{328--368}, {2020}.

\bibitem{LY83}
P.~Li and S.-T. Yau.
\newblock On the schr{\"o}dinger equation and the eigenvalue problem.
\newblock {\em Communications in Mathematical Physics}, 88(3):309--318, 1983.

\bibitem{MTV21}
Francesco~Paolo Maiale, Giorgio Tortone, and Bozhidar Velichkov.
\newblock Epsilon-regularity for the solutions of a free boundary system.
\newblock {\em Rev. Mat. Iberoam.}, 39(5):1947--1972, 2023.

\bibitem{MP19}
D.~Mazzoleni and A.~Pratelli.
\newblock Some estimates for the higher eigenvalues of sets close to the ball.
\newblock {\em Journal of Spectral Theory}, 9(4):1385--1403, 2019.

\bibitem{MTV17}
D.~Mazzoleni, S.~Terracini, and B.~Velichkov.
\newblock Regularity of the optimal sets for some spectral functionals.
\newblock {\em Geom. Funct. Anal.}, 27(2):373--426, 2017.

\bibitem{S07}
Ovidiu Savin.
\newblock Small perturbation solutions for elliptic equations.
\newblock {\em Communications in Partial Differential Equations},
  32(4):557--578, 2007.

\bibitem{S29}
C.~L. Siegel.
\newblock \"{U}ber einige {A}nwendungen diophantischer {A}pproximationen
  [reprint of {A}bhandlungen der {P}reu\ss ischen {A}kademie der
  {W}issenschaften. {P}hysikalisch-mathematische {K}lasse 1929, {N}r. 1].
\newblock In {\em On some applications of {D}iophantine approximations},
  volume~2 of {\em Quad./Monogr.}, pages 81--138. Ed. Norm., Pisa, 2014.

\bibitem{Tal76}
G.~Talenti.
\newblock Elliptic equations and rearrangements.
\newblock {\em Annali della Scuola Normale Superiore di Pisa - Classe di
  Scienze}, 3(4):697--718, 1976.

\bibitem{BBP21}
M.~van~den Berg, G.~Buttazzo, and A.~Pratelli.
\newblock On relations between principal eigenvalue and torsional rigidity.
\newblock {\em Commun. Contemp. Math.}, 23(8):Paper No. 2050093, 28, 2021.

\bibitem{V19}
Bozhidar Velichkov.
\newblock {\em Regularity of the one-phase free boundaries}.
\newblock Springer Nature, 2023.

\bibitem{W44}
G.~N. Watson.
\newblock {\em A {T}reatise on the {T}heory of {B}essel {F}unctions}.
\newblock Cambridge University Press, Cambridge, England; The Macmillan
  Company, New York, 1944.

\end{thebibliography}

\end{document}